\documentclass[reqno,10pt]{amsart}

\usepackage{amscd}
\usepackage{amsmath}
\usepackage[new]{old-arrows}
\usepackage{amssymb}
\usepackage[american]{babel}
\usepackage{bbm}
\usepackage{bookmark}
\usepackage{cmap} 
\usepackage{dsfont}
\usepackage{enumerate}
\usepackage{epigraph}
\usepackage[mathscr]{euscript}
\usepackage[myheadings]{fullpage}
\usepackage{graphicx}
\usepackage[geometry]{ifsym}
\usepackage{mathabx}
\usepackage{accents}
\usepackage{mathrsfs}
\usepackage{mathtools}
\usepackage{enumitem}
\usepackage{refcount}
\usepackage{cleveref}
\usepackage{setspace}
\usepackage{stmaryrd}
\usepackage{verbatim}
\usepackage[all,2cell]{xy} \UseTwocells
\usepackage{xr}
\usepackage[multiple]{footmisc}
\usepackage{hyperref}

\usepackage{tikz}
\usetikzlibrary{decorations.pathmorphing}
\usetikzlibrary{cd}

\numberwithin{equation}{subsection}

\newtheorem{thm}{Theorem}[subsubsection]
\newtheorem*{thm*}{Theorem}

\newtheorem{cor}[thm]{Corollary}
\newtheorem*{cor*}{Corollary}
\newtheorem{lem}[thm]{Lemma}

\newtheorem{prop}[thm]{Proposition}
\newtheorem{prop-const}[thm]{Proposition-Construction}

\newtheorem*{conjecture*}{Conjecture}

\newtheorem*{princ*}{Principle}

\theoremstyle{remark}
\newtheorem{rem}[thm]{Remark}
\newtheorem{example}[thm]{Example}

\newtheorem{defin}[thm]{Definition}

\newcounter{steps}[thm]
\newtheorem{step}[steps]{Step}

\newcommand{\into}{\hookrightarrow}
\newcommand{\onto}{\twoheadrightarrow}

\newcommand{\bA}{{\mathbb A}}
\newcommand{\bB}{{\mathbb B}}
\newcommand{\bC}{{\mathbb C}}
\newcommand{\bD}{{\mathbb D}}

\newcommand{\bF}{{\mathbb F}}
\newcommand{\bG}{{\mathbb G}}

\newcommand{\bN}{{\mathbb N}}
\newcommand{\bO}{{\mathbb O}}
\newcommand{\bP}{{\mathbb P}}

\newcommand{\bR}{{\mathbb R}}
\newcommand{\bS}{{\mathbb S}}

\newcommand{\bZ}{{\mathbb Z}}

\newcommand{\cB}{{\mathcal B}}

\newcommand{\cF}{{\mathcal F}}
\newcommand{\cG}{{\mathcal G}}

\newcommand{\cN}{{\mathcal N}}

\newcommand{\cS}{{\mathcal S}}

\newcommand{\cU}{{\mathcal U}}

\newcommand{\cW}{{\mathcal W}}

\newcommand{\sA}{{\EuScript A}}

\newcommand{\sC}{{\EuScript C}}
\newcommand{\sD}{{\EuScript D}}

\newcommand{\sG}{{\EuScript G}}
\newcommand{\sH}{{\EuScript H}}

\newcommand{\sL}{{\EuScript L}}
\newcommand{\sM}{{\EuScript M}}
\newcommand{\sN}{{\EuScript N}}

\newcommand{\sP}{{\EuScript P}}

\newcommand{\sR}{{\EuScript R}}

\newcommand{\sW}{{\EuScript W}}
\newcommand{\sX}{{\EuScript X}}
\newcommand{\sY}{{\EuScript Y}}

\newcommand{\fX}{{\mathfrak X}}

\newcommand{\fb}{{\mathfrak b}}

\newcommand{\fg}{{\mathfrak g}}
\newcommand{\fh}{{\mathfrak h}}

\newcommand{\fk}{{\mathfrak k}}

\newcommand{\fm}{{\mathfrak m}}
\newcommand{\fn}{{\mathfrak n}}

\newcommand{\fp}{{\mathfrak p}}

\newcommand{\ft}{{\mathfrak t}}
\newcommand{\fu}{{\mathfrak u}}

\newcommand{\fz}{{\mathfrak z}}

\newcommand{\on}{\operatorname}

\newcommand{\Spec}{\on{Spec}}

\newcommand{\ind}{\on{ind}}

\renewcommand{\lim}{\on{lim}}

\renewcommand{\subset}{\subseteq}

\newcommand{\dR}{\on{dR}}

\makeatletter
\newcommand{\biggg}{\bBigg@{4}}
\newcommand{\Biggg}{\bBigg@{5}}
\makeatother

\newcommand{\sss}{\subsubsection{}}

\begin{document}

\frenchspacing

\setlength{\epigraphwidth}{0.4\textwidth}
\renewcommand{\epigraphsize}{\footnotesize}

%\AtBeginDocument{\addtocontents{toc}{\protect\setlength{\parskip}{0pt}}}

\title{Singular support for $G$-categories}

\author{Gurbir Dhillon and Joakim F\ae rgeman}

\address{University of California, Los Angeles, Department of Mathematics, 520 Portola Plaza, Los Angeles, CA 90095}

\email{gsd@math.ucla.edu}

\address{Yale University, 
Department of Mathematics, 
219 Prospect St, 
New Haven, CT 06511}

\email{joakim.faergeman@yale.edu}

\maketitle

\begin{abstract}
For a reductive group $G$, we introduce a notion of singular support for cocomplete dualizable DG-categories equipped with a strong $G$-action. This is done by considering the singular support of the sheaves of matrix coefficients arising from the action. We focus particularly on dualizable $G$-categories whose singular support lies in the nilpotent cone of $\fg^*$ and refer to these as nilpotent $G$-categories. For such categories, we give a characterization of the singular support in terms of the vanishing of its generalized Whittaker models.

We study parabolic induction and restriction functors of nilpotent $G$-categories and show that they interact with singular support in a desired way.

We prove that if an orbit is maximal in the singular support of a nilpotent $G$-category $\sC$, the Hochschild homology of the generalized Whittaker model of $\sC$ coincides with the microstalk of the character sheaf of $\sC$ at that orbit. This should be considered a categorified analogue of the result of \cite{moeglin1987modeles} that the dimension of the generalized Whittaker model of a smooth admissible representation of a reductive group over a non-Archimedean local field of characteristic zero coincides with the Fourier coefficient in the wave-front set of that orbit.

As a consequence, we give another proof of a theorem of Bezrukavnikov-Losev, classifying finite-dimensional modules for $\sW$-algebras with fixed regular central character. More precisely, we realize the (rationalized) Grothendieck group of this category as a certain subrepresentation of the Springer representation. Along the way, we show that the Springer action of the Weyl group on the twisted Grothendieck--Springer sheaves is the categorical trace of the wall crossing functors, extending an observation of Zhu for integral central characters. 

\end{abstract}

\tableofcontents

\section{Introduction}

The basic goal of the present paper is to introduce some tools for the analysis of categorical representations via their matrix coefficients, and in particular develop a theory of singular support, i.e., wave front sets, in this setting. Since as far as we know, this may be the first paper devoted to categorical matrix coefficients, we begin by explaining why one would want such a theory, and what one might expect from it; the reader wishing to skip such a discussion may proceed directly to Section \ref{s:mc}.

\subsection{Some context} To motivate the present paper, it is helpful to recall two standard facts of life in representation theory.

\sss{} On the one hand, given an action of a group $G$ on a vector space $V$, one obtains a collection of special (generalized) functions on $G$, namely the matrix coefficients of $V$. Many basic properties of $V$ may then be read off from {\em geometric} properties of these special functions, for example their microsupport or asymptotic growth at infinity in $G$. 

This theory is particularly rich and successful in the study of finite groups of Lie type and reductive groups over local fields, and as such plays a basic role in the local Langlands correspondence. As we will recall in more detail below, a basic invariant one attaches in these cases to an irreducible representation $V$ is the union of the closures of some nilpotent orbits, the {\em wave front set} of $V$, which may be read off from a single matrix coefficient, namely the character of $V$. 

\sss{} On the other hand, a similarly basic lesson from geometric representation theory is that  answers to concrete questions, e.g. character formulae for representations of groups and Lie algebras, are often encoded in actions of groups on categories, or the induced actions of Hecke categories on categories of equivariant objects. 

Relatedly, it was realized  by Frenkel--Gaitsgory, building on earlier work of Beilinson--Drinfeld, that the notion of a categorical representation of a loop group is necessary to even formulate the local geometric Langlands conjecture; we recall the latter implies many concrete results about characters for affine Lie algebras and related vertex algebras.  

\sss{} By combining the above two lessons, it is therefore natural to anticipate a useful theory of matrix coefficients for categorical representations. That is, given an action of a group $G$ on a category $\sC$, one should obtain a collection of special sheaves on $G$,   thought of as the matrix coefficients of $\sC$. Moreover, geometric invariants of these sheaves should again encode properties of the representation $\sC$, and such a theory should play a basic role in the local geometric Langlands correspondence. 

\sss{} In the present paper, we supply some of the basic ingredients for such a theory. We focus on the case of a finite dimensional reductive group $G$ over a field $k$ of characteristic zero, which should be regarded as the analog of a finite group of Lie type. We will return to the theory for loop groups, i.e., the analogues of $p$-adic groups, elsewhere. 

\sss{} We now provide a brief non-technical overview of the results of this paper.  

\begin{itemize}
    \item First,  given a categorical representation $\sC$ of $G$, we call $\sC$ {\em nilpotent} if all its suitably defined matrix coefficients have  microsupport in 
    $$G \times \mathcal{N} \hookrightarrow G \times \fg^* \simeq T^*G, $$
    where $\mathcal{N} \hookrightarrow \fg^*$ denotes the nilpotent cone. While for the Lie theoretic applications we have in mind it is essential to work in the de Rham setting, i.e., with D-modules, our first main Theorem \ref{t:nilp} shows that the abstract 2-category of nilpotent $G$-categories  is independent of our chosen sheaf theory in a strong sense - it is equivalent to modules over all the twisted Hecke categories of $G$.  
    
    \item Second, for a nilpotent $G$-category $\sC$, we introduce a definition for its singular support, which is a closed union of nilpotent orbits. We show that it interacts with the operations of taking Whittaker invariants and parabolic induction and restriction in desirable ways. While these results concerning singular support closely parallel similar assertions for finite groups of Lie type, our definition does not; this suggests we have nevertheless identified the correct formulation of singular support in the present setting. 
    
    \item Finally, as an application of the above techniques, we give a new proof of a fundamental theorem of Losev-Ostrik and Bezrukavnikov-Losev classifying finite-dimensional representations of finite $\sW$-algebras. Namely, we first characterize the category of finite-dimensional representations of a $\sW$-algebra in terms of categorical singular support, and then apply categorical trace methods.     
    In particular,  the techniques introduced in this paper have bearing on  concrete representation theoretic problems.\footnote{In fact, a principal motivation for us to revisit the classification of Bezrukavnikov--Losev--Ostrik was  ongoing work with Arakawa on the simple highest weight characters for finite and affine $\sW$-algebras.}   \iffalse  our approach provides strong functoriality between modules for $\sW$-algebras and Springer theory (see §\ref{S:CLASSIFICATIONOFFDIMRAAAAAPS} below for precise statements).

We expect that categorical trace methods are useful for other representation-theoretic questions about $\sW$-algebras than the one addressed above. In fact, in joint work-in-progress with Arakawa, we plan to apply similar methods to obtain character formulae for highest weight representations of affine $\sW$-algebras.
\fi

\end{itemize}

Before stating our results precisely, particularly those alluded to in the second bullet point above, it will be useful to recall some aspects of the usual theory of matrix coefficients.

\subsection{What do matrix coefficients control?}\label{S:WHATDOMCCONTROL}

Suppose for concreteness that $G$ is defined over a non-archimedean local field $\bF$. Roughly, interesting properties about a representation $V$ are encoded in the support of its matrix coefficients. Here, the word support has dual meanings:

\sss{} First, we may consider the support in $G$, which we think of as position space. A basic property here is that supercuspidality of $V$ is equivalent to its matrix coefficients having compact support modulo the center of $G$.

\sss{} Second, we may consider the microsupport in $\fg^*$, which we think of as momentum space, as follows. By restricting the matrix coefficients to a neighbourhood of the identity  $1\in G$, taking the logarithm and applying the Fourier transform, one obtains generalized functions on $\fg^*$. The support of these generalized functions encodes a surprising amount of information about $V$. 

    The main property that is relevant for us is the connection between the microsupport of matrix coefficients in the above sense and the Whittaker models of $V$. Namely, to $V$, we may associate nilpotent invariants of two kinds:\footnote{We remark that the idea of assigning nilpotent orbits to irreducible representations occured earlier in representation theory as well: Kirillov's orbit method, the Springer correspondence, etc.}

    \begin{itemize}
\item For any nilpotent orbit, we may consider its generalized Whittaker model:
    \[
    \sW_{\bO}(V):=(\sW_{\bO}\otimes V)_G.
    \]

    \noindent Here, $\sW_{\bO}$ is the generalized Whittaker, i.e. Gelfand--Graev, representation associated to $\bO$, cf. \cite[§5.1]{gomez2017generalized}.

    \item On the other hand, let $\chi_{V}$ be the character of $V$, which is a generalized function on $G$. As above, if we restrict to a neighbourhood of $1\in G$ and apply the logarithm, we obtain a generalized function on $\fg$, which turns out to be the Fourier transform of a linear combination of $G$-invariant measures on coadjoint nilpotent orbits on $\fg^*$. Let $c_{\bO,V}$ be the multiplicity of the $G$-measure of $\bO$ in the Fourier transform of $\chi_{V}$.
    \end{itemize}

    \noindent Denote by $\on{WS}(V)$ the set of maximal nilpotent orbits for which $\sW_{\bO}(V)$ is nonzero. Similarly, denote by $\on{WF}(V)$ the set of maximal nilpotent orbits for which $c_{\bO,V}$ is nonzero. In \cite{moeglin1987modeles},\footnote{In \emph{loc.cit}, they work under the assumption that the characteristic of the residue field of $\bF$ is odd. This assumption was removed in \cite{varma2014result}.} it is proven that:
\begin{enumerate}

    \item $\on{WS}(V)=\on{WF}(V)$.

    \item For $\bO\in \on{WS}(V)$, we have an equality:
    \[
    \on{dim}\sW_{\bO}(V)=c_{\bO,V}.
    \]

\end{enumerate}
That is, the maximal orbits in the microsupport of $\chi_V$ agree with the maximal orbits for which the Whittaker models of $V$ are non-vanishing, and the dimensions of the latter are encoded in the coefficients of the local character expansion of the former. See also \cite{gomez2017generalized} and \cite{lusztig1992unipotent} for similar results when $\bF$ is archimedean and finite, respectively.

\begin{rem}As another representative, if less immediately relevant, example, it is shown in \cite{bezrukavnikov2016depth} that the depth $r$ Bernstein projector admits a beautiful presentation in terms of microsupport in the above sense.   
\end{rem}

\sss{} For loop groups, all of the above points have analogues for categorical matrix coefficients, as will be explained elsewhere. By contrast, for reductive groups, every categorical representation is principal series \cite{ben2020highest}, so we will be primarily interested in momentum space, i.e., with the microsupport of matrix coefficients and their relation with Whittaker models. With this motivating discussion in mind, let us now turn to a precise formulation of our results.

\subsection{Sheaves of matrix coefficients}\label{s:mc} Let $\sC$ be a dualizable DG-category equippped with a strong action of $G$. We remind that this means we have an action
\[
D(G)\curvearrowright \sC,
\]

\noindent where $D(G)$ is the category of D-modules on $G$ equipped with its monoidal structure via convolution. From the action map
\[
D(G)\otimes \sC\to \sC,
\]

\noindent and the self-duality of $D(G)$, we get a functor
\[
\on{MC}_{\sC}: \on{End}_{\on{DGCat}_{\on{cont}}}(\sC)\simeq \sC\otimes \sC^{\vee}\to D(G),
\]

\noindent which associates to a (continuous) endomorphism $\phi$ of $\sC$ its corresponding sheaf of matrix coefficients on $G$. Concretely, given $g\in G$, the fiber of $\on{MC}_{\sC}(\phi)$ at $g$ is $\on{Tr}(\sC,g_*\circ\phi)$, the categorical trace of the endomorphism $g_*\circ\phi$ of $\sC$, where $g_*:\sC\to \sC$ is the endofunctor of $\sC$ defined by acting by the delta sheaf corresponding to $g$.

\subsubsection{} Given a closed, conical Ad-stable $Z\subset \fg^*$, we may consider the full subcategory
\[
D_Z(G)\subset D(G)
\]

\noindent of D-modules on $G$ whose singular support is contained in $G\times Z\subset G\times \fg^*\simeq T^*G$. This category is not an ideal inside $D(G)$ in general.

\subsubsection{} We say that a dualizable $G$-category $\sC$ has singular support contained in $Z$ if $\on{MC}_{\sC}$ factors as
\[
\on{MC}_{\sC}: \sC\otimes \sC^{\vee}\to D_Z(G)\into D(G).
\]

\noindent In general, we let $\on{SS}(\sC)\subset \fg^*$ be the smallest closed, conical Ad-stable subset such that $\sC$ has singular support in $\on{SS}(\sC)$.

\begin{rem}
Let us emphasize two basic differences between the above definition of singular support and its definition for finite groups of Lie type and $p$-adic groups. 

First, unlike the $p$-adic setting, we have immediate access to a notion of singular support for $G$-categories. That is, the fact that we may directly talk about the singular support of objects of $D(G)$ bypasses the need to invoke a Fourier transform.

 Second, in the $p$-adic setting, the wavefront set of a representation $V$ is defined purely in terms of the character of $V$, i.e., the image of the identity endomorphism under the matrix coefficient map. While one could  make such a definition for categorical representations, it would be pathological - if one tensors any categorical representation $\sC$ with a category $\sM$ with vanishing Hochschild homology, viewed as  a trivial representation of $G$, the resulting category $\sC \otimes \sM$'s character sheaf vanishes.\footnote{Heuristically, one should regard categorical representations as the analogue of {\em complexes} of modules for a finite group of Lie type, for which one would run into similar problems defining wave front sets via characters; this analogy can be made precise over finite fields via the categorical trace of Frobenius.} For this reason,  we must consider not only the character, but all of the matrix coefficients, when computing singular support.

 \iffalse
 For this reason, one might prematurely conclude  that a good theory of singular support for categorical representations could not  exist. However, we show that if one considers instead \emph{all} sheaves of matrix coefficients, one does get a well behaved notion of singular support for categorical representations. \fi 
 
 \iffalse 
 If one tried to copy this definition of singular support for categorical representations, it is easy to see that one would not get a good theory (for example, one would have non-zero categorical representations with empty singular support). As such, one might initially think that a theory of singular support for categorial representations can not exist.

Ultimately, this reflects the inherently derived nature of our setup.\fi 
\end{rem}

\subsection{Examples}\label{s:examp}
We fix an Ad-invariant isomorphism $\fg^*\simeq \fg$. Denote by $\cN\subset \fg^*$ the locus of nilpotent elements under this isomorphism.

Below are some natural examples of $G$-categories and their singular support.
\begin{itemize}

     \item Consider $D(G)$ as a module category over itself. Then $\on{SS}(D(G))=\fg^*$.

     \item At the other extreme, let $\on{Cst}\subset D(G)$ be the subcategory generated by the constant sheaf under colimits. Then $\on{Cst}$ is naturally a $G$-category with singular support equal to $0$. 
     
    \item Let $\lambda\in \ft^*$ be a weight. Consider the category $\fg\on{-mod}_{\lambda}$ of $U(\fg)$-modules with generalized central character $\lambda$. Then $\on{SS}(\fg\on{-mod}_{\lambda})=\cN$.

    \item (Corollary \ref{c:boundonann}): Let $\fg\on{-mod}_{0,Z}\subset \fg\on{-mod}_{0}$ be the subcategory consisting of modules $M$ such that the associated variety of $U(\fg)_0/\on{Ann}_{U(\fg)_0}(M)$ is contained in $Z$. Then $\fg\on{-mod}_{0,Z}$ is $G$-stable, and we have $\on{SS}(\fg\on{-mod}_{0,Z})=Z^{\on{sp}}$, where $Z^{\on{sp}}$ is the union of the closures of the special nilpotent orbits (in the sense of §\ref{s:sospecial} below) contained in $Z$.

    \item (Proposition \ref{p:geomexamp}): Let $X$ be a $G$-variety and $\mu: T^*X\to \fg^*$ the moment map. Let $Z\subset \fg^*$ be the smallest closed Ad-invariant subset such that $\mu$ factors through $Z$. If $Z\subset \cN$, then $\on{SS}(D(X))=Z$. We expect a similar equality to hold when $Z$ is not required to lie in $\cN$ but were unable to prove it.\footnote{The inclusion $Z\subset \on{SS}(D(X))$ is trivially true.}\footnote{To elaborate, denote by $a:G\times X\to X$ the action map. By definition, $\on{SS}(D(X))\subset Z$ if and only if 
\begin{equation}\label{eq:Kunneth}
a^!(\cF)\in D_Z(G)\otimes D(X)
\end{equation}

\noindent for all $\cF\in D(X)$. Note that since the action map is smooth, we have that
\[
a^!(\cF)\in D_{G\times Z\times T^*X}(G\times X).
\]

\noindent However, since the inclusion
\[
D_Z(G)\otimes D(X)\subset D_{G\times Z\times T^*X}(G\times X)
\]

\noindent is rarely an equivalence, it is not clear how to establish (\ref{eq:Kunneth}). We prove (\ref{eq:Kunneth}) for $Z\subset \cN$ using the full strength of the results described in §\ref{s:whitredflag}. The proofs of these results rely heavily on representation-theoretic methods, and it would be interesting to find a purely geometric proof of (\ref{eq:Kunneth}).}

\end{itemize}

\subsection{Nilpotent categories}\label{s:nilpcats}
We say that a dualizable $G$-category is \emph{nilpotent} if its singular support is contained in the nilpotent cone $\cN$. We will mainly be concerned with nilpotent $G$-categories in this paper.

\subsubsection{} Let $B\subset G$ be a Borel subgroup, and let $N=[B,B], T=B/N$. Denote by $X^{\bullet}(T)$ the character lattice of $T$. For each $k$-point $\chi\in\ft^*/X^{\bullet}(T)$, we get a corresponding character sheaf on $T$, which we similarly denote by $\chi$.

\subsubsection{} For a $G$-category $\sC$, we let 
\[
\sC^{(B,\chi)\on{-mon}}\subset \sC^N
\]

\noindent the full subcategory of $\sC^N$ generated under colimits by the essential image of the forgetful functor
\[
\sC^{B,\chi}\to\sC^N.
\]

\subsubsection{} The following characterization of nilpotent $G$-categories is crucially used throughout the paper.
\begin{thm}[Theorem \ref{t:nilp}]\label{t:criterionnilp}
Let $\sC$ be a dualizable $G$-category. Then $\sC$ is nilpotent if and only if
\[
\sC^N=\underset{\chi\in\ft^*/X^{\bullet}(k)}{\oplus}\sC^{(B,\chi)\on{-mon}}.
\]

\end{thm}

\subsubsection{A canonical filtration}\label{s:canfilt}
Let $\sC$ be a nilpotent $G$-category. In light of the above theorem, we assume for simplicity that $\sC^N=\sC^{(B,\chi)\on{-mon}}$ for some $\chi\in \ft^*/X^{\bullet}(k)$.

For any closed conical Ad-stable $Z\subset \cN$, we define $\sC_Z$ to be the image of the fully faithful embedding
\[
D(G/B,\chi)\underset{\sH_{G,\chi}}{\otimes}\sH_{G,\chi,Z}\underset{\sH_{G,\chi}}{\otimes} \sC^{B,\chi}\to \sC
\]

\noindent induced by the action of $D(G)$ on $\sC$. By construction, $\sC_Z\subset \sC$ is the largest $G$-stable subcategory of $\sC$ satisfying that $\on{SS}(\sC_Z)\subset Z$. Note that $\sC_{\cN}=\sC$.
\begin{thm}[Proposition \ref{l:stable}]\label{t:loc}
The category $\sC_Z$ is dualizable. Moreover, the inclusion
\[
\sC_Z\into \sC
\]

\noindent admits a left adjoint.
\end{thm}

\noindent As such, we obtain a filtration $\lbrace \sC_{\overline{\bO}}\rbrace$ of $\sC$ indexed by nilpotent orbits, where $\sC_{\overline{\bO}}$ is a dualizable $G$-category with $\on{SS}(\sC_{\overline{\bO}})\subset \overline{\bO}$.

\subsection{Parabolic induction and restriction}
We next explain how parabolic induction and restriction of nilpotent $G$-categories interact with singular support.

\subsubsection{} Let $P\subset G$ be a parabolic subgroup with Levi quotient $M$. Let $\fp,\fm$ be the corresponding Lie algebras, and let $U_P$ denote the unipotent radical of $P$.

\subsubsection{}\label{s:componmbse}
For a $G$-category $\sC$, we let
\[
\on{Res}^G_M(\sC):=\sC^{U_P}
\]

\noindent denote the $U_P$-invariants of $\sC$ considered as an $M$-category. Conversely, for an $M$-category $\sD$, we let 
\[
\on{Ind}_M^G(\sD):=D(G)\underset{D(P)}{\otimes} \sD
\]

\noindent considered as a $G$-category. The functors $(\on{Ind}_M^G,\on{Res}^G_M)$ define an adjunction of $2$-categories:
\[\begin{tikzcd}
	{D(M)\textrm{\textbf{-mod}}} && {D(G)\textrm{\textbf{-mod}}.}
	\arrow["{\mathrm{Ind}_M^G}", shift left, from=1-1, to=1-3]
	\arrow["{\mathrm{Res}_M^G}", shift left, from=1-3, to=1-1]
\end{tikzcd}\]

\subsubsection{} Consider the correspondence
\[\begin{tikzcd}
	& {\mathfrak{p}} \\
	{\mathfrak{g}} && {\mathfrak{m}.}
	\arrow["{p}"', from=1-2, to=2-1]
	\arrow["{q}", from=1-2, to=2-3]
\end{tikzcd}\]

\noindent For closed conical Ad-invariant subsets $Z\subset \fg$ and $Z_M\subset \fm$, we let
\[
\on{Res}_M^G(Z):=q(p^{-1}(Z))\subset \fm,\;\; \on{Ind}_M^G(Z_M):=\on{Ad}_G(p(q^{-1}(Z_M)))\subset \fg.
\]
\begin{thm}[Theorem \ref{t:pres}, Theorem \ref{t:pind}]
Let $\sC,\sC_M$ be nilpotent $G$ and $M$-categories, respectively. Then
\[
\on{SS}(\on{Res}_M^G(\sC))=\on{Res}_M^G(\on{SS}(\sC)),\;\; \on{SS}(\on{Ind}_M^G(\sC_M))=\on{Ind}_M^G(\on{SS}(\sC_M)).
\]
\end{thm}

\subsection{Microlocal characterization of generalized Whittaker models}\label{S:1.5}
The main result of this paper, Theorem \ref{t:microchar} below, gives a characterization of the singular support of nilpotent $G$-categories in terms of the vanishing of its generalized Whittaker models.

\subsubsection{}\label{s:defgenwhit} Let $e\in \cN$ be a nilpotent element and $(e,h,f)$ an $\mathfrak{sl}_2$-triple extending $e$. We get a decomposition $\fg=\underset{i\in \bZ}{\oplus} \fg(i)$, where
\[
\fg(i):=\lbrace x\in \fg\;\vert\; [h,x]=i\cdot x\rbrace.
\]

\noindent Fix a Lagrangian $\ell\subset \fg(-1)$ with respect to the natural symplectic structure on $\fg(-1)$ (see §\ref{s:genwhit}). Let $U^{-}=U^{-}_{\ell,e}$ be the unipotent subgroup of $G$ such that
\[
\on{Lie}(U^{-})=\ell\oplus\underset{i\leq -2}{\bigoplus} \fg(i).
\]

\noindent By construction, $U^{-}$ comes equipped with a natural non-degenerate character
\[
\psi_e: U^{-}\to \bA^1.
\]

\subsubsection{} For a $G$-category $\sC$, we let
\[
\on{Whit}_{e}(\sC):=\sC^{U^{-},\psi_e}\subset \sC
\]

\noindent denote the subcategory of $(U^{-},\psi_e)$-invariant objects and refer to it as the \emph{generalized Whittaker model of} $\sC$.

\subsubsection{}\label{s:indeporbit} It is not difficult to see that for any other $e'\in \bO$ (along with an $\mathfrak{sl}_2$-triple and a Lagrangian as above), we have an equivalence
\[
\on{Whit}_e(\sC)\simeq \on{Whit}_{e'}(\sC).
\]

\noindent Thus, for a nilpotent orbit $\bO$, we write $\on{Whit}_{\bO}(\sC):=\on{Whit}_e(\sC)$ for some fixed $e\in \bO$.

\subsubsection{} Our main result is the following:
\begin{thm}[Theorem \ref{t:main}]\label{t:microchar}
Let $\sC$ be a nilpotent $G$-category and $\bO\subset \cN$ a nilpotent orbit. If $\bO\cap\on{SS}(\sC)=\emptyset$, then
\[
\on{Whit}_{\bO}(\sC)=0.
\]

\noindent Moreover, if $\bO$ is maximal with respect to the property that $\bO\subset \on{SS}(\sC)$, then
\[
\on{Whit}_{\bO}(\sC)\neq 0.
\]
\end{thm}

\begin{rem}
In fact, one implication of Theorem \ref{t:microchar} is still true without the assumption that $\sC$ be nilpotent. Namely, if $\sC$ is a dualizable $G$-category with $\bO\cap\on{SS}(\sC)=\emptyset$, then it is still true that
\[
\on{Whit}_{\bO}(\sC)=0.
\]

\noindent We do not know if the converse holds without the nilpotency assumption.\footnote{A similar situation occurs in \cite{faergeman2022non}. Here, it is shown that automorphic D-modules with irregular singular support have vanishing Whittaker coefficients. The converse holds when the D-modules have nilpotent singular support, although it is expected to hold without the nilpotency assumption.}
\end{rem}

\subsection{Whittaker reduction of D-modules on the flag variety}\label{s:whitredflag}
The main input needed in the proof of Theorem \ref{t:microchar} is a micro-local criterion for when the Whittaker reduction of a (twisted) D-module on the flag variety vanishes.

\subsubsection{} Fix a character sheaf $\chi\in D(T)$ corresponding to a $k$-point in $\ft^*/X^{\bullet}(T)$. We consider the category 
\[
D(G/B,\chi)
\]

\noindent of $(B,\chi)$-equivariant D-modules of $G$. 

\subsubsection{} For $\cF\in D(G/B,\chi)$, the singular support of $\cF$ naturally lives on $\widetilde{\cN}=T^*(G/B)$. Let 
\[
\mu:\widetilde{\cN}\to \cN
\]

\noindent be the Springer resolution. For a closed conical Ad-stable $Z\subset \cN$, we let
\[
D_Z(G/B,\chi)\subset D(G/B,\chi)
\]

\noindent be the subcategory of D-modules whose singular support is contained in $\mu^{-1}(Z)$.

\subsubsection{} Let
\[
\sH_{G,\chi}:= D(B,\chi\backslash G/B,\chi)
\]

\noindent be the Hecke category associated to $\chi$. We write
\[
\sH_{G,\chi,Z}\subset \sH_{G,\chi}
\]

\noindent for the subcategory of D-modules mapping to $D_Z(G/B,\chi)$ under the forgetful functor
\[
\sH_{G,\chi}\to D(G/B,\chi).
\]

\begin{rem}\label{r:compgen}
Using the notation of §\ref{s:canfilt}, for $\sC=D(G/B,\chi)$, we have:
\[
\sC_Z=D(G/B,\chi)\underset{\sH_{G,\chi}}{\otimes} \sH_{G,\chi,Z}.
\]
\end{rem}

\subsubsection{} For a nilpotent element $e\in\cN$, we may consider the Whittaker averaging functor
\[
\on{Av}_*^{\psi_e}: D(G/B,\chi)\to D(G)\to D(G/U^{-},\psi_e).
\]

\noindent Here, the first functor is the forgetful functor, and the second functor is the right adjoint to the forgetful functor:
\[
D(G/U^-,\psi_e)\to D(G).
\]

\begin{thm}[Theorem \ref{t:van}]\label{t:whitred}
We have an identification
\[
D(G/B,\chi) \underset{\sH_{G,\chi}}{\otimes} \sH_{G,\chi,Z}=\underset{e\in\bO\not\subset Z}{\bigcap}\on{Ker}(\on{Av}_*^{\psi_e})
\]

\noindent as subcategories of $D(G/B,\chi)$. That is, for any $\cF\in D(G/B,\chi) \underset{\sH_{G,\chi}}{\otimes} \sH_{G,\chi,Z}$ and $e\notin Z$, one has $\on{Av}_*^{\psi_e}(\cF)=0$. Conversely, if for all $e\notin Z$, we have $\on{Av}_*^{\psi_e}(\cF)=0$, then $\cF\in D(G/B,\chi) \underset{\sH_{G,\chi}}{\otimes} \sH_{G,\chi,Z}$.
\end{thm}

\begin{rem}
We expect that the tautological inclusion $D(G/B,\chi) \underset{\sH_{G,\chi}}{\otimes} \sH_{G,\chi,Z}\into D_Z(G/B,\chi)$ is an equivalence. We show that any holonomic object of the right-hand side in fact belongs to the left-hand side, see Theorem \ref{t:van1}.

In \cite{faergeman2022non} a special case of Theorem \ref{t:whitred} was proven. Namely, it was shown that for any holonomic $\cF\in D(G/B,\chi)$ whose singular support is contained in the irregular locus, one has $\on{Av}_*^{\psi_e}(\cF)=0$ for all regular $e\in \cN$.
\end{rem}

\begin{rem}
If instead of Whittaker averaging on the right, we average on the left:
\[
D(G/B,\chi)\to D(U^{-},\psi_e,\backslash G/B,\chi),
\]

\noindent one may similarly ask if Theorem \ref{t:whitred} holds in this case. One can prove that the first statement of Theorem \ref{t:whitred} remains true (i.e., averaging for an orbit not contained in the singular support kills the sheaf). However, for the converse to hold, one would need to allow field-valued points $e\in\bO(K)$, where $K/k$ is a possibly transcendental field extension, in the definition of Whittaker averaging.
\end{rem}

\subsubsection{} The categories in Theorem \ref{t:whitred} admit another description. We have a Beilinson-Bernstein equivalence $D(G/B,\chi)\simeq \fg\on{-mod}_{\lambda}$ for a suitable $\lambda\in \ft^*$, where the latter is the category $U(\fg)$-modules with central character $\lambda$. Under this equivalence, a module $M\in \fg\on{-mod}_{\lambda}$ belongs to $D(G/B,\chi)\underset{\sH_{G,\chi}}{\otimes} \sH_{G,\chi,Z}$ if and only the associated variety of $U(\fg)/\on{Ann}_{U(\fg)}(M)$ is contained in $Z$.

\subsection{Realization of Hochschild homology of Whittaker model as a microstalk}\label{S:pseudomicro}

\subsubsection{} For each $e\in\bO \subset\cN$, we associate a Whittaker functional
\[
\on{coeff}_{\bO}: D(G\overset{\on{ad}}{/}G)\to\on{Vect.}
\]

\noindent Namely, let $\fp_e: U^{-}\overset{\on{ad}}{/}U^{-}\to G\overset{\on{ad}}{/}G$ be the canonical map and note that $\psi_e: U^{-}\to \bA^1$ descends to a map $ U^{-}\overset{\on{ad}}{/}U^{-}\to \bA^1$, which we similarly denote by $\psi_e$. Let $\on{exp}\in D(\bA^1)$ be the exponential sheaf normalized to be in perverse degree $-1$. Then we define
\[
\on{coeff}_{\bO}:=C_{\on{dR}}(U^{-}\overset{\on{ad}}{/}U^{-}, \fp_e^!(-)\overset{!}{\otimes} \psi_e^!(\on{exp})).
\]

\subsubsection{} We may describe the (underlying classical stack of the) cotangent bundle of $G\overset{\on{ad}}{/}G$ as the commuting stack 
\[
T^*(G\overset{\on{ad}}{/}G)=\lbrace (g,x)\in G\times \fg^*\;\vert\; gxg^{-1}=x\rbrace /G.
\]

\noindent In particular, we have a natural map
\[
\pi: T^*(G\overset{\on{ad}}{/}G)\to \fg^*/G.
\]

\subsubsection{} For a closed conical $G$-stable subset $Z\subset \cN$, let
\[
\Lambda_Z:=\lbrace (g,x)\in G\times Z\;\vert\; gxg^{-1}=x\rbrace.
\]

\noindent It is a Lagrangian subvariety of $T^*G$. If $\bO$ is a maximal nilpotent orbit contained in $Z$, then 
\[
\Lambda_{\bO}=\Lambda_Z\underset{Z}{\times}\bO\subset \Lambda_Z^{\on{sm}}
\]

\noindent is contained in the smooth locus of $\Lambda_Z$. Note that $\Lambda_{\bO}$ is possibly disconnected. We let $\Lambda_{\bO,0}$ be the connected component containing $(1,\bO)\subset \Lambda_{\bO}$.

\subsubsection{} Denote by $D_Z(G\overset{\on{ad}}{/}G)$ the category of D-modules on $G\overset{\on{ad}}{/}G$ whose singular supported is contained in $\Lambda_Z/G$. This is the category of character sheaves on $G$. We remind that these are all regular holonomic.

\subsubsection{Microstalks}\label{s:microstalk} Suppose our ground field is the complex numbers so that we may apply the Riemann-Hilbert correspondence. Recall the notion of a microstalk as appearing in the work of Kashiwara-Schapira  \cite[§7.5]{kashiwara1990sheaves}. Given a complex variety $X$, a Lagrangian $\Lambda\subset T^*X$, and a smooth point $(x,\xi)\in\Lambda$, we have a microstalk functor:
\[
\mu_{(x,\xi)}: \on{Shv}_{\Lambda}^{\on{Betti}}(X)\to\on{Vect.}
\]

\noindent Here, $\on{Shv}_{\Lambda}^{\on{Betti}}(X)$ is the category of all sheaves of $\bC$-vector spaces on the topological space underlying $X$. The microstalk functor is computed by taking vanishing cycles for a suitably transverse function. In particular, the functor is not canonical but depends on the chosen function.

We remind that the microstalk functor is t-exact. Moreover, for a perverse sheaf $\cF\in \on{Shv}_{\Lambda}^{\on{Betti}}(X)$, the dimension of $\mu_{(x,\xi)}(\cF)$ is given by the multiplicity of the point $(x,\xi)$ in the characteristic cycle of $\cF$ (see \cite[§9.5+10.3]{kashiwara1990sheaves}).

\subsubsection{} Let $Z\subset \cN$ be a closed conical Ad-invariant subset, and let $\bO\subset Z$ be a maximal nilpotent orbit contained in $Z$. Pick an element $e\in\bO$. Since $(1,e)\in\Lambda_Z$ is a smooth point, applying Riemann-Hilbert, we obtain a microstalk functor:
\[
\mu_{\bO}=\mu_{(1,e)}: D_Z(G\overset{\on{ad}}{/}G)\to\on{Vect.}
\]

\noindent Our first main result of this subsection is the following:
\begin{thm}\label{t:micro} We have an isomorphism of functors:
\[
\on{coeff}_{\bO}\simeq \mu_{\bO}: D_Z(G\overset{\on{ad}}{/}G)\to \on{Vect.}
\]

\noindent That is, $\on{coeff}_{\bO}$ computes the microstalk at $(1,e)$.
\end{thm}
\begin{rem}
The above isomorphism is non-canonical: indeed, the very definition of a microstalk depends on a transverse function. Rather, we prove that $\on{coeff}_{\bO}$ satisfies the desired properties of a microstalk and show that this pins down the functor up to non-canonical isomorphism.
\end{rem}

\subsubsection{} When our ground field is not the complex numbers, the notion of microstalk does not make sense. Instead, one should read Theorem \ref{t:micro} as saying that 
\[
\on{coeff}_{\bO}: D_Z(G\overset{\on{ad}}{/}G)\to \on{Vect} 
\]

\noindent is t-exact, and for coherent $\cF\in D_Z(G\overset{\on{ad}}{/}G)$, the Euler characteristic $\chi(\on{coeff}_{\bO}(\cF))$ computes the multiplicity of $\Lambda_{\bO,0}$ in the characteristic cycle of $\cF$. In Section \ref{s:5}, we study functors satisfying these properties and refer to them as \emph{pseudo-microstalks}.

For the remainder of this subsection, we will ignore this sublety and talk about microstalks even when we are not over the complex numbers, remarking that we really mean pseudo-microstalks in this case.

\subsubsection{Hochschild homology of Whittaker model}

The main application of Theorem \ref{t:micro} is the realization of the Hochschild homology of the Whittaker model of a nilpotent $G$-category as the microstalk of its character sheaf, which we explain next.

\subsubsection{} Recall that given a dualizable DG-category $\sD$, we may talk about its Hochschild homology:
\[
\on{HH}_*(\sD)\in\on{Vect.}
\]

\noindent Concretely, it is calculated as the categorical trace of the identity functor of $\sC$, see e.g. \cite[§2]{ben2013nonlinear}. In particular, if $\sC$ is a dualizable $G$-category, we may consider the Hochschild homology of its (generalized) Whittaker model:
\[
\on{HH}_*(\on{Whit}_{\bO}(\sC)).
\]

\subsubsection{} For a dualizable $G$-category $\sC$, we let
\[
\chi_{\sC}:=\on{MC}_{\sC}(\on{id})\in D(G\overset{\on{ad}}{/}G)
\]

\noindent be the image of the identity endomorphism of $\sC$ under the matrix coefficient functor $\on{MC}_{\sC}$. We refer to $\chi_{\sC}$ as the \emph{character sheaf associated to} $\sC$.
\begin{thm}\label{t:catstate}
Let $\sC$ be a nilpotent $G$-category and $\bO$ a maximal nilpotent orbit contained in $\on{SS}(\sC)$. Then
\[
\on{HH}_*(\on{Whit}_{\bO}(\sC))=\on{coeff}_{\bO}(\chi_{\sC})=\mu_{\bO}(\chi_{\sC}).
\]

\noindent That is, the Hochschild homology of $\on{Whit}_{\bO}(\sC)$ is given by the microstalk of its character sheaf at $(1,e)$.
\end{thm}

\subsection{Classification of finite-dimensional modules of $\sW$-algebras}\label{S:CLASSIFICATIONOFFDIMRAAAAAPS}

We next describe a surprising application of the above results to the representation theory of $\sW$-algebras. Namely, we provide a short proof of the classification of finite-dimensional representations for $\sW$-algebras, whose story we recall below.

\subsubsection{} Let $e\in\bO$ be a nilpotent element. Choosing a Whittaker datum as in Section \ref{S:1.5}, we get the $\sW$-algebra $\sW_e$ (see \cite{losev2011finite}). Let $\lambda\in \ft^*$ be a regular weight. The center $Z(\sW_e)$ of $\sW_e$ identifies with the center of $U(\fg)$, and we may consider the quotient 
\[
\sW_{e,\lambda}=\sW_e/\sW_e\cdot \on{Ker}(\chi_{\lambda}),
\]

\noindent where $\chi_{\lambda}$ is the corresponding character of $Z(\sW_e)$.

\subsubsection{} Let
\[
\sW_{e,\lambda}\on{-mod}^{\on{fin}}
\]

\noindent be the subcategory of $\sW_{e,\lambda}\on{-mod}$ consisting of complexes each of whose cohomologies are (ind-)finite-dimensional. We let 
\[
\sW_{e,\lambda}\on{-mod}^{\on{fin},\heartsuit}
\]

\noindent be the heart of its natural t-structure.

\subsubsection{} Let
\[
K_0(\sW_{e,\lambda}\on{-mod}^{\on{fin},\heartsuit})_k
\]

\noindent be the Grothendieck group of $\sW_{e,\lambda}\on{-mod}^{\on{fin},\heartsuit}$, base-changed to $k$. We have a Chern character map (see §\ref{s:justificationfin}):
\[
K_0(\sW_{e,\lambda}\on{-mod}^{\on{fin},\heartsuit})_k\to \on{HH}_0(\sW_{e,\lambda}\on{-mod}).
\]

\noindent We show that (cf. Lemma \ref{l:HHiscoho})
\begin{equation}\label{eq:HHiscoho'}
\on{HH}_*(\sW_{e,\lambda}\on{-mod})\simeq H^*(\cB_e)[2\on{dim}\cB_e],
\end{equation}

\noindent where $\cB_e$ is the Springer fiber of $e$. In particular
\begin{equation}\label{eq:HHiscoho'0}
\on{HH}_0(\sW_{e,\lambda}\on{-mod})\simeq H^{\on{top}}(\cB_e).
\end{equation}
\begin{rem}
The isomorphism (\ref{eq:HHiscoho'0}) was proven in \cite{etingof2010traces} by different means.
\end{rem}

\subsubsection{} Thus, we obtain a map
\begin{equation}\label{eq:Dodd}
K_0(\sW_{e,\lambda}\on{-mod}^{\on{fin},\heartsuit})_k\to  \on{HH}_0(\sW_{e,\lambda}\on{-mod})\simeq H^{\on{top}}(\cB_e).
\end{equation}
\begin{rem}
An injective map $K_0(\sW_{e,\lambda}\on{-mod}^{\on{fin},\heartsuit})_k\to H^{\on{top}}(\cB_e)$ was constructed in \cite{dodd2014injectivity} using a different approach (see also Remark \ref{r:Dodd}). We expect the map (\ref{eq:Dodd}) to coincide with that of \emph{loc.cit}.
\end{rem}

\subsubsection{} Let $\chi=\chi_{\lambda}$ be the character sheaf on $T$ corresponding to $\lambda$, and let
\[
W_{[\lambda]}=\lbrace w\in W\;\vert \; w(\lambda)-\lambda\in X^{\bullet}(T)\rbrace\subset W
\]

\noindent be the integral Weyl group of $\lambda$. Let
\[
\sH_{G,\chi}^{\on{ren}}=D(B,\chi\backslash G/B,\chi)^{\on{ren}}
\]

\noindent be the corresponding Hecke category, renormalized so that simple sheaves are compact. A consequence of Skryabin's equivalence and Beilinson-Bernstein localization is that we have an action:
\begin{equation}\label{eq:actonfin}
\sW_{e,\lambda}\on{-mod}^{\on{fin}}\curvearrowleft \sH_{G,\chi}^{\on{ren}}.
\end{equation}

\subsubsection{} One has an isomorphism of algebras (see Lemma \ref{l:HHisgroupalg}):
\begin{equation}\label{eq:algdescripofHH}
K_0(\sH_{G,\chi}^{\on{ren}})_k\simeq \on{HH}_0(\sH_{G,\chi}^{\on{ren}})\simeq k[W_{[\lambda]}].
\end{equation}

\noindent Combined with (\ref{eq:HHiscoho'}) and (\ref{eq:actonfin}), we obtain actions:\footnote{We use that $K_0(\sW_{e,\lambda}\on{-mod}^{\on{fin},\heartsuit})_k$ may be realized as $K_0(-)_k$ of the small category consisting of bounded complexes of finite-dimensional $\sW_{e,\lambda}$-modules. It is not a priori clear that this coincides with $K_0(\sW_{e,\lambda}\on{-mod}^{\on{fin}})$, since the category $\sW_{e,\lambda}\on{-mod}^{\on{fin}}$ contains compact object that are unbounded from below.}
\begin{equation}\label{eq:actsinquestion}
K_0(\sW_{e,\lambda}\on{-mod}^{\on{fin},\heartsuit})_k, H^*(\cB_e)\curvearrowleft k[W_{[\lambda]}].
\end{equation}

\noindent We show:
\begin{thm}\label{t:mandenligneretegern}
The action
\[
H^*(\cB_e)\curvearrowleft k[W_{[\lambda]}]
\]

\noindent coincides with the usual Springer action in such a way that the trivial representation corresponds to $e=0$.
\end{thm}

\begin{rem}\label{r:mansqman}
More generally, we obtain an action
\[
\on{Spr}_{\chi}\curvearrowleft k[W_{[\lambda]}]
\]

\noindent on the (twisted) Springer sheaf $\on{Spr}_{\chi}\in D(G\overset{\on{ad}}{/}G)$ coming from trace methods. The key input in proving Theorem \ref{t:mandenligneretegern} is to show that the action on $\on{Spr}_{\chi}$ coincides with the usual Springer action. For $\chi=0$, this was already observed by Xinwen Zhu in \cite[Example 3.2.7]{zhu2018geometric}, and in Section \ref{s:Spr} we prove the general case. We remark that, as far as we could tell, the presence of a non-trivial character sheaf $\chi$ seems to complicate the proof. 
\end{rem}

\begin{rem} Let us emphasize one appealing feature of the above perspective  on the Springer action. Namely, two fundamental appearances of a Weyl group action in geometric representation theory are (i) the wall crossing Harish--Chandra bimodules acting on Lie algebra representations at a fixed regular infinitesimal character, and (ii) the Springer action. Of course, in the situation of (i), one obtains a Weyl group action only after passing to Grothendieck groups, and at the categorical level one instead gets a braid group action, which generates an action of the Hecke category. 

On the other side of Beilinson--Bernstein localization, this action of the Hecke category, in its guise as Harish--Chandra bimodules, becomes simply the convolution action of $D(B, \chi \backslash G / B, \chi)$ on $D(G/B, \chi)$. In particular, in view of Remark \ref{r:mansqman}, these fundamental appearances (i) and (ii) of the Weyl group are `really the same':
the Springer action is the trace of the wall crossing functors. 
\end{rem}

\subsubsection{} Write
\[
K_0(\sH_{G,\chi}^{\on{ren}})_{\overline{\bO},k}
\]

\noindent for the subspace of $K_0(\sH_{G,\chi}^{\on{ren}})_k$ generated by simple sheaves in $\sH_{G,\chi}^{\on{ren},\heartsuit}$ whose singular support is contained in $\overline{\bO}$. We define the vector space
\[
K_0(\sH_{G,\chi}^{\on{ren}})_{\partial\bO,k}
\]

\noindent similarly. Let
\[
K_0(\sH_{G,\chi}^{\on{ren}})_{\bO,k}=K_0(\sH_{G,\chi}^{\on{ren}})_{\overline{\bO},k}/K_0(\sH_{G,\chi}^{\on{ren}})_{\partial\bO,k}.
\]

\noindent Note that by (\ref{eq:algdescripofHH}), we have an injection
\[
H^{\on{top}}(\cB_e)\underset{k[W_{[\lambda]}]}{\otimes}K_0(\sH_{G,\chi}^{\on{ren}})_{\bO,k}\into H^{\on{top}}(\cB_e).
\]

\subsubsection{} The following is the main theorem of this subsection:
\begin{thm}\label{t:classification1}
The map (\ref{eq:Dodd}):
\[
K_0(\sW_{e,\lambda}\on{-mod}^{\on{fin},\heartsuit})_k \to H^{\on{top}}(\cB_e)
\]

\noindent maps isomorphically onto the subspace
\[
H^{\on{top}}(\cB_e)\underset{k[W_{[\lambda]}]}{\otimes}K_0(\sH_{G,\chi}^{\on{ren}})_{\bO,k}.
\]

\noindent The resulting isomorphism
\[
K_0(\sW_{e,\lambda}\on{-mod}^{\on{fin},\heartsuit})_k\xrightarrow{\simeq} H^{\on{top}}(\cB_e)\underset{k[W_{[\lambda]}]}{\otimes}K_0(\sH_{G,\chi}^{\on{ren}})_{\bO,k}
\]

\noindent is $k[W_{[\lambda]}]$-linear for the actions (\ref{eq:actsinquestion}).
\end{thm}

\begin{rem}
Let $A_{\bO}$ be the component group of the centralizer of the $\mathfrak{sl}_2$-triple that occurs in the definition of $\sW_e$. Then $A_{\bO}$ naturally acts on the vector spaces $K_0(\sW_{e,\lambda}\on{-mod}^{\on{fin},\heartsuit})_k, H^{\on{top}}(\cB_e)$. It follows by construction that the map (\ref{eq:Dodd}) is $A_{\bO}$-equivariant.
\end{rem}

\subsubsection{}
This theorem (along with many other interesting results that we do not address in this paper) was proven by Losev-Ostrik \cite[Thm. 7.4(iii)]{losev2014classification} for trivial central character and by Bezrukavnikov-Losev \cite[Thm. 5.2(2)]{bezrukavnikov2018dimension} for arbitrary central character. Our methods differ from those of \emph{loc.cit}.  In particular, we do not appeal to methods in positive characteristic.\footnote{Contrary to \cite{bezrukavnikov2018dimension}, the arguments in \cite{losev2014classification} are carried out in characteristic zero. However, they do appeal to results of \cite{dodd2014injectivity} that rely on methods in positive characteristic.} In fact, the proof of Theorem \ref{t:classification1} is a fairly straightforward consequence of the results described so far, as we explain below.
\begin{rem}\label{r:reconcile}
The statement of Theorem \ref{t:classification1} is stated slightly differently than in \cite{losev2014classification} and \cite{bezrukavnikov2018dimension}. To see that Theorem \ref{t:classification1} indeed coincides with those of \emph{loc.cit}, recall that given a two-sided cell in $W_{[\lambda]}$, we may assign a nilpotent orbit in $\fg$. Concretely, given $w\in W_{[\lambda]}$, we get a simple sheaf
\[
\on{IC}_{w,\chi}\in \sH_{G,\chi}^{\on{ren},\heartsuit}.
\]

\noindent We may then take the unique open nilpotent orbit contained in the singular support of $\on{IC}_{w,\chi}$. This orbit only depends on the two-sided cell of $W_{[\lambda]}$ containing $w$, and so we get a map
\begin{equation}\label{eq:twosidedcells}
\lbrace \on{two-sided\;cells\; in} W_{[\lambda]}\rbrace\to \lbrace \on{nilpotent\; orbits\; in} \fg\rbrace.
\end{equation}

\noindent A two-sided cell $c\subset W_{[\lambda]}$ naturally gives rise to a $W_{[\lambda]}$-module $[c]$ (see \cite[§2.2]{bezrukavnikov2018dimension}). By construction, we have an isomorphism of $W_{[\lambda]}$-modules
\[
\underset{c}{\bigoplus} [c]\simeq K_0(\sH_{G,\chi}^{\on{ren}})_{\bO,k},
\]

\noindent where the sum is taken over all two-sided cells of $W_{[\lambda]}$ mapping to $\bO$ under the map (\ref{eq:twosidedcells}). Thus, Theorem \ref{t:classification} exactly recovers \cite[Thm. 5.2(2)]{bezrukavnikov2018dimension}.
\end{rem}

\subsubsection{} Let us end this subsection by briefly describing the strategy of the proof of Theorem \ref{t:classification1}. 

First, a $\sW_{e,\lambda}$-module is finite-dimensional if and only if under Skryabin's equivalence and Beilinson-Bernstein localization:
\[
\sW_{e,\lambda}\on{-mod}\simeq D(U^{-},\psi_e,\backslash G/B,\chi),
\]

\noindent the associated variety of the quotient of $U(\fg)$ by its annihilator is contained in $\overline{\bO}$. The results of Section \ref{s:whitredflag} now provides an equivalence:
\begin{equation}\label{eq:tensoprod}
\sW_{e,\lambda}\on{-mod}^{\on{fin}}\simeq \sW_{e,\lambda}\on{-mod}\underset{\sH_{G,\chi}^{\on{ren}}}{\otimes}\sH_{G,\chi,\overline{\bO}}^{\on{ren}}.
\end{equation}

\noindent It is really from this equivalence that we derive everything. In fact, in light of (\ref{eq:tensoprod}), Theorem \ref{t:classification1} is very natural.

\subsubsection{} For a dualizable DG-category $\sC$, we write
\[
\on{HP}_*(\sC)\in\on{Vect}
\]

\noindent for its periodic cyclic homology (see §\ref{s:hp}). If $\sC$ is compactly generated, there is a natural Chern character map
\[
K_0(\sC)_k=K_0(\sC)\underset{\bZ}{\otimes}k\to \on{HP}_0(\sC).
\]

\noindent As such, Theorem \ref{t:classification1} is a consequence of the following two statements:
\begin{enumerate}
    \item The map
\[
K_0(\sW_{e,\lambda}\on{-mod}^{\on{fin},\heartsuit})_k\to \on{HP}_0(\sW_{e,\lambda}\on{-mod}^{\on{fin}})
\]

\noindent is an isomorphism.

\item 

We have a canonical identification
\[
\on{HP}_0(\sW_{e,\lambda}\on{-mod}\underset{\sH_{G,\chi}^{\on{ren}}}{\otimes}\sH_{G,\chi,\overline{\bO}}^{\on{ren}})\simeq H^{\on{top}}(\cB_e)\underset{k[W_{[\lambda]}]}{\otimes}K_0(\sH_{G,\chi}^{\on{ren}})_{\bO,k}.
\]
\end{enumerate}

\noindent Note that in view of (\ref{eq:HHiscoho'}) and (\ref{eq:algdescripofHH}), step (2) is perhaps not very surprising. The proof, however, requires some care, and the realization of Whittaker coefficients as microstalks described in Section \ref{S:pseudomicro} is a crucial input.

\subsection{Relation to other work and alternative approaches}

\subsubsection{Harish-Chandra version}
Recall cf. \cite[Thm. 3.5.7]{beraldo2017loop} that one has a Morita equivalence
\[
D(G)\textbf{-mod}\simeq \on{HC}\textbf{-mod},\;\; \sC\mapsto \sC^{G,w}=\fg\on{-mod}\underset{D(G)}{\otimes}\sC,
\]

\noindent where $\on{HC}$ denotes the monoidal category of Harish-Chandra bimodules of $U(\fg)$. Since one may talk about the associated variety of Harish-Chandra bimodules, we obtain another definition of singular support for (dualizable) $G$-categories (equivalently $\on{HC}$-categories) by replacing $D(G)$ with $\on{HC}$ in the definition in §\ref{s:mc}. That is, if $\sC$ is a dualizable $G$-category, we may say that $\sC$ has singular support contained in $Z\subset \fg^*$ if
\[
\on{End}_{\on{DGCat}_{\on{cont}}}(\sC^{G,w})\to \on{HC}
\]

\noindent factors through $\on{HC}_Z$, where the latter denotes the category of Harish-Chandra bimodules whose associated variety is contained in $Z$. Let $\on{SS}^{\on{HC}}(\sC)\subset \fg^*$ be the singular support of $\sC$ in this sense. We do not know if $\on{SS}^{\on{HC}}(\sC)=\on{SS}(\sC)$ in general. However, all the proofs of our theorems hold if we work in the Harish-Chandra setting. In particular, from the Harish-Chandra analogue of Theorem \ref{t:criterionnilp} and Theorem \ref{t:microchar}, we obtain that if either $\on{SS}(\sC)$ or $\on{SS}^{\on{HC}}(\sC)$ is contained in the nilpotent cone, then they coincide.

\subsubsection{Different renormalization}\label{ss:diffren}
By definition, the category $D_Z(G)$ consists of all D-modules $\cF$ on $G$ satisfying that for any $n\in\bZ$ and any coherent $\cG\subset H^n(\cF)$, the singular support of $\cG$ is contained in $Z$ in the usual sense. There is a slightly smaller category $D_Z(G)^{\on{access}}$ defined as follows.\footnote{The notation ''access'' is adopted from \cite{arinkin2020stack}.} Namely, let
\[
D_Z(G)^{c}=D_Z(G)\cap D(G)^c
\]

\noindent be the category of compact objects\footnote{That is, bounded complexes of D-modules each of whose cohomologies is finitely generated.} of $D(G)$ contained in $D_Z(G)$. Then we define
\[
D_Z(G)^{\on{access}}=\on{Ind}(D_Z(G)^c).
\]

\noindent By construction, there is a fully faithful functor
\[
D_Z(G)^{\on{access}}\into D_Z(G)
\]

\noindent This realizes $D_Z(G)$ as the left completion of the natural t-structure on $D_Z(G)^{\on{access}}$, see \cite[§E.5.5]{arinkin2020stack}.

We could have defined the notion of singular support for $G$-categories using $D_Z(G)^{\on{access}}$ instead of $D_Z(G)$. Let us denote by $\on{SS}^{\on{access}}(\sC)\subset \fg^*$ the resulting subset. In general, one has
\[
\on{SS}(\sC)\subset \on{SS}^{\on{access}}(\sC),
\]

\noindent but they genuinely differ in general.\footnote{Say $\sC=D_0(\bP^1)$, the category of lisse sheaves on $\bP^1$, equipped with its $SL_2$-action. Then $\on{SS}(\sC)=0$, but $\on{SS}^{\on{access}}(\bP^1)=\cN$.} However, it follows from Theorem \ref{t:criterionnilp} that $\on{SS}(\sC)$ is nilpotent if and only if $\on{SS}^{\on{access}}(\sC)$ is. If one wants the analogue of Theorem \ref{t:microchar} to be true for $\on{SS}^{\on{access}}(\sC)$, one needs to modify the definition of Whittaker coefficients.

\subsubsection{Work of Gannon}
In \cite{gannon2022classification}, Gannon defined the notion of degenerate and non-degenerate $G$-categories. If $\sC$ is a dualizable $G$-category such that $\on{SS}^{\on{access}}(\sC)$ is contained in the irregular locus of $\fg^*$, then one can show that $\sC$ is degenerate in the sense of \emph{loc.cit}. Moreover, if $\sC$ is degenerate and $\on{SS}(\sC)^{\on{access}}$ is nilpotent, the converse holds.

\subsubsection{Work of Lysenko} In \cite{lysenko2022fourier}, Lysenko associates to a nilpotent orbit a collection of Fourier coefficients of automorphic sheaves with nilpotent singular support. Similar to the statements of §\ref{s:whitredflag} above, he conjectures that one may characterize the kernel of the Fourier coefficients in terms of singular support. The fact that we have access to representation theory simplifies our situation, and our methods do not shed light on Lysenko's conjecture.

\subsubsection{Singular support for non-dualizable $G$-categories}
One slight annoyance in our definition of singular support for $G$-categories is the requirement that the category be dualizable. An advantage of the version $\on{SS}^{\on{access}}(-)$ is that it can be defined in the non-dualizable setting. Namely, for a $G$-category $\sC$, one says that $\on{SS}^{\on{access}}(\sC)\subset Z$ if the coaction map
\[
\on{coact}: \sC\to D(G)\otimes \sC
\]

\noindent coming from the self-duality of $D(G)$ factors through $D_Z(G)^{\on{access}}\otimes \sC$. The point here is that the map
\[
D_Z(G)^{\on{access}}\otimes \sC\to D(G)\otimes \sC
\]

\noindent is fully faithful because the embedding $D_Z(G)^{\on{access}}\into D(G)$ admits a continuous right adjoint. We do not know whether
\[
D_Z(G)\otimes \sC\to D(G)\otimes \sC
\]

\noindent is fully faithful, unless $\sC$ is dualizable.\footnote{In general, $D_Z(G)\into D(G)$ neither admits a continuous left or right adjoint.}

\subsubsection{Other sheaf-theoretic contexts}
In this paper, we work exclusively with D-modules and not, say, $\ell$-adic sheaves. However, this is not for serious reasons: if one works in a setting where the categorical Künneth formula holds for $\ell$-adic sheaves, our results naturally extend to this context.\footnote{Such a context has long been anticipated but at the time of writing has not appeared yet. That being said, as we saw above, the study of nilpotent $G$-categories reduces to the study of module categories for (twisted) Hecke categories, and the latter categories are independent of the sheaf theory.}
\subsection{Organization of the paper}

In Section \ref{s:notation}, we recall preliminary material.

In Section \ref{s:yeahboooi}, we prove the theorems stated in §\ref{s:whitredflag}.

In Section \ref{s:4}, we define and study singular support of $G$-categories. Specifically, we study how parabolic induction/restriction interacts with singular support. We prove the theorems in §\ref{s:nilpcats}, and finally we prove Theorem \ref{t:microchar}, the main result of this paper.

In Section \ref{s:5}, we show that Whittaker coefficients corresponding to nilpotent orbits that are maximal in the singular support of principal series character sheaves are given by microstalks. As a consequence, we prove Theorem \ref{t:catstate}.

In Section 6, we prove the classification of finite-dimensional modules of $\sW$-algebras.

Finally, there is an appendix. It is devoted to the proof of a technical proposition appearing in Section \ref{s:yeahboooi}.

\subsection{Acknowledgements}
We thank Pramod Achar, Pablo Boixeda Alvarez, Tomoyuki Arakawa, Roman Bezrukavnikov, Thomas Creutzig, Dennis Gaitsgory, Do Kien Hoang, Penghui Li, Simon Riche, Nick Rozenblyum, Pavel Safronov, Pierre Schapira, Peng Shan, David Yang, Zhiwei Yun, and Xinwen Zhu for discussions which influenced our thinking on this subject. 

We particularly thank Ivan Losev for very helpful discussions related to Section \ref{S:6} and for helping us navigate the literature. 

Finally, the authors extend their deepest gratitudes to Sam Raskin. The ideas presented in this paper were significantly shaped by discussions held with him. In particular, the idea of comparing K-theory and Hochschild homology of the category of $\sW$-modules via an intermediate comparison with periodic cyclic homology is a suggestion of his.

G.D. was supported by an NSF Postdoctoral Fellowship under grant No. 2103387.

\section{Notation and conventions}\label{s:notation}
In this section, we set up notation and conventions used in the paper.

\subsection{Notation}\label{s:notation'}

\subsubsection{}
We work an algebraically closed field $k$ of characteristic zero. Let $G$ be a connected reductive algebraic group over $k$. We let $B\subset G$ be a Borel subgroup, and we choose an opposite Borel $B^{-}$. We let $N, N^{-}$ denote the unipotent radicals of $B,B^-$ and let $T=B/N$. Denote by $\fg,\fb^{-},\fn,\fn^{-},\ft$ the corresponding Lie algebras. We let $\cN\subset \fg$ be the nilpotent cone.

We fix a $G$-invariant non-degenerate symmetric form $(\cdot,\cdot)$ on $\fg$. In particular, we obtain a $G$-invariant isomorphism $\fg\simeq \fg^*$. As such, we will also view $\cN$ as living inside $\fg^*$.

\subsubsection{} We let $X^{\bullet}(T), X_{\bullet}(T)$ be the lattice of characters and cocharacters of $T$. Let $W$ be the finite Weyl group of $G$ and $\widetilde{W}^{\on{aff}}=X^{\bullet}(T)\rtimes W$ the extended affine Weyl group for $G$.

\subsection{Higher category theory}
We use the language of higher category theory and higher algebra in the sense of \cite{lurie2006higher}, \cite{lurie2017higher}, \cite{gaitsgory2019study}. Throughout, by a $DG$-category, we mean a $k$-linear presentable stable $(\infty,1)$-category. We denote by $\on{DGCat}_{\on{cont}}$ the $(\infty,2)$-category of DG-categories in which morphisms between DG-categories are required to be colimit-preserving (continuous). We will often refer to DG-categories simply as \emph{categories}.

We denote by $\on{Vect}$ the (DG-)category of $k$-vector spaces.

If a category $\sC$ is equipped with a t-structure, we write $\sC^{\heartsuit}$ for its heart.

\subsection{D-modules}
\subsubsection{}
Given an algebraic stack $\sY$ locally of finite type, we let $D(\sY)$ be the category of D-modules on $\sY$ in the sense of \cite{gaitsgory2019study}.

\subsubsection{} In general, we refer to \cite[Chap. 4]{gaitsgory2017study} for the study of functoriality between D-modules on algebraic stacks locally of finite type. Below, we summarize the functoriality needed for our purposes.

\subsubsection{} Let $f:\sX\to \sY$ be a map of algebraic stacks locally of finite type. We have the pullback functor:
\[
f^!: D(\sY)\to D(\sX).
\]

\noindent Whenever its left adjoint is defined, we denote it by $f_!$.

\subsubsection{} If $f$ is representable, we also have a pushforward functor:
\[
f_{*,\on{dR}}: D(\sX)\to D(\sY).
\]

\noindent Whenever its left adjoint is defined, we denote it by $f^{*,\on{dR}}$. We remind that the functors $f_!,f^{*,\on{dR}}$ are defined on holonomic D-modules.

\subsection{Singular support conventions}

\subsubsection{Singular support for varieties}\label{s:singsupp}
Let $X$ be a smooth variety and $Z\subset T^*X$ a closed conical subset. Recall that $D(X)$ carries a natural t-structure (see e.g. \cite[\S 4.]{gaitsgory2011crystals}). We write $D(X)^{\heartsuit}$ for the heart of its t-structure. 

\subsubsection{} We let $D_{Z}(X)^{\heartsuit,c}$ be the abelian category of coherent D-modules $\cF\in D(X)^{\heartsuit}$ whose singular support is contained in $Z$. Note that $D_{Z}(X)^{\heartsuit,c}$ is stable under extensions and subquotients. We consider the category obtained by taking its ind-completion:
\[
D_{Z}(X)^{\heartsuit}:=\on{Ind}(D_{Z}(X)^{\heartsuit,c})\subset \on{Ind}(D(X)^{\heartsuit,c})=D(X)^{\heartsuit}.
\]

\noindent We let $D_{Z}(X)\subset D(X)$ be the full subcategory consisting of $\cF\in D(X)$ such that each cohomology sheaf $H^i(\cF)$ lies in $D_{Z}(X)^{\heartsuit}$.

For an arbitrary $\cF\in D(X)$, we let $\on{SS}(\cF)\subset T^*X$ be the smallest closed conical subset such that $\cF\in D_{\on{SS}(\cF)}(X)$.

\subsubsection{Functoriality}\label{s:functsing}
We record the following well-known statements about the estimate of singular support under smooth pullback and proper pushforward (see e.g. \cite[§2.2.5]{faergeman2022quasi} for a more detailed account). For a map $f:X\to Y$ of varieties, we get an induced correspondence
\[\begin{tikzcd}
	{T^*Y} && {T^*Y\underset{Y}{\times} X} && {T^*X.}
	\arrow["\pi"', from=1-3, to=1-1]
	\arrow["{df^{\bullet}}", from=1-3, to=1-5]
\end{tikzcd}\]

\noindent If $f$ is smooth, then for any $\cF\in D(Y)$, we have:
\[
\on{SS}(f^!(\cF))=df^{\bullet}(\pi^{-1}(\on{SS}(\cF)))\subset T^*X.
\]

\noindent If $f$ is proper, then for any $\cG\in D(X)$, we have:
\[
\on{SS}(f_{*,\on{dR}}(\cG))\subset \pi((df^{\bullet})^{-1}(\on{SS}(\cG)))\subset T^*Y.
\]

\noindent Moreover, if $f$ is a closed embedding, the above inclusion is an equality.

\subsubsection{Conventions concerning the flag variety}
Let $X=G/B$ be the flag variety, and let $\mu$ be the Springer map
\[
\widetilde{\cN}=T^*(G/B)\to \cN.
\]

\noindent For a closed $G$-stable subset $Z\subset \cN$, we write
\[
D_{Z}(G/B):=D_{\mu^{-1}(Z)}(G/B).
\]

\noindent Explicitly, $\cF\in D(G/B)$ lies in $D_{Z}(G/B)$ if and only if for all $n\in \bZ$ and every coherent sub D-module $\cG\subset H^n(\cF)$, $\mu(\on{SS}(\cG))\subset Z$.

\subsubsection{}\label{s:Icantbelieveitsnotdoneyet} More generally, if $\chi=\chi_{\lambda}$ is the character sheaf on $B$ corresponding to some field-valued point $\lambda\in \ft^*/X^{\bullet}(T)$, we may consider the category 
\[
D(G/B,\chi)
\]

\noindent of $(B,\chi)$-equivariant D-modules of $G$. If $\lambda$ is a $k$-point (as opposed to $K$-point, where $K/k$ is a trascendental field extension), then for any coherent $\cF\in D(G/B,\chi)$, the singular support of $\cF$ naturally lives on $\widetilde{\cN}$. We similarly write
\[
D_{Z}(G/B,\chi):=D_{\mu^{-1}(Z)}(G/B,\chi)
\]

\noindent in this case. For $\cF\in D(G/B,\chi)$, we abuse notation by writing $G\cdot \on{SS}(\cF)\subset \cN$ (resp. $\overline{G\cdot \on{SS}(\cF)}$) for $\on{Ad}_G(\mu(\on{SS}(\cF)))$ (resp. its closure) in $\cN$.

\subsubsection{Hecke categories} For $\chi\in\ft^*/X^{\bullet}(k)$, consider the Hecke category
\[
D(B,\chi\backslash G/B,\chi)
\]

\noindent of bi-$(B,\chi)$-equivariant D-modules on $G$. We write 
\[
D_Z(B,\chi\backslash G/B,\chi)
\]

\noindent for the subcategory consisting of $\cF\in D(B,\chi\backslash G/B,\chi)$ that map to $D_Z(G/B,\chi)$ under the forgetful functor
\[
D(B,\chi\backslash G/B,\chi)\to D(G/B,\chi).
\]

\subsubsection{} We record the following proposition of \cite{campbell2021affine} (see also \cite{ben2020highest} for a general setup), which we use throughout the paper.\footnote{The proof in \cite{campbell2021affine} is not formulated in the presence of a character $\chi$, but the proof remains the same.}
\begin{prop}\label{p:genfull}
For any $G$-category $\sC$, the convolution functor
\[
D(G/B,\chi)\underset{D(B,\chi\backslash G/B,\chi)}{\otimes} \sC^{B,\chi}\to \sC
\]

\noindent is fully faithful.
\end{prop}

\subsubsection{}\label{s:chigen}

Henceforth, we write
\[
\sC^{(B,\chi)\on{-gen}}:=D(G/B,\chi)\underset{D(B,\chi\backslash G/B,\chi)}{\otimes} \sC^{B,\chi}\subset \sC.
\]

\noindent We say that $\sC$ is $(B,\chi)$-generated if the inclusion
\[
\sC^{(B,\chi)\on{-gen}}\subset \sC
\]

\noindent is an equivalence.

\subsection{Categorical actions}
Throughout the paper, we use the language of categorical actions and categorical representation theory. In general, we refer to \cite{beraldo2017loop}, \cite{gaitsgory2019study} for our conventions.

\subsubsection{} Given an algebraic group $H$, the category $D(H)$ is equipped with a convolution monoidal structure. If a category $\sC$ is equipped with a module structure for $D(H)$, we say that $\sC$ carries a \emph{strong} $H$-action. We denote by 
\[
D(H)\on{\textbf{-mod}}
\]

\noindent the $(\infty,2)$-category of DG-categories equipped with a strong $H$-action. We will refer to objects of $D(H)\on{\textbf{-mod}}$ simply as $H$-categories.

\subsubsection{} Let $K\subset H$ be an algebraic subgroup of $H$. If $\sC$ is an $H$-category, we define
\[
\sC^K:=\on{Hom}_{D(K)\on{\textbf{-mod}}}(\on{Vect}, \sC)
\]

\noindent to be the category of $K$-invariants of $\sC$. 

\subsubsection{} The category 
\[
D(K\backslash H/K)\simeq \on{End}_{D(H)\on{\textbf{-mod}}}(D(H/K))
\]

\noindent of bi-$K$-invariant D-modules on $H$ is naturally equipped with a monoidal structure by convolution. We have a natural functor
\[
D(H)\on{\textbf{-mod}}\to D(K\backslash H/K)\on{\textbf{-mod}},\;\; \sC\to \sC^K.
\]

\subsubsection{} Whenever $\chi$ is a character sheaf on $K$, one may talk about \emph{twisted} invariants:
\[
\sC^{K,\chi}:=\on{Hom}_{D(K)\on{\textbf{-mod}}}(\on{Vect}, \sC),
\]

\noindent where the action $D(K)\curvearrowright \on{Vect}$ is given by 
\[
\cF\star V:=C_{\on{dR}}(K,\chi\overset{!}{\otimes} \cF)\otimes V\in\on{Vect.}
\]

\noindent The above constructions for $K$-invariants go through verbatim for $(K,\chi)$-invariants.
 
\subsection{Whittaker theory}\label{s:genwhit} Let $e\in\cN$ be a nilpotent element, and let $\bO$ the orbit containing $e$. Extend $e$ to an $\mathfrak{sl}_2$-triple $(e,h,f)$. Recall cf. §\ref{s:notation'} that we fixed a $G$-invariant non-degenerate symmetric form $(\cdot,\cdot)$ on $\fg$.

\subsubsection{} From the $\mathfrak{sl}_2$-triple, we get a decomposition $\fg=\underset{i\in \bZ}{\oplus} \fg(i)$, where
\[
\fg(i):=\lbrace x\in \fg\;\vert\; [h,x]=i\cdot x\rbrace.
\]

\noindent Let $\fp^-:=\underset{i\leq 0}{\oplus} \fg(i)$. Note that $f\in \fg(-2)\subset \fp^{-}$ and $e\in\fg(2)$. We write $\psi_e$ for the additive character 
\[
\psi_e=(e,\cdot): \underset{i\leq -2}{\oplus}\fg(i)\to k
\]

\noindent induced by $e\in \fg(2)$ under the form $(\cdot,\cdot):$ $\fg\xrightarrow{\simeq} \fg^*$. Note that $\psi_e$ induces a symplectic form $\omega_e$ on $\fg(-1)$ given by 
\[
\omega_e(x,y):=\psi_e([x,y]).
\]

\noindent We fix a Lagrangian $\ell\subset \fg(-1)$ with respect to $\omega_e$ and write
\[
\fu^-=\fu^-_{\ell,e}:=\ell\oplus\underset{i\leq -2}{\bigoplus} \fg(i)\subset \underset{i\leq -1}{\oplus} \fg(i).
\]

\noindent $\psi_e$ tautologically extends to a character on $\fu^-$, also denoted $\psi_e$, which vanishes on $\ell$.

\subsubsection{} We write $U^-$ for the unipotent subgroup of $G$ such that $\on{Lie}(U^-)=\fu^-$. Note that $\psi_e$ gives rise to a character
\[
U^-\to \bA^1.
\]

\noindent Let $\on{exp}\in D(\bA^1)$ be the exponential sheaf on $\bA^1$ concentrated in perverse degree $-1$, the same degree as the dualizing sheaf. Pulling $\on{exp}$ back to $U^-$, we get a character sheaf on $U^-$, which we similarly denote by $\psi_e$.

\section{Whittaker reduction of D-modules on the flag variety}\label{s:yeahboooi}

In this section, we study generalized Whittaker coefficients of (twisted) D-modules on the flag variety. Our main result is Theorem \ref{t:van} below that provides a micro-local characterization of the vanishing of generalized Whittaker coefficients.

\subsubsection{}\label{s:convsheaff} For each nilpotent element $e$, we have an averaging functor:
\[
\on{Av}_*^{\psi_e}: D(G/B,\chi) \to D(G/U^-,\psi_e).
\]

\noindent Explicitly, writing $m$ for the map multiplication map $B\times U^-\to G$, the above functor is given by convolving on the right with the sheaf $m_{*,\on{dR}}(\chi\boxtimes \psi_e)$ considered as an object in the category
\[
D(B,\chi\backslash G/U^-,\psi_e).
\]

\subsubsection{} Let $Z\subset \cN\subset \fg^*$ be a closed, conical Ad-invariant subset of the nilpotent cone. To simplify notation, we write
\[
\sH_{G,\chi}:=D(B,\chi\backslash G/B,\chi), \;\;\; \sH_{G,\chi,Z}:=D_Z(B,\chi\backslash G/B,\chi).
\]

\noindent We may consider the corresponding subcategory
\[
D(G/B,\chi) \underset{\sH_{G,\chi}}{\otimes} \sH_{G,\chi,Z}\subset D(G/B,\chi).
\]

\subsubsection{} The purpose of this section is to prove the following theorem:
\begin{thm}\label{t:van}
We have an identification
\[
D(G/B,\chi) \underset{\sH_{G,\chi}}{\otimes} \sH_{G,\chi,Z}=\underset{e\in\bO\not\subset Z}{\bigcap}\on{Ker}(\on{Av}_*^{\psi_e})
\]

\noindent as subcategories of $D(G/B,\chi)$. That is, for any $\cF\in D(G/B,\chi) \underset{\sH_{G,\chi}}{\otimes} \sH_{G,\chi,Z}$ and $e\notin Z$, one has $\on{Av}_*^{\psi_e}(\cF)=0$. Conversely, if for all $e\notin Z$, we have $\on{Av}_*^{\psi_e}(\cF)=0$, then $\cF\in D(G/B,\chi) \underset{\sH_{G,\chi}}{\otimes} \sH_{G,\chi,Z}$.
\end{thm}

\begin{rem}
In \cite{faergeman2022non}, it was proved that for the regular orbit $\bO_{\on{reg}}$, any holonomic $\cF\in D(G/B)$ with $\on{SS}(\cF)\cap \bO_{\on{reg}}=\emptyset$ satisfies $\on{Av}_*^{\psi_e}(\cF)=0$ for all $e\in\bO_{\on{reg}}$.
\end{rem}

\subsection{Non-vanishing of generalized Whittaker coefficients}\label{s:3.1r}
In this section, we prove one direction of Theorem \ref{t:van}:
\begin{thm}\label{t:van2}
Let $\cF\in D(G/B,\chi)$. If for all $e\notin Z$, we have $\on{Av}_*^{\psi_e}(\cF)=0$, then $\cF\in D(G/B,\chi) \underset{\sH_{G,\chi}}{\otimes} \sH_{G,\chi,Z}$.
\end{thm}

\subsubsection{} First, we have the following standard observation.
\begin{lem}\label{l:maxvsopen}
Let $Z\subset \cN$ be a locally closed $G$-stable subset of the nilpotent cone. Then an orbit $\bO$ is maximal with respect to the property that $\bO\cap Z\neq \emptyset$ if and only if $\bO\cap Z$ is open in $Z$.
\end{lem}

\subsubsection{} We reduce Theorem \ref{t:van2} to the corresponding assertion on the Hecke category:
\begin{thm}\label{p:defertoappendix}
Let $Z\subset \cN$ be a closed $G$-stable subset, and let $\bO$ be a maximal orbit intersecting $Z$. If $\cF\in D_Z(B,\chi\backslash G/B,\chi)$ is killed under the functor
\[
D(B,\chi\backslash G/B,\chi)\xrightarrow{\on{Av}_*^{\psi_e}} D(B,\chi\backslash G/U^-,\psi_e),
\]

\noindent for all $e\in \bO$, then $\on{SS}(\cF)\subset Z\setminus \bO$.
\end{thm}

\noindent \emph{Proof of Theorem \ref{t:van2}.} Let $\cF\in D(G/B,\chi)$ satisfy the hypothesis of the theorem. Since $D(G/B,\chi)$ is dualizable as a $\sH_{G,\chi}$-module category, it suffices to show that for every $\phi\in \on{Hom}_{\sH_{G,\chi}}(D(G/B,\chi),\sH_{G,\chi})$, we have:
\[
(\phi\otimes\on{id})(\cF)\in \sH_{G,\chi,Z}.
\]

\noindent By Theorem \ref{p:defertoappendix}, if this was not the case, we could find $e\notin Z$ such that $\on{Av}_*^{\psi_e}(\cF)\neq 0$.

\qed

\subsubsection{} In turn, we are going to reduce the proof of Theorem \ref{p:defertoappendix} to an analogous statement about Harish-Chandra bimodules, where the assertion follows from results of Losev \cite{losev2011finite}.

\subsubsection{}
Write $Z(\fg)$ for the center of $U(\fg)$. For $\lambda\in \ft^*$, let $U(\fg)_{\lambda}=U(\fg)/U(\fg)\cdot \on{Ker}(\chi_{\lambda})$, where $\chi_{\lambda}: Z(\fg)\to k$ is the character of $Z(\fg)$ given by the image of $\lambda$ under the map $\ft^*\to \on{Spec}(Z(\fg))$ induced by the Harish-Chandra isomorphism. We denote by $\fg\on{-mod}_{\lambda}$ the category of $U(\fg)_{\lambda}$-modules. For coherent $M\in \fg\on{-mod}_{\lambda}^{\heartsuit}$, we may consider its associated variety in $\cN=\on{Spec}(\on{gr}U(\fg)_{\lambda})$, which we will similarly denote by $\on{SS}(M)$.

\subsubsection{}\label{s:BBloc} Fix a character sheaf $\chi\in \ft^*/X^{\bullet}(T)(k)$ on $T$. We may find a dominant and regular weight $\lambda=\lambda_{\chi}\in\ft^*$ such that $\chi$ and $\lambda$ coincide under the maps
\[
\ft^*/X^{\bullet}(T)\to \ft^*/\widetilde{W}^{\on{aff}} \leftarrow \ft^*.
\]

\noindent Beilinson-Bernstein localization gives a t-exact equivalence of categories
\begin{equation}\label{eq:BBloc}
D(G/B,\chi)\simeq \fg\on{-mod}_{\lambda}, \;\; \cF\mapsto \Gamma(\cF).
\end{equation}

\noindent The singular support of $\cF\in D(G/B,\chi)^{\heartsuit,c}$ maps onto the associated variety of $\Gamma(\cF)$ under the Springer map $\widetilde{N}\to\cN$, see \cite[\S 1.9]{borho1985differential}.

\subsubsection{} We denote by $\on{HC}_{\lambda,\lambda}$ the category of Harish-Chandra bimodules with central character $(\lambda,\lambda)$. Concretely,
\[
\on{HC}_{\lambda,\lambda}=\fg\on{-mod}_{\lambda}\underset{D(G)}{\otimes}\fg\on{-mod}_{\lambda}.
\]

\noindent By (\ref{eq:BBloc}), we may rewrite $\on{HC}_{\lambda,\lambda}$ as:
\[
\on{HC}_{\lambda,\lambda}\simeq D(B,\chi\backslash G) \underset{D(G)}{\otimes} D(G/B,\chi)\simeq D(B,\chi\backslash G/B,\chi).
\]

\noindent For $M\in\on{HC}_{\lambda,\lambda}$, we may talk about its associated variety, denoted $\on{SS}(M)$. It naturally lives on $\cN/G$. It follows from \cite{borho1985differential} that the t-exact equivalence
\[
\on{HC}_{\lambda,\lambda}\simeq D(B,\chi\backslash G/B,\chi)
\]

\noindent is compatible with singular support. That is, for $\cF\in D(B,\chi\backslash G/B,\chi)$, the singular support of the corresponding module in $\on{HC}_{\lambda,\lambda}$ is the image of the Springer map
\[
\fn/B\simeq \widetilde{\cN}/G\to \cN/G.
\]

\subsubsection{} Let $e\in\bO$ be a nilpotent element. We will be concerned with the functor
\[
\on{Av}_!^{\psi_e,\psi_e}: D(B,\chi\backslash G/B,\chi)\to D(B,\chi\backslash G)\otimes D(G/B,\chi)\xrightarrow{\on{Av}_!^{\psi_e}\otimes\on{Av}_!^{\psi_e}} D(B,\chi\backslash G/U^-,\psi_e)\otimes D(U^-,\psi_e\backslash G/B,\chi).
\]

\noindent Here:
\begin{itemize}
    \item The first functor is given by 
    \[
    D(B,\chi\backslash G/B,\chi)\simeq D(B,\chi\backslash G\times G/B,\chi)^{\Delta G}\xrightarrow{\on{oblv}^{\Delta G}} D(B,\chi\backslash G)\otimes D(G/B,\chi),
    \]

    \noindent where $G=\Delta G\into G\times G$ is the diagonal embedding.

    \item The second functor is the partially defined left adjoint to the forgetful functor
    \[
    D(B,\chi\backslash G/U^-,\psi_e)\otimes D(U^-,\psi_e\backslash G/B,\chi)\to D(B,\chi\backslash G)\otimes D(G/B,\chi).
    \]

    \noindent We remind that this functor is defined on holonomic D-modules; in particular, it is defined on the image of the functor $D(B,\chi\backslash G/B,\chi)\to D(B,\chi\backslash G)\otimes D(G/B,\chi)$ above.
\end{itemize}

\subsubsection{} Combining Beilinson-Bernstein localization and Skryabin's equivalence, we obtain an equivalence
\[
D(B,\chi\backslash G/U^-,\psi_e)\simeq \fg\on{-mod}_{\lambda}^{U^-,\psi_e}\simeq \sW_{e,\lambda}\on{-mod},
\]

\noindent where $\sW_e$ is the $\sW$-algebra corresponding to $e$, and $\sW_{e,\lambda}\on{-mod}$ is the category of $\sW_e$-modules with central character $\lambda$ (see \cite{losev2010quantized}).

As such, the functor $\on{Av}_!^{\psi_e,\psi_e}$ defined above gives rise to a functor
\begin{equation}\label{eq:biwhit}
\on{Av}_!^{\on{HC},\sW_e}:\on{HC}_{\lambda,\lambda}\to \sW_{e,\lambda}\on{-Bimod.}
\end{equation}

\noindent Here, $\sW_{e,\lambda}\on{-Bimod}$ denotes the category of bimodules for $\sW_e$ with central character $(\lambda,\lambda)$. In fact, the functor (\ref{eq:biwhit}) naturally lands in the category of Harish-Chandra bimodules for $\sW_e$ considered in \cite{losev2011finite}.

\subsubsection{} We will need the following theorem, which follows easily from results of Losev \cite{losev2011finite}:
\begin{thm}\label{t:HCWhit}
Let $\cF\in D(B,\chi\backslash G/B,\chi)$ be non-zero and suppose $\bO$ is a nilpotent orbit which is maximal with respect to the property that $\bO\cap \on{SS}(\cF)\neq \emptyset$. Then for any $e\in\bO$, one has $\on{Av}_!^{\psi_e,\psi_e}(\cF)\neq 0$.
\end{thm}

\begin{proof}
It suffices to show that for any $M\in \on{HC}_{\lambda,\lambda}$ such that $\bO$ is maximal in $\on{SS}(M)$, the image of $M$ under $\on{Av}_!^{\on{HC},\sW_e}$ is non-zero. This functor is considered in \cite{losev2011finite}.\footnote{While it is not immediately clear that the functor $\on{Av}_!^{\on{HC},\sW_e}$ coincides with the one in \cite{losev2011finite}, this follows from §3.5 of \emph{loc.cit} (at least up to cohomological shift).} By \cite[Thm. 1.3.1 (1)]{losev2011finite}, $\on{Av}_!^{\on{HC},\sW_e}$ is t-exact up to a cohomological shift. In particular, we may assume that $M\in\on{HC}_{\lambda,\lambda}^{\heartsuit}$ lies in the heart of the t-structure and is finitely generated. In this case, the fact that $\on{Av}_!^{\on{HC},\sW_e}(M)\neq 0$ follows from \cite[Thm. 1.3.1 (2)]{losev2011finite}.
\end{proof}

\subsubsection{} 

We now turn our attention to the study of the functor
\begin{equation}\label{eq:!ave}
D(B,\chi\backslash G/B,\chi)\xrightarrow{\on{oblv}^{B,\chi}} D(B,\chi\backslash G)\xrightarrow{\on{Av}_!^{\psi_e}} D(B,\chi\backslash G/U^-,\psi_e).
\end{equation}

\noindent Concretely, let $m: B\times U^-\to G,\; (b,u)\mapsto bu$. Then the above composition is given by convolving with $m_!(\chi\boxtimes \psi_e)$ considered as an object of $D(B,\chi\backslash G/U^-,\psi_e)$.

\subsubsection{} We want to show an analogue of Theorem \ref{t:HCWhit} for the functor (\ref{eq:!ave}). Before doing so, we need a preliminary lemma. For a smooth variety $X$, denote by $\on{Shv}(X)$ the category of ind-holonomic D-modules on $X$. Concretely, 
\[
\on{Shv}(X)=\on{Ind}(\on{Shv}(X)^c),
\]

\noindent where $\on{Shv}(X)^c$ is the category of bounded complexes of D-modules on $X$, each of whose cohomologies is a coherent holonomic D-module.
\begin{lem}\label{l:*cons}
Let $\cF\in\on{Shv}(X)$. There exists a $k$-point $x\in X$ such that the $*$-fiber of $\cF$ at $x$ is non-zero.
\end{lem}

\begin{rem}
We remark that the analogue of Lemma \ref{l:*cons} for $!$-fibers is false in general.
\end{rem}

\begin{proof}
First, note that the assertion of the lemma is true whenever $\cF\in\on{Shv}(X)^c$. Next, recall that $\on{Shv}(X)$ has a (constructible) t-structure characterized by the fact that taking $*$-fibers at $k$-points are t-exact. As such, we may assume that $\cF$ lies in the heart of this t-structure and is an element of $\on{Shv}(X)^c$. 
\end{proof}

\begin{thm}\label{p:!ave}
Let $\cF\in D(B,\chi\backslash G/B,\chi)$ be non-zero and suppose $\bO$ is a maximal orbit intersecting $\on{SS}(\cF)$. Then there exists $e\in\bO$ such that
\[
0\neq \on{Av}_!^{\psi_e}(\cF)\in D(B,\chi\backslash G/U^-,\psi_e).
\]
\end{thm}

\begin{proof}

\emph{Step 1.}
Fix some $\tilde{e}\in \bO$. Consider the functor
\[
F: D(B,\chi\backslash G/B,\chi)\xrightarrow{\on{oblv}^{\Delta G}} D(B,\chi\backslash G)\otimes D(G/B,\chi)\to
\]
\[
\xrightarrow{\on{Av}_!^{\psi_{\tilde{e}}}\otimes \on{id}} D(B,\chi\backslash G/U^-,\psi_{\tilde{e}})\otimes D(G/B,\chi)\xrightarrow{\on{oblv}^{\psi_{\tilde{e}}}\otimes\on{id}} D(B,\chi\backslash G)\otimes D(G/B,\chi).
\]

\noindent Let $g\in G(k)$. Denote by $g^{*,\on{dR}}: D(B,\chi\backslash G)\to D(B,\chi\backslash G)$ the functor of $*$-pullback along the map $G\to G$ induced by multiplication by $g$ on the right. Write $i_g:G\times \lbrace g\rbrace\to G\times G$ for the natural embedding. By base-change, the functor
\[
i_g^{*,\on{dR}}\circ F: D(B,\chi\backslash G/B,\chi)\to D(B,\chi\backslash G)
\]

\noindent coincides with the composition
\[
D(B,\chi\backslash G/B,\chi)\xrightarrow{\on{oblv}^{B,\chi}} D(B,\chi\backslash G)\xrightarrow{g^{*,\on{dR}}} D(B,\chi\backslash G)\to
\]
\[
\xrightarrow{\on{Av}_!^{\psi_{\tilde{e}}}} D(B,\chi\backslash G/U^-,\psi_{\tilde{e}})\xrightarrow{\on{oblv}^{\psi_{\tilde{e}}}} D(B,\chi\backslash G).
\]

\noindent Write $U^-_g=gU^-g^{-1}\subset G$ with its natural character $\psi_{g\tilde{e}g^{-1}}=\psi_{\tilde{e}}\circ \on{Ad}_{g^{-1}}: U^-_g\to \bA^1$. We may rewrite the above composition as:
\[
D(B,\chi\backslash G/B,\chi)\xrightarrow{\on{oblv}^{B,\chi}} D(B,\chi\backslash G)\to
\]
\[
\xrightarrow{\on{Av}_!^{\psi_{g\tilde{e}g^{-1}}}} D(B,\chi\backslash G/U^-_g,\psi_{g\tilde{e}g^{-1}})\xrightarrow{\on{oblv}^{\psi_{g\tilde{e}g^{-1}}}} D(B,\chi\backslash G)\xrightarrow{g^{*,\on{dR}}} D(B,\chi\backslash G).
\]

\emph{Step 2.} By Theorem \ref{t:HCWhit} and Lemma \ref{l:*cons}, we may find $g\in G(k)$ such that $i_g^{*,\on{dR}}\circ F(\cF)\neq 0$. From Step 1, we see that $\on{Av}_!^{\psi_{g\tilde{e}g^{-1}}}(\cF)\neq 0$.

\end{proof}

\subsubsection{} We are now in a position to prove Theorem \ref{p:defertoappendix}. That is, we need to prove the analogue of Theorem \ref{p:!ave} with $\on{Av}_!^{\psi_e}$ replaced by $\on{Av}_*^{\psi_e}$.

Let 
\[
\psi_{e,!}:=\on{Av}_*^{B,\chi}(m_!(\chi\boxtimes \psi_e))\in D(B,\chi\backslash G/B,\chi),
\]

\noindent where $\on{Av}_*^{B,\chi}: D(B,\chi\backslash G/U^-,\psi_e)\to D(B,\chi\backslash G/B,\chi)$ is right adjoint to $\on{Av}_!^{\psi_e}$, and $m$ is the multiplicatio map $B\times U^-\to G$. Similarly, we let
\[
\psi_{e,*}:=\on{Av}_!^{B,\chi}(m_{*,\on{dR}}(\chi\boxtimes \psi_e))\in D(B,\chi\backslash G/B,\chi),
\]

\noindent where $\on{Av}_!^{B,\chi}: D(B,\chi\backslash G/U^-,\psi_e)\to D(B,\chi\backslash G/B,\chi)$ is left adjoint to $\on{Av}_*^{\psi_e}$.

\vspace{2mm}

\begin{proof}[Proof of Theorem \ref{p:defertoappendix}] Let $\cF\in D(B,\chi\backslash G/B,\chi)$ be non-zero. We need to find $e\in\bO$ such that $0\neq\on{Av}_*^{\psi_e}(\cF)\in D(B,\chi\backslash G/U^-,\psi_e)$. Let $e\in \bO$ be the element appearing in the statement of Theorem \ref{p:!ave}.

Consider the composition 
\[
\on{Av}_*^{B,\chi}\circ \on{Av}_!^{\psi_e}: D(B,\chi\backslash G/B,\chi)\xrightarrow{\on{Av}_!^{\psi_e}} D(B,\chi\backslash G/U^-,\psi_e)\xrightarrow{\on{Av}_*^{B,\chi}} D(B,\chi\backslash G/B,\chi).
\]

\noindent  The above composition is given by convolving with $\psi_{e,!}$. By Theorem \ref{p:!ave}, we have
\begin{equation}\label{eq:conv!}
\on{Hom}_{D(B,\chi\backslash G/B,\chi)}(\cF,\psi_{e,!}\star \cF)\neq 0.
\end{equation}

\noindent Consider the inversion map $\iota: G\to G,\; g\mapsto g^{-1}$, and let $\bD$ be the functor of Verdier duality. The composition
\[
\bD\circ \iota^!: D(G)^c\to (D(G)^c)^{\on{op}}
\]

\noindent descends to an endofunctor
\[
\bD\circ \iota^!: D(B,\chi\backslash G/B,\chi)^{\on{coh}}\to  (D(B,\chi\backslash G/B,\chi)^{\on{coh}})^{\on{op}},
\]

\noindent where the former is the category of coherent bi-$(B,\chi)$-equivariant D-modules on $G$ (that is, D-modules that pull back to something compact in $D(G)$).

Consider the sheaf $\bD(\iota^!(\psi_{e,!}))$. By \cite[Prop. 22.10.1]{frenkel2006local}, the functor
\[
\bD(\iota^!(\psi_{e,!}))\star -: D(B,\chi\backslash G/B,\chi)\to D(B,\chi\backslash G/B,\chi)
\]

\noindent is left adjoint to $\psi_{e,!}\star -$. Note that $\bD(\iota^!(\psi_{e,!}))$ coincides with $\psi_{e,*}$ up to a cohomological shift. By (\ref{eq:conv!}), we have
\[
\on{Hom}_{D(B,\chi\backslash G/B,\chi)}(\psi_{e,*}\star \cF,\cF)\neq 0.
\]

\noindent Since the functor $\psi_{e,*}\star -$ coincides with the composition
\[
D(B,\chi\backslash G/B,\chi)\xrightarrow{\on{Av}_*^{\psi_e}} D(B,\chi\backslash G/U^-,\psi_e)\xrightarrow{\on{Av}_!^{B,\chi}} D(B,\chi\backslash G/B,\chi),
\]

\noindent it follows that $\on{Av}_*^{\psi_e}(\cF)\neq 0$.

\end{proof}

\subsection{Vanishing of generalized Whittaker coefficients}
In this subsection, we prove the other direction of Theorem \ref{t:van}. In fact, we prove a slight strengthening:
\begin{thm}\label{t:van1}
Let $\cF\in D(G/B,\chi) \underset{\sH_{G,\chi}}{\otimes} \sH_{G,\chi,Z}$. If $e\notin Z$, then $\on{Av}_*^{\psi_e}(F)=0$. Moreover, if $\cF\in D_Z(G/B,\chi)$ is holonomic, then $\on{Av}_*^{\psi_e}(F)=0$.
\end{thm}

\subsubsection{}
Throughout this subsection, we fix $e\in\bO$ not contained in $Z$.  As in \S\ref{s:3.1r}, we may find a dominant and regular weight $\lambda=\lambda_{\chi}\in\ft^*$ such that we get Beilinson-Bernstein localization equivalences:
\[
D(B,\chi\backslash G)\simeq \fg\on{-mod}_{\lambda},\;\; D(B,\chi\backslash G/B,\chi)\simeq \on{HC}_{\lambda,\lambda}.
\]

\subsubsection{} Theorem \ref{t:van1} will be a consequence of Theorem \ref{t:van1'} below. The proof of the latter follows closely that of \cite[Thm. 3.2.4.1]{faergeman2022non}.

\begin{thm}\label{t:van1'}
The category
\[
D(U^-,\psi_e\backslash G/B,\chi)\underset{\sH_{G,\chi}}{\otimes}\sH_{G,\chi,Z}
\]

\noindent is zero.
\end{thm}

\subsubsection{} We have the following standard result:

\begin{lem}\label{l:Losev}
Let $M\in\on{HC}_{\lambda,\lambda}^{\heartsuit}$ be simple. Then $\on{SS}(U(\fg)_{\lambda}/\on{Ann}_{U(\fg)_{\lambda}}(M))=\on{SS}(M).$
\end{lem}

\begin{proof}[Proof of Theorem \ref{t:van1'}]

\step 

We need to prove that for any $\cF\in \sH_{G,\chi,Z}$ and $\cG\in D(U^-,\psi_e\backslash G/B,\chi)$, we have:
\[
\cG\star\cF=0\in D(U^-,\psi_e\backslash G/B,\chi).
\]

\noindent Let $M\in \on{HC}_{\lambda,\lambda}$ be the image of $\cF$ under the Beilinson-Bernstein equivalence. Let $A:=U(\fg)_{\lambda}/\on{Ann}_{U(\fg)_{\lambda}}(M)$. As in \cite[\S 3.2]{faergeman2022non}, we have a strong right $G$-action on $A\on{-mod}$ such that the forgetful map 
\[
A\on{-mod}\to \fg\on{-mod}_{\lambda}
\]

\noindent is $G$-equivariant. Note that the map 
\[
A\on{-mod}\to \fg\on{-mod}_{\lambda}\simeq D(B,\chi\backslash G)\xrightarrow{\cG\star-} D(U^-,\psi_e\backslash G)
\]

\noindent is given by an object of 
\[
\on{Hom}_{D(G)}(A\on{-mod},D(U^-,\psi_e\backslash G))\simeq A\on{-mod}^{U^-,\psi_e}.
\]

\noindent Thus, it suffices to show that the category $A\on{-mod}^{U^-,\psi_e}$ is zero.

\step 

Note that any $A$-module has singular support contained in $\on{SS}(A)$. The latter coincides with $\on{SS}(M)\subset Z$ by Lemma \ref{l:Losev}. Thus, it suffices to show that any non-zero object $N\in A\on{-mod}^{U^-,\psi_e}$ satisfies $\bO\cap\on{SS}(N)\neq\emptyset$, where $\bO$ is the orbit containing $e$.

We may assume that $N$ is coherent. Suppose for contradiction that $\bO \cap \on{SS}(N)=\emptyset$, and let $J$ be the primitive ideal given by the annihilator of $N$ in $U(\fg)_{\lambda}$. Since, $U(\fg)_{\lambda}/J$ is a quotient of $A$, we have $\bO\cap \on{SS}(U(\fg)_{\lambda}/J)=\emptyset$. However, using Skryabin's equivalence between Whittaker modules and modules for the $\sW$-algebra corresponding to $e$ (see \cite[App. A]{premet2002special}, \cite[§6]{gan2002quantization}), we see by \cite[Thm 1.2.2 (ii), (v)]{losev2010quantized} that $\bO\subset \on{SS}(U(\fg)_{\lambda}/J)$. This is a contradiction, and so we conclude that $\bO \cap \on{SS}(N)\neq\emptyset$.

\end{proof}

\begin{proof}[Proof of Theorem \ref{t:van1}]

It is clear that Theorem \ref{t:van1'} implies the first part of Theorem \ref{t:van1}. Thus, it remains to show that if $\cF\in D_Z(G/B,\chi)$ is holonomic, then $\on{Av}_*^{\psi_e}(F)=0$. Since $\cF$ is of finite length, we may assume that $\cF$ is simple. Let $M\in \fg\on{-mod}_{\lambda}$ be the corresponding module, and let $A=U(\fg)_{\lambda}/\on{Ann}_{U(\fg)_{\lambda}}(M)$. If we can prove that $\on{SS}(A)\subset Z$, then the proof of Theorem \ref{t:van1'} applies to show that $\on{Av}_*^{\psi_e}(F)=0$. More precisely, we will show that if $\bO'\subset Z$ is a maximal orbit intersecting $\on{SS}(M)$, then $\on{SS}(A)=\overline{\bO'}$.

First, it follows from \cite{losev2017bernstein}[Lemma 5.2] that the smooth locus of $\on{SS}(M)\cap \bO'$ is isotropic. Hence $\on{dim} \on{SS}(M)=\frac{1}{2}\on{dim}\bO'$. On the other hand, $\on{dim}\on{SS}(A)\leq 2\cdot\on{dim} \on{SS}(M)= \on{dim}\bO'$, cf. \cite{krause2000growth}[Thm. 9.11]. By a classical result of Joseph \cite{joseph1985associated}, $\on{SS}(A)$ is irreducible, and hence is the closure of a nilpotent orbit. It follows that $\on{SS}(A)=\overline{\bO'}$.

\end{proof}

\subsubsection{}

We record the following results for future use.
\begin{lem}\label{l:boundonann}
Let $M\in \fg\on{-mod}_{\lambda}^{\heartsuit}$ be coherent and lie in the heart of the t-structure. Moreover, suppose $M$ lies in the image of the fully faithful embedding:
\[
D(G/B,\chi)\underset{\sH_{G,\chi}}{\otimes} \sH_{G,\chi,\overline{\bO}} \simeq \fg\on{-mod}_{\lambda}\underset{\sH_{G,\chi}}{\otimes} \sH_{G,\chi,\overline{\bO}}\subset \fg\on{-mod}_{\lambda}.
\]

\noindent Let $A=U(\fg)_{\lambda}/\on{Ann}_{U(\fg)_{\lambda}}(M)$. Then $\on{SS}(A)\subset \overline{\bO}$.
\end{lem}

\begin{proof}

\step
Let $N\in \sH_{G,\chi,\overline{\bO}}\subset \on{HC}_{\lambda,\lambda}$. Let $A'=U(\fg)_{\lambda}/\on{Ann}_{U(\fg)_{\lambda}}(N)$. According to Lemma \ref{l:Losev}, we have $\on{SS}(A')\subset \overline{\bO}$. For any $\cG\in D(G/B,\chi)$, consider the image, denoted $\cG\star N$, of $\cG\boxtimes N$ under the convolution functor:
\[
D(G/B,\chi)\otimes \sH_{G,\chi,\overline{\bO}}\to D(G/B,\chi)\simeq \fg\on{-mod}_{\lambda}.
\]
\noindent We claim that for all $i\in \bZ$, we have $\on{SS}(U(\fg)_{\lambda}/\on{Ann}_{U(\fg)_{\lambda}}(H^i(\cG\star N)))\subset \overline{\bO}$. Indeed, the above functor factors as:

\[
D(G/B,\chi)\otimes \sH_{G,\chi,\overline{\bO}}\to D(G/B,\chi)\otimes A'\on{-mod}^{B,\chi}\to A'\on{-mod}\xrightarrow{\on{oblv}} \fg_{\lambda}\on{-mod}_{\lambda}.
\]

\noindent Here, the second functor is induced from the strong $G$-action on $A'\on{-mod}$.

\step

The lemma now follows from writing $M$ as a finite colimit of objects $\cG_i\star N_i$ as above.

\end{proof}

\begin{cor}\label{c:boundonann}
Let $M\in \fg\on{-mod}_{\lambda}^{\heartsuit}$ be coherent and lie in the heart of the t-structure. Then $M$ lies in the full subcategory
\[
\fg\on{-mod}_{\lambda}\underset{\sH_{G,\chi}}{\otimes} \sH_{G,\chi,\overline{\bO}}\subset  \fg\on{-mod}_{\lambda}
\]

\noindent if and only if $\on{SS}(U(\fg)_{\lambda}/\on{Ann}_{U(\fg)_{\lambda}}(M))\subset \overline{\bO}$.
\end{cor}

\begin{proof}
By Lemma \ref{l:boundonann}, it suffices to prove that any such $M$ with $\on{SS}(U(\fg)_{\lambda}/\on{Ann}_{U(\fg)_{\lambda}}(M))\subset \overline{\bO}$ lies in the subcategory $\fg\on{-mod}_{\lambda}\underset{\sH_{G,\chi}}{\otimes} \sH_{G,\chi,\overline{\bO}}$.

By Theorem \ref{t:van2}, it suffices to show that for any $e\notin \overline{\bO}$, $M$ vanishes under the functor
\[
\fg\on{-mod}_{\lambda}\simeq D(G/B,\chi)\xrightarrow{\on{Av}_*^{\psi_e}} D(G/U^-,\psi_e).
\]

\noindent However, this follows from the proof of Theorem \ref{t:van1}.

\end{proof}

\section{Singular support for $G$-categories}\label{s:4}
In this section, we introduce the notion of singular support for DG-categories equipped with a strong $G$-action. Our main result is Theorem \ref{t:main}, which characterizes the singular support of nilpotent categories in terms of the vanishing of its generalized Whittaker models.

\subsection{Definition and examples}

\subsubsection{} Let $\sC$ be a $G$-category that is dualizable as a plain DG-category. We remind that this is equivalent to $\sC$ being dualizable as a $G$-category. Thus, we will simply refer to such $G$-categories as \emph{dualizable $G$-categories}.

\subsubsection{} Since $D(G)$ is dualizable and self-dual, the action map
\[
\on{act}_{\sC}: D(G)\otimes \sC\to \sC
\]

\noindent gives rise to a coaction map
\[
\on{coact}_{\sC}: \sC\to D(G)\otimes \sC.
\]

\noindent Explicitly, for $\cF\in D(G)$, write $\lambda_{\cF}$ for the continuous functor
\[
\lambda_{\cF}: D(G)\to \on{Vect},\;\; \cG\mapsto C_{\on{dR}}(G,\cF\overset{!}{\otimes} \cG).
\]

\noindent Then we have an isomorphism of functors
\[
(\lambda_{\cF}\otimes \on{id}_{\sC})\circ \on{coact}_{\sC}\simeq \on{act}_{\sC}(\cF\boxtimes -): \sC\to \sC.
\]

\begin{rem}\label{r:omitnotation}
We will often write $\on{act}$ (resp. $\on{coact}$) instead of $\on{act}_{\sC}$ (resp. $\on{coact}_{\sC}$) when there is no ambiguity about which category is acted on.
\end{rem}

\subsubsection{} Let $Z$ be a closed conical $G$-stable subset of $\fg^*$. Recall that $D_Z(G)\subset D(G)$ denotes the category of D-modules on $G$ whose singular support is contained in $\fg^*\times Z\subset T^*G$.

We make the following definition:
\begin{defin}\label{d:ss}
We say that $\sC$ has singular support contained in $Z$ if the coaction map factors as:
\[\begin{tikzcd}
	\sC && {D(G)\otimes \sC} \\
	&& {D_Z(G)\otimes \sC.}
	\arrow["{\mathrm{coact}}", from=1-1, to=1-3]
	\arrow[dashed, from=1-1, to=2-3]
	\arrow[hook', from=2-3, to=1-3]
\end{tikzcd}\]
\end{defin}

\begin{rem}\label{r:ff}
Since $\sC$ is dualizable, the map
\[
D_Z(G)\otimes \sC\to D(G)\otimes \sC
\]

\noindent is fully faithful.

\end{rem}

\begin{rem}\label{r:MC}
We have the following equivalent but perhaps more intuitive definition of what it means for the singular support of $\sC$ to be contained in $Z$. From the action map, we obtain the functor
\[
\on{MC}_{\sC}: \on{End}_{\on{DGCat}_{\on{cont}}}(\sC)\simeq \sC\otimes \sC^{\vee}\to D(G).
\]

\noindent The singular support of $\sC$ is contained in $Z$ if and only if the image of $\on{MC}_{\sC}$ lands in $D_Z(G)$. 

As described in the introduction, given $\phi\in \on{End}(\sC)$, one ought to think of $\on{MC}_{\sC}(\phi)$ as the sheaf of matrix coefficients corresponding to $\phi$.
\end{rem}

\subsubsection{} We let $\on{SS}(\sC)\subset \fg^*$ be the smallest closed conical $G$-stable subset such that the singular support of $\sC$ is contained in $\on{SS}(\sC)$. If $\on{SS}(\sC)\subset \cN$ lies in the nilpotent cone of $\fg^*$, we say that $\sC$ is \emph{nilpotent}.
\begin{rem}
One might also reasonably use the term $G$-category with \emph{restricted variation}\footnote{Especially in light of Theorem \ref{t:nilp} below.} instead of nilpotent $G$-category, in line with the philosophy of \cite{arinkin2020stack}. Indeed, recall that the results of \emph{loc.cit} show that the automorphic sheaves with nilpotent singular exactly match the local systems with restricted variation under geometric Langlands.
\end{rem}

\subsubsection{Examples}\label{e:exples}
\begin{itemize}

     \item $\on{SS}(D(G))=\fg^*$.

     \item Let $\on{Cst}\subset D(G)$ be the subcategory generated by the constant sheaf under colimits. Then $\on{Cst}$ is naturally a $G$-category with singular support equal to $0$. 
     
    \item Consider the category $\fg\on{-mod}_{\lambda}$ of $U(\fg)$-modules with generalized central character $\lambda\in \ft^*$. Then $\on{SS}(\fg\on{-mod}_{\lambda})=\cN$.

    \item Let $Z\subset \cN$ be a closed $G$-stable subset. It is not difficult to see that the singular support of $D(G/B,\chi)\underset{\sH_{G,\chi}}{\otimes}\sH_{G,\chi,Z}$, considered as a $G$-category by the action on the first factor is contained in $Z$. If $\chi=0$, then $\on{SS}(D(G/B)\underset{\sH_{G}}{\otimes}\sH_{G,Z})=Z^{\on{sp}}$, where $Z^{\on{sp}}$ is the union of the closures of all special nilpotent orbits contained in $Z$ (see §\ref{s:special}).

\end{itemize}

\subsubsection{} One of the most natural examples of a $G$-category is $D(X)$, where $X$ is a $G$-variety. Consider the moment map
\[
\mu: T^*X\to \fg^*.
\]

\noindent In general, one might expect that if $\mu$ factors through some $Z\subset \fg^*$, then $\on{SS}(D(X))\subset Z$. We do not know whether this is true in general. The sublety comes from the fact while the coaction map (i.e., pullback via the action map):
\[
D(X)\to D(G)\otimes D(X)
\]

\noindent lands in $D_{Z\times T^*X}(G\times X)$, it is not clear that it lands in $D_Z(G)\otimes D(X)$ (the latter usually being a strictly smaller category than the former). Using the results of Section 3, we are able to obtain the following partial result, however:
\begin{prop}\label{p:geomexamp}
Suppose $\mu: T^*X\to \fg^*$ factors through the nilpotent cone. Let $Z=\overline{\on{Im}(\mu)}$ be the closure of the set-theoretic image. Then $\on{SS}(D(X))=Z$.
\end{prop}

\begin{proof}
Denote by $\on{act}: G\times X\to X$ the action map.

\step Let us show that $\on{SS}(D(X))\subset Z$. We need to show that for any $\cF\in D(X)$, we have 
\[
\on{act}^!(\cF)\in D_Z(G)\otimes D(X).
\]

First we reduce to the case when $\cF\in D(B\backslash X)$. Note that the assumption that $\mu: T^*X\to \fg^*$ factors through the nilpotent cone implies that the moment map $T^*(N\backslash X)\to \ft^*$ is trivial. In particular, if $i_x: \on{Spec}(k)\to N\backslash X$ is a $k$-point with automorphism group $\on{Aut}_x$, the action of $T$ on $N\backslash X$ restricts to an action
\[
T\times \bB \on{Aut}_x\to \bB \on{Aut}_x.
\]

\noindent We claim that the image of the forgetful functor
\[
\on{oblv}^T: D(B\backslash X)\to D(N\backslash X)
\]

\noindent generates the target under colimits. To show this, we need to show that if $\cF\in D(N\backslash X)$ is killed under the averaging functor
\[
\on{oblv}^T\circ \on{Av}_*^T: D(N\backslash X)\to D(N\backslash X),
\]

\noindent then $\cF=0$. Suppose $\cF$ is such a sheaf. By \cite[Lemma 9.2.8]{arinkin2020stack}, it suffices to show that for any field-valued point $i_x: \on{Spec}(K)\to N\backslash X$, the image of $\cF$ under the composition
\[
D(N\backslash X)\to \on{Vect}_K\underset{\on{Vect}_k}{\otimes} D(N\backslash X)\simeq D_{K/k}(N\backslash X)\xrightarrow{i_x^!} \on{Vect}_K
\]

\noindent vanishes. Here, the first functor is base-change to $K$, and $D_{K/k}(N\backslash X)$ denotes the category of D-modules (over $K$) of $\on{Spec}(K)\underset{\on{Spec}(k)}{\times} N\backslash X$. The argument we provide below is evidently stable under base-change, and so we assume that $K=k$.\footnote{Concretely, all we need is that the functor $\on{oblv}^T\circ \on{Av}_*^T$ commutes with base-change to $K$.}

We similarly write $i_x$ for the corresponding map of stacks
\[
i_x: \bB \on{Aut}_x\to N\backslash X.
\]

\noindent By the above, both stacks are equipped with a $T$-action, and $i_x$ is $T$-equivariant. Thus, we have
\[
\on{oblv}^T\circ \on{Av}_*^T\circ i_x^!(\cF)=i_x^!\circ \on{oblv}^T\circ \on{Av}_*^T(\cF)=0.
\]

\noindent Since every object in $D(\bB \on{Aut}_x)$ is $T$-monodromic (the category $D(\bB \on{Aut}_x)$ is generated under colimits by the dualizing sheaf, see \cite[§2.4.4]{beraldo2021tempered}), it follows that $i_x^!(\cF)=0$. This shows that $\cF=0$, as desired.

By \cite{ben2020highest}, we have an equivalence of categories
\[
D(G/N)\underset{D(N\backslash G/N)}{\otimes} D(N\backslash X)\xrightarrow{\simeq} D(X),
\]

\noindent where the functor from left to right is given by convolution. Since every object of $D(N\backslash X)$ is $B$-monodromic, this implies that in fact we have an equivalence
\[
D(G/B)\underset{D(B\backslash G/B)}{\otimes} D(B\backslash X)\xrightarrow{\simeq} D(X).
\]

\noindent Since convolution with $D(G/B)$ sends $D_Z(B\backslash G)$ to $D_Z(G)$, the reduction step follows.

\vspace{2mm}

\step

\noindent Let $\cF\in D(B\backslash X)$. We may assume that $\cF$ is coherent. By Theorem \ref{t:van}, we need to show that $\on{act}^!(\cF)$ lies in
\[
\big(\underset{e\in\bO\not\subset Z}{\bigcap}\on{Ker}(\on{Av}_*^{\psi_e}: D(B\backslash G)\to D(U^-,\psi_e\backslash G))\big)\otimes D(X)\simeq 
\underset{e\in\bO\not\subset Z}{\bigcap}\on{Ker}(\on{Av}_*^{\psi_e}\otimes \on{id}_{D(X)})\subset D(B\backslash G)\otimes D(X).
\]

\noindent Since $\cG:=(\on{Av}_*^{\psi_e}\otimes \on{id}_{D(X)})(\on{act}^!(\cF))$ is coherent,\footnote{We are using that $\on{Av}_*^{\psi_e}$ preserves compact objects here.} it suffices to show that for any $k$-point $i_x: \on{Spec}(k)\to X$, we have $(\on{id}\otimes i_x^!)(\cG)=0$. Let $\bO_x:=G\cdot x\subset X$ be the $G$-orbit through $x$. Denote by $\on{act}_x$ the map $B\backslash G\to B\backslash \bO_x,\; g\mapsto gx$. Note:
\[
(\on{id}\otimes i_x^!)(\cG)=\on{Av}_*^{\psi_e}(\on{act}_x^!(\cF)).
\]

\noindent The moment map $T^*(\bO_x)\to \fg^*$ also factors through $Z$. By Lemma \ref{l:Bourbaki} below, $\bO_x$ is a partial flag variety, and thus the assertion follows from Theorem \ref{t:van1}.

\step Finally, let us show that $\on{SS}(D(X))=Z$. Suppose for contradiction that $\on{SS}(D(X))=:Z'$ is a proper subset of $Z$. This means that for all $\cF\in D(X)$, we have
\[
\on{act}^!(\cF)\in D_{Z'}(G)\otimes D(X).
\]

\noindent Since $\on{act}$ is smooth, the image of $\on{SS}(\cF)$ under $\mu: T^*X\to \fg^*$ lands in $Z'$. Since this holds for all $\cF\in D(X)$, we see that $\mu$ must factor through $Z'$, a contradiction.
\end{proof}

\begin{lem}\label{l:Bourbaki}
Let $X$ be a homogeneous $G$-variety such that the moment map $\mu: T^*X\to \fg^*$ factors through the nilpotent cone. Then $X$ is a partial flag variety.

\end{lem}

\begin{proof}

Write $X=G/H$ for some subgroup $H$ with Lie algebra $\fh$. We need to show that $\fh$ is a parabolic subalgebra. In other words, we need to show that if a subalgebra $\fh\subset \fg$ satisfies that its orthogonal complement $\fh^{\perp}$ with respect to the Killing form consists of nilpotent elements, then it is parabolic. However, this is exactly the statement of Bourbaki's \cite{burbaki1975elements}[\S 10, Théorème 1].
    
\end{proof}

\subsubsection{} We end this subsection with the following useful proposition.

\begin{prop}\label{p:heckedescrip}
Let $\sC$ be a dualizable $G$-category. Then the embedding
\[
\sH_{G,\chi,\on{SS}(\sC)}\underset{\sH_{G,\chi}}{\otimes} \sC^{B,\chi}\into \sC^{B,\chi}
\]

\noindent is an equivalence.
\end{prop}

\begin{proof}
Note that the functor
\[
D(G/B,\chi)\to \on{Hom}_{\sH_{G,\chi}}(D(B,\chi\backslash G),\sH_{G,\chi}), \;\; \cG\mapsto (\cF\mapsto \cF\star\cG)
\]

\noindent is an equivalence. As such, it suffices to show that for every $\cG\in D(G/B,\chi)$, the composition
\[\begin{tikzcd}
	{\sC^{B,\chi}} && {D_{\on{SS}(\sC)}(B,\chi\backslash G)\otimes \sC} \\
	\\
	&& {\sH_{G,\chi}\otimes \sC}
	\arrow["{\mathrm{coact}^{B,\chi}}", from=1-1, to=1-3]
	\arrow["{(-\star\mathcal{G})\otimes\mathrm{id}}", from=1-3, to=3-3]
\end{tikzcd}\]

\noindent lands in $\sH_{G,\chi,\on{SS}(\sC)}\otimes\sC$. Here, $\on{coact}^{B,\chi}$ is the functor given by taking $(B,\chi)$-invariance with respect to the coaction functor. However, the further composition with the forgetful functor $\sH_{G,\chi}\otimes\sC\to D(B,\chi\backslash G)\otimes \sC$ coincides with the functor
\[\begin{tikzcd}
	{\sC^{B,\chi}} && {D_{\mathrm{SS}(\sC)}(B,\chi\backslash G)\otimes \sC} \\
	\\
	&& {D_{\mathrm{SS}(\sC)}(B,\chi\backslash G)\otimes \sC.}
	\arrow["{\mathrm{coact}^{B,\chi}}", from=1-1, to=1-3]
	\arrow["{\mathrm{id}\otimes(\mathcal{G}\star-)}", from=1-3, to=3-3]
\end{tikzcd}\]

\end{proof}

\subsection{Definition of parabolic induction and restriction}

\subsubsection{} Fix a parabolic subgroup $P\subset G$ with Lie algebra $\fp$. Denote by $U_P$ its unipotent radical and $M$ its Levi quotient. Consider the correspondence
\[\begin{tikzcd}
	& P \\
	G && M.
	\arrow["p"', from=1-2, to=2-1]
	\arrow["q", from=1-2, to=2-3]
\end{tikzcd}\]

\begin{defin}
For a $G$-category $\sC$, we let $\on{Res}^G_M(\sC)=\sC^{U_P}$ be the $U_P$-invariant subcategory of $\sC$, considered as an $M$-category. Conversely, for an $M$-category $\sD$, we let $\on{Ind}_M^G(\sD)=D(G)\underset{D(P)}{\otimes} \sD$.
\end{defin} 

\subsubsection{}\label{s:compose} The two functors $(\on{Ind}_M^G,\on{Res}^G_M)$ define an adjunction pair of $2$-categories
\[\begin{tikzcd}
	{D(M)\textrm{\textbf{-mod}}} && {D(G)\textrm{\textbf{-mod}}.}
	\arrow["{\mathrm{Ind}_M^G}", shift left, from=1-1, to=1-3]
	\arrow["{\mathrm{Res}_M^G}", shift left, from=1-3, to=1-1]
\end{tikzcd}\]

\noindent Note that the functors compose. That is, if $M'$ is a Levi subgroup of $M$, then
\[
\on{Res}^M_{M'}\circ\on{Res}^G_M\simeq \on{Res}^G_{M'},\;\;\; \on{Ind}^G_{M}\circ\on{Ind}^M_{M'}\simeq \on{Ind}^G_{M'}.
\]

\noindent To ease notation, we will usually write $\sC^{U_P}$ instead of $\on{Res}_M^G(\sC)$, remembering that it carries an $M$-action.

\subsection{Characterization of nilpotence}
In this subsection, we prove Theorem \ref{t:criterionnilp}, giving a characterization of nilpotent $G$-categories.

\subsubsection{Mellin transform}\label{s:Mellin} The Mellin transform (see e.g. \cite[App. A]{gannon2022classification}) provides a symmetric monoidal equivalence
\[
D(T)\simeq \on{QCoh}(\ft^*/X^{\bullet}(T)).
\]
\noindent Here, the monoidal structure on $D(T)$ is given by convolution, and the monoidal structure on $\on{QCoh}(\ft^*/X^{\bullet}(T))$ is given by tensor product.

Given a $K$-point of $\ft^*/X^{\bullet}(T)$ for some field $K/k$, we obtain a character sheaf on $T$.
\begin{lem}\label{Algebraskin}
Let $\lambda: \on{Spec}(K)\to \ft^*/X^{\bullet}(T)$, and let $\chi=\chi_{\lambda}$ be the corresponding character sheaf on $T$. If $K/k$ is a transcendental field extension and $\lambda$ does not factor through a $k$-point, then no coherent $D_T$-submodule of $\chi$ lies in the category $D_0(T)$ of lisse D-modules on $T$. In particular, $\chi\notin D_0(T)$.
\end{lem}

\begin{proof} Write $T=\bG_m^{\times r}$. Then $\chi$ is of the form $\underset{i}{\boxtimes}\chi_i$, and at least one of the $\chi_i$ is not base-changed from $k$. Thus, we assume that $T=\bG_m$ has rank $1$. In particular, we consider $\lambda$ as an element of $K$, transcendental over $k$. Recall that $D_{\bG_m}=k[t^{\pm 1},\partial_t]$. As a $D_{\bG_m}$-module, $\chi$ is given by
\[
\chi=\lbrace \underset{n\in \bZ}{\sum} a_nt^{\lambda + n}\rbrace ,
\]

\noindent where $a_n\in K$ and only finitely many summands are non-zero. Fix $g\in \chi$ non-zero. We need to show that the $D_{\bG_m}$-module $k[t^{\pm 1},\partial_t]\cdot g\subset \chi$ generated by $g$ is not finitely generated as a $k[t^{\pm 1}]$-module. First, note that for any given finitely generated $k[t^{\pm 1}]$-submodule $M\subset \chi$, we may find a finite-dimensional $k$-vector space $W\subset K$ such that
\begin{equation}\label{eq:boundedcoeffs}
M\subset W\otimes k[t^{\pm 1}]\cdot t^{\lambda}.
\end{equation}

\noindent In other words, we may find $W$ so that for any element $h=\underset{n\in \bZ}{\sum} a_nt^{\lambda + n}\in M$, all the coefficients $a_n$ lie in $W$.

We may assume that $g$ has the form $g=\overset{N}{\underset{n=0}{\sum}}a_nt^{\lambda + n}$ with $a_0\neq 0$. Note that for any $r\in \bN$, we have
\[
(t\partial_t)^r\cdot g= \overset{N}{\underset{n=0}{\sum}}a_n(\lambda+n)^r t^{\lambda + n}.
\]

\noindent In particular, the numbers $\lbrace a_0\lambda^r\rbrace $ all occur as coefficients of elements of $k[t^{\pm 1},\partial_t]\cdot g$. Since $\lambda$ is transcendental over $k$, we conclude that no finite-dimensional $W\subset K$ exists satisfying (\ref{eq:boundedcoeffs}). In particular, $k[t^{\pm 1},\partial_t]\cdot g$ is not finitely generated over $k[t^{\pm 1}]$.
\end{proof}

\subsubsection{} We say that a character sheaf $\chi$ on $T$ is transcendental if it comes from a $K$-point $\lambda: \on{Spec}(K)\to \ft^*/X^{\bullet}(T)$ with $K/k$ transcendental, and such that $\lambda$ does not factor through a $k$-point $\on{Spec}(k)\to \ft^*/X^{\bullet}(T)$.
\begin{lem}\label{l:nilp}
Let $\chi$ be a transcendental character sheaf on $T$, and let $\cF\in D(B,\chi\backslash G)$. If the singular support of $\cF$ is nilpotent as an object of $D(N\backslash G)$, then $\cF=0$.
\end{lem}

\begin{proof}
Consider the Bruhat decomposition of $G$:
\[
G=\underset{w\in W}{\bigsqcup} BwN.
\]

\noindent Let $j_w: BwN\into G$ be the inclusion. Choose $w\in W$ maximal such that $j_w^!(\cF)\neq 0$. In this case $j_w^!(\cF)$ has nilpotent singular support when considered as an object of $D(N\backslash BwN)$.

Consider the smooth map $m: B\times N\to BwN,\; (b,n)\mapsto bwn$. Then
\[
m^!\circ j_w^!(\cF)\in D(T,\chi\backslash T)\otimes D(N)\simeq D(B,\chi\backslash B)\otimes D(N)\simeq D(N).
\]

\noindent We claim that
\begin{equation}\label{eq:lisseandtranscendental}
m^!\circ j_w^!(\cF)\in (D(T,\chi\backslash T)\cap D_0(T))\otimes D(N).
\end{equation}

\noindent Indeed, this follows from the fact that, by nilpotence of $\on{SS}(\cF)$, the image of the map
\[
\on{SS}(\cF)\underset{N\backslash G}{\times} B\times N\to T^*(B\times N)
\]

\noindent induced by $j_w\circ m: B\times N\to N\backslash G$ intersects the image of the map
\[
T^*(T\times N)\underset{T\times N}{\times} B\times N\to T^*(B\times N)
\]

\noindent induced by the projection $B\to T$ in the zero section. By (\ref{eq:lisseandtranscendental}) and Lemma \ref{Algebraskin}, we have $m^!\circ j_w^!(\cF)=0$, and in particular $\cF=0$.
\end{proof}

\subsubsection{} For a $G$-category $\sC$, we may consider the full subcategory
\begin{equation}\label{eq:fullsubcat}
\underset{\lambda\in \ft^*/X^{\bullet}(k)}{\oplus} \sC^{(B,\chi_{\lambda})\on{-mon}}\subset \sC^N.
\end{equation}

\noindent Here, $\sC^{(B,\chi_{\lambda})\on{-mon}}$ denotes the full subcategory of $\sC^N$ given by the cocompletion of the image of the forgetful functor $\sC^{B,\chi_{\lambda}}\to \sC^N$. Note that in (\ref{eq:fullsubcat}), we only sum over $k$-points of $\ft^*/X^{\bullet}(T)$. In other words, we only consider non-transcendental character sheaves.\footnote{We remind that $k$ is algebraically closed.}
\begin{thm}\label{t:nilp}
Let $\sC$ be a dualizable $G$-category. Then $\sC$ is nilpotent if and only if 
\[
\underset{\lambda\in \ft^*/X^{\bullet}(k)}{\oplus} \sC^{(B,\chi_{\lambda})\on{-mon}}=\sC^N.
\]
\end{thm}

\begin{proof}
\emph{Step 1.} Suppose first that 
\[
\underset{\lambda\in \ft^*/X^{\bullet}(k)}{\oplus} \sC^{(B,\chi_{\lambda})\on{-mon}}=\sC^N.
\]

\noindent As in §\ref{s:chigen}, denote by $\sC^{(B,\chi_{\lambda})\on{-gen}}$ the $G$-subcategory
\[
D(G/B,\chi_{\lambda})\underset{D(B,\chi_{\lambda}\backslash G/B,\chi_{\lambda})}{\otimes} \sC^{B,\chi_{\lambda}}\into \sC.
\]

\noindent It follows from \cite{ben2020highest} that the image of the functor
\[
\underset{\lambda\in \ft^*/X^{\bullet}(k)}{\oplus}\sC^{(B,\chi_{\lambda})\on{-gen}}\to \sC
\]

\noindent generates the target under colimits.\footnote{In fact, we have $\sC\simeq \underset{\lambda\in \ft^*/W\ltimes X^{\bullet}(k)}{\oplus}\sC^{(B,\chi_{\lambda})\on{-gen}}$.}

We need to prove that $\on{SS}(\sC^{(B,\chi_{\lambda})\on{-gen}})\subset \cN$. Note that the coaction map
\[
\sC^{(B,\chi_{\lambda})\on{-gen}}\to D(G)\otimes \sC^{(B,\chi_{\lambda})\on{-gen}}
\]

\noindent lands in the full subcategory $D(G)^{(B,\chi)\on{-gen}}\otimes \sC^{(B,\chi_{\lambda})\on{-gen}}$. Thus, we need to show that
\[
D(G)^{(B,\chi)\on{-gen}}=D(G/B,\chi_{\lambda})\underset{D(B,\chi_{\lambda}\backslash G/B,\chi_{\lambda})}{\otimes} D(B,\chi_{\lambda}\backslash G)\into D(G)
\]

\noindent lands in $D_{\cN}(G)$. However, this follows from the fact that any sheaf in $D(B,\chi_{\lambda}\backslash G)$ has nilpotent singular support and the fact that the convolution map
\[
D(G/B,\chi_{\lambda})\otimes D(B,\chi_{\lambda}\backslash G)\to D(G)
\]

\noindent preserves having nilpotent singular support, cf. §\ref{s:functsing}.

\vspace{2mm}

\emph{Step 2.} Let us prove the converse assertion. Suppose $\on{SS}(\sC)\subset \cN$. Note that a quasi-coherent sheaf $\cF\in \on{QCoh}(\ft^*/X^{\bullet}(T))$ has zero-dimensional support if and only if the pullback of $\cF$ under a field-valued point $\lambda: \on{Spec}(K)\to \ft^*/X^{\bullet}(T)$ with $K/k$ transcendental and such that $\lambda$ does not factor through a $k$-point, vanishes.

Using the Mellin transform §\ref{s:Mellin}, we see that it suffices to show that for any transcendental character sheaf $\widetilde{\chi}$ on $T$, we have
\[
\sC^{B,\widetilde{\chi}}=0.
\]

\noindent From the coaction map
\[
\sC\to D(G)\otimes \sC,
\]

\noindent we obtain the map
\[
\on{coact}^{B,\widetilde{\chi}}_{\sC}: \sC^{B,\widetilde{\chi}}\to D(B,\widetilde{\chi}\backslash G)\otimes \sC
\]

\noindent by taking $(B,\widetilde{\chi})$-invariants. Since $\sC$ is dualizable, it suffices to show that for any $\phi\in \sC^{\vee}$, the composition
\[
(\on{id}\otimes \phi)\circ \on{coact}^{B,\widetilde{\chi}}_{\sC}: \sC^{B,\widetilde{\chi}}\to D(B,\widetilde{\chi}\backslash G)\otimes \sC\to D(B,\widetilde{\chi}\backslash G)
\]

\noindent vanishes. By assumption, any object of $D(B,\widetilde{\chi}\backslash G)$ lying in the image of the above map has nilpotent singular support as an object of $D(N\backslash G)$. We conclude by Lemma \ref{l:nilp}.
\end{proof}

\begin{cor}\label{c:nilp}
Let $\sC$ be a dualizable $G$-category, and let $M\subset G$ be a Levi subgroup. Then $\sC$ is nilpotent if and only if $\on{Res}^G_M(\sC)$ is nilpotent.
\end{cor}

\begin{proof}
This is an immediate consequence of Theorem \ref{t:nilp} and the fact that the parabolic restriction functors compose, cf. §\ref{s:compose}.
\end{proof}

\subsection{Singular support estimate of parabolic restriction}\label{s:jwchi}
In this subsection, we show that parabolic restriction plays well with singular support for nilpotent $G$-categories. Let $P\subset G$ be a standard parabolic with Levi subgroup $M$.

\subsubsection{} Consider the correspondence
\[\begin{tikzcd}
	& {\mathfrak{p}} \\
	{\mathfrak{g}} && {\mathfrak{m}.}
	\arrow["{q}"', from=1-2, to=2-1]
	\arrow["{p}", from=1-2, to=2-3]
\end{tikzcd}\]

\noindent For closed conical subsets $Z\subset \fg$ and $Z_M\subset \fm$, we define the subsets
\[
\on{Res}_M^G(Z):=p(q^{-1}(Z))\subset \fm,\;\; \on{Ind}_M^G(Z_M):=\on{Ad}_G(q(p^{-1}(Z_M)))\subset \fg.
\]

\subsubsection{} The following lemma follows from \cite[§7.1]{collingwood1993nilpotent}:
\begin{lem}
Let $Z_M\subset \cN_M$ be a closed conical $\on{Ad}_M$-stable subset of the nilpotent cone of $\fm$. Then $\on{Ind}_M^G(Z_M)$ is nilpotent and closed.
\end{lem}

\noindent  The singular support estimate of parabolic restriction relies on the following result:
\begin{lem}\label{l:pullb}
Denote by $B_M=B\cap M$ the Borel subgroup of $M$, and let $Z\subset \cN$ be a closed conical $G$-stable subset of the nilpotent cone. Denote by $i$ the inclusion $M\into G$. For $\chi=\chi_{\lambda}$ a character sheaf on $T$ corresponding to $\lambda\in \ft^*/X^{\bullet}(k)$, $i^!$ restricts to functor
\[
i^!: \sH_{G,\chi,Z}\underset{\sH_{G,\chi}}{\otimes}D(B,\chi\backslash G)\to \sH_{M,\chi,\on{Res}_M^G(Z)}\underset{\sH_{M,\chi}}{\otimes}D(B_M,\chi\backslash M).
\]
\end{lem}

\begin{rem}
Note that $i^!: D(B,\chi\backslash G)\to D(B_M,\chi\backslash M)$ is tautologically isomorphic to the functor obtained by pull-push along the correspondence $M\leftarrow P\to G$.
\end{rem}

\begin{rem}
It is in general false that restriction along a closed embedding preserves singular support. See \cite[\S 2.2.1]{nadler2022whittaker} for a basic example.
\end{rem}

\begin{proof}
\step Writing $D(B,\chi\backslash G/B,\chi)\simeq D(B_M,\chi\backslash (U_P\backslash G)/B,\chi)$, we have an action
\begin{equation}\label{eq:action}
\sH_{M,\chi}\curvearrowright \sH_{G,\chi,Z}.
\end{equation}

\noindent We begin by proving that the embedding
\begin{equation}\label{eq:emb}
\sH_{M,\chi,\on{Res}^G_M(Z)}\underset{\sH_{M,\chi}}{\otimes} \sH_{G,\chi,Z} \into \sH_{G,\chi,Z}
\end{equation}

\noindent is an equivalence. The functor
\[
\sH_{G,\chi}\to\on{Hom}_{\sH_{M,\chi}}(\sH_{G\chi}, \sH_{M,\chi}),\;\; \cG\mapsto (\cF\mapsto i^!(\cF\star\cG))
\]

\noindent is an equivalence. Thus, it suffices to show that for all $\cF\in \sH_{G,\chi,Z}, \sG\in \sH_{G,\chi}$, we have:
\[
i^!(\cF\star \cG)\in \sH_{M,\chi,\on{Res}_M^G(Z)}.
\]

\noindent Since convolution preserves having singular support in $Z$, it in turn suffices to show that for any $\cF\in \sH_{G,\chi,Z}$, we have
\begin{equation}\label{eq:embhecke}
i^!(\cF)\in \sH_{M,\chi,\on{Res}_M^G(Z)}.
\end{equation}

\noindent Let $w_0\in W$ be the longest element and consider the corresponding open embedding $j_{w_0}: Bw_0B\into G$. Consider the character sheaf $w_0(\chi)=(w_0)_{*,\on{dR}}(\chi)$ on $T$. Let $m: B\times w_0\times B\to Bw_0B$ be the multiplication map. Note that $\chi\boxtimes w_0(\chi)\in D(B\times w\times B)$ descends to a sheaf $\phi\in D(B,\chi\backslash Bw_0B/B,w_0(\chi))$. We let
\[
j_{w_0,\chi,!}:=(j_{w_0})_!(\phi)\in D(B,\chi\backslash G/B,w_0(\chi)).
\]

\noindent Consider the composition
\begin{equation}\label{eq:comp2}
\sH_{M,\chi}\xrightarrow{i_!} \sH_{G,\chi}\xrightarrow{-\star j_{w_0,\chi,!}} D(B,\chi\backslash G/B,w_0(\chi)).
\end{equation}

\noindent Since the functor $-\star j_{w_0,\chi,!}$ is invertible (\cite[Lemma 3.5]{lusztig2020endoscopy}), it suffices to show that the right adjoint to (\ref{eq:comp2}) sends $\cF_0:=\cF\star j_{w_0,\chi,!}$ to an object of $\sH_{M,\chi,\on{Res}_M^G(Z)}$.

Let $j$ denote the open embedding
\[
Pw_0B\into G.
\]

\noindent Then the right adjoint to (\ref{eq:comp2}) is given by 
\begin{equation}\label{eq:rightadj}
D(B,\chi\backslash G/B,w_0(\chi))\xrightarrow{j^!} D(B,\chi\backslash Pw_0B/B,w_0(\chi))\simeq D(B,\chi\backslash P/B,\chi)\simeq \sH_{M,\chi}.
\end{equation}

\noindent Here, the first equivalence is the inverse to
\[
D(B,\chi\backslash P/B,\chi)\to D(B,\chi\backslash PwB/B,w_0(\chi)),\;\; \cG\mapsto \cG\star j_{w_0,\chi,!}.
\]

\noindent Since $j^!$ preserves singular support, being the pullback functor along an open embedding, the image of $\cF_0$ under (\ref{eq:rightadj}) has singular support contained in $\on{Res}_M^G(Z)\subset \fm^*$.

\step
Let us finish the proof. Since the functor
\[
i^!: \sH_{G,\chi,Z}\underset{\sH_{G,\chi}}{\otimes}D(B,\chi\backslash G)\to \sH_{M,\chi,\on{Res}_M^G(Z)}\underset{\sH_{M,\chi}}{\otimes}D(B_M,\chi\backslash M)
\]

\noindent is $\sH_{M,\chi}$-linear, the proof follows from the embedding (\ref{eq:emb}) being an equivalence.
\end{proof}

\subsubsection{} We need a converse statement to the above lemma.
\begin{lem}\label{l:converseprop}
Let $\cF\in \sH_{G,\chi}$ and write $Z:=\on{SS}(\cF)$. Let $Z_M\subset \fm^*$ be a closed conical Ad-invariant subset. If $\cF$ lies in the full subcategory
\[
\sH_{M,\chi,Z_M}\underset{\sH_{M,\chi}}{\otimes} \sH_{G,\chi} \subset \sH_{G,\chi},
\]

\noindent then $\on{Res}_M^G(Z)\subset Z_M$.

\end{lem}

\begin{proof}
Suppose for contradiction that this is not the case. Then there exists $w'\in W$ such that the map
\[
Z\underset{\cN/G}{\times}T^*(B\backslash G/U_P)\underset{B\backslash G/U_P}{\times}B\backslash Pw'B/U_P\to T^*(B\backslash G/U_P)\to \fm^* 
\]

\noindent does not factor through $Z_M$. Here, the last map is the moment map of the action $B\backslash G/U_P\curvearrowleft M$.

Let $W_P\subset W$ be the Weyl group of $M$, considered as a subgroup of $W$. Let $w$ be the unique longest element in $W_Pw'$. Note that
\[
Pw'B=PwB.
\]

\noindent Let $i_w$ be the locally closed embedding
\[
i_w: PwB\into G.
\]

\noindent To arrive at a contradiction, it suffices to show that the image of $\cF$ under the $\sH_{M,\chi}$-equivariant map
\[
\sH_{G,\chi}\xrightarrow{i_w^!} D(B,\chi\backslash PwB/B,\chi)
\]

\noindent does not land in 
\[
\sH_{M,\chi,Z_M}\underset{\sH_{M,\chi}}{\otimes} D(B,\chi\backslash PwB/B,\chi).
\]

\noindent As in the proof of Lemma \ref{l:pullb}, we may consider the equivalence
\[
-\star j_{w^{-1}w_0,\chi,!}: \sH_{G,\chi}\xrightarrow{\simeq} D(B,\chi\backslash G/B,w^{-1}w_0(\chi)).
\]

\noindent The functor restricts to an equivalence
\[
D(B,\chi\backslash PwB/B,\chi)\simeq D(B,\chi\backslash Pw_0B/B,w^{-1}w_0(\chi)).
\]

\noindent Thus, it suffices to show that
\[
i_w^!(\cF)\star j_{w^{-1}w_0,\chi,!}\in D(B,\chi\backslash Pw_0B/B,w^{-1}w_0(\chi))
\]

\noindent does not lie in 
\[
\sH_{M,\chi,Z_M}\underset{\sH_{M,\chi}}{\otimes} D(B,\chi\backslash Pw_0B/B,w^{-1}w_0(\chi)).
\]

\noindent Denote by $j$ the open embedding
\[
j: Pw_0B\into G.
\]

\noindent We have a commutative diagram
\[\begin{tikzcd}
	{\sH_{G,\chi}} && {D(B,\chi\backslash PwB/B,\chi)} \\
	\\
	{D(B,\chi\backslash G/B,w^{-1}w_0(\chi))} && {D(B,\chi\backslash Pw_0B/B,w^{-1}w_0(\chi)),}
	\arrow["{i_w^!}", from=1-1, to=1-3]
	\arrow["{-\star j_{w^{-1}w_0,\chi,!}}", from=1-1, to=3-1]
	\arrow["{-\star j_{w^{-1}w_0,\chi,!}}", from=1-3, to=3-3]
	\arrow["{j^!}", from=3-1, to=3-3]
\end{tikzcd}\]

\noindent where the vertical maps are equivalences. By construction, the image of $\on{SS}(\cF\star j_{w^{-1}w_0,\chi,!})$ under the map
\[
T^*(B\backslash G/U_P)\underset{B\backslash G/U_P}{\times}B\backslash Pw_0B/U_P\to \fm^*
\]

\noindent does not factor through $Z_M$. Since the functor $j^!$ preserves singular support, it follows that:
\[
i_w^!(\cF)\star j_{w^{-1}w_0,\chi,!}\notin 
 \sH_{M,\chi,Z_M}\underset{\sH_{M,\chi}}{\otimes} D(B,\chi\backslash Pw_0B/B,w^{-1}w_0(\chi)),
\]

\noindent as desired.

\end{proof}

\subsubsection{} The following is the main result of this subsection:
\begin{thm}\label{t:pres}
Let $\sC$ be a nilpotent $G$-category. Then $\on{SS}(\on{Res}^G_M(\sC))=\on{Res}^G_M(\on{SS}(\sC))$.
\end{thm}

\begin{proof}
\step 

First, we show that
\[
\on{SS}(\on{Res}^G_M(\sC))\subset\on{Res}^G_M(\on{SS}(\sC)).
\]

\noindent For any $\chi\in \ft^*/X^{\bullet}(T)(k)$, we have a commutative diagram
\[\begin{tikzcd}
	{D(B,\chi\backslash G)\otimes \sC} && {D(B_M,\chi\backslash M)\otimes \sC} \\
	\\
	{\sC^{B,\chi}=(\sC^{U_P})^{B_M,\chi}} && {D(B_M,\chi\backslash M)\otimes \sC^{U_P}.}
	\arrow["{\mathrm{coact}_{\sC^{U_P}}^{B_M,\chi}}", from=3-1, to=3-3]
	\arrow["{\mathrm{id}\otimes \mathrm{oblv}^{U_P}}"', from=3-3, to=1-3]
	\arrow["{\mathrm{coact}_{\sC}^{B,\chi}}", from=3-1, to=1-1]
	\arrow["{i^!\otimes \mathrm{id}_{\sC}}", from=1-1, to=1-3]
\end{tikzcd}\]

\noindent Here, $\on{coact}_{\sC}^{B,\chi}$ denote the functor of taking $(B,\chi)$-invariance for the $G$-linear functor $\on{coact}_{\sC}: \sC\to D(G)\otimes \sC$. The functor $\on{coact}_{\sC^{U_P}}^{B_M,\chi}$ is defined similarly. Since $\sC^{U_P}$ is nilpotent as an $M$-category (cf. Corollary \ref{c:nilp}), it suffices to show that the image of $\on{coact}_{\sC^{U_P}}^{B_M,\chi}$ lands in $D_{\on{Res}^G_M(\on{SS}(\sC))}(B_M,\chi\backslash M)\otimes \sC^{U_P}$ for all $\chi\in \ft^*/X^{\bullet}(T)(k)$.

From Proposition \ref{p:heckedescrip}, Lemma \ref{l:pullb} and the above diagram, we conclude that the composition
\[
\sC^{B,\chi}\xrightarrow{\on{coact}_{\sC^{U_P}}^{B_M,\chi}} D(B_M,\chi\backslash M)\otimes \sC^{U_P}\xrightarrow{\on{id}\otimes \on{oblv}^{U_P}} D(B_M,\chi\backslash M)\otimes \sC
\]

\noindent lands in 
\[
\sH_{M,\chi,\on{Res}^G_M(\on{SS}(\sC))}\underset{\sH_{M,\chi}}{\otimes} D(B_M,\chi\backslash M)\otimes \sC\subset D_{\on{Res}^G_M(\on{SS}(\sC))}(B_M,\chi\backslash M)\otimes \sC.
\]

\noindent In particular, the image of $\on{coact}_{\sC^{U_P}}^{B_M,\chi}$ lands in $D_{\on{Res}^G_M(\on{SS}(\sC))}(B_M,\chi\backslash M)\otimes \sC^{U_P}$. 

\step
Next, we show that the inclusion
\[
\on{SS}(\on{Res}^G_M(\sC))\subset \on{Res}^G_M(\on{SS}(\sC))
\]

\noindent is an equality. Suppose for contradiction that it is not. In this case, by Theorem \ref{t:nilp} and Proposition \ref{p:heckedescrip}, we can find $\chi\in\ft^*/X^{\bullet}(T)(k), c\in \sC^{B,\chi}, \phi\in\sC^{\vee}$ and $\cG\in D(G/B,\chi)$ such that the image of $c$ (denoted $\cF$) under the composition
\[
\sC^{B,\chi}\xrightarrow{\on{coact}_{\sC}^{B,\chi}} D(B,\chi\backslash G)\otimes \sC\xrightarrow{\on{id}\otimes \phi} D(B,\chi\backslash G)\xrightarrow{-\star\cG} D(B,\chi\backslash G/B,\chi)
\]

\noindent satisfies that $\on{Res}_M^G(\on{SS}(\cF))$ is not contained $\on{SS}(\on{Res}_M^G(\sC))$. By Lemma \ref{l:converseprop}, $\cF$ is not an object of
\[
\sH_{M,\chi,\on{SS}(\on{Res}_M^G(\sC))}\underset{\sH_{M,\chi}}{\otimes} \sH_{G,\chi}.
\]

\noindent However, applying \ref{p:heckedescrip} again, we must have
\[
\sC^{B,\chi}\simeq \sH_{M,\chi,\on{SS}(\on{Res}_M^G(\sC))}\underset{\sH_{M,\chi}}{\otimes} \sC^{B,\chi}
\]

\noindent so that 
\[
\cF=(-\star\cG)\circ (\on{id}\otimes\phi)\circ \on{coact}_{\sC}^{B,\chi}(c)\in \sH_{M,\chi,\on{SS}(\on{Res}_M^G(\sC))}\underset{\sH_{M,\chi}}{\otimes} \sH_{G,\chi},
\]

\noindent yielding a contradiction.

\end{proof}

\subsection{Singular support estimate on parabolic induction}\label{s:pind}
We now consider how parabolic induction of categories interacts with singular support. As we will see, the situation is simpler compared to that of parabolic restriction.\footnote{The intuition here is that parabolic induction is given by a pull-push along $M\leftarrow P\rightarrow G$, where the first map is smooth and the second is a closed embedding, and these operations play well with singular support.} 

\subsubsection{}
First, note that the notion of singular support for a $G$-category makes sense for any algebraic group. For a map of algebraic groups $H\to K$ and any dualizable $K$-category $\sC$, we let $\on{SS}_H(\sC)$ denote the singular support of $\sC$ considered as an $H$-category. 

We have the following basic lemma:
\begin{lem}\label{l:smoothpullback}
Let $f:H\to K$ be a smooth map of connected algebraic groups, and let $\sC$ be a dualizable $K$-category. We have:
\[
\on{SS}_H(\sC)=\on{Im}(\on{SS}_K(\sC)\subset \fk^*=\on{Lie}(K)^*\to \fh^*).
\]
\end{lem}

\begin{proof}
Let $Z:=\on{Im}(\on{SS}_K(\sC)\subset \fk^*\to \fh^*)$. Note that the coaction map of $H\curvearrowright \sC$ is given by the composition
\[
\sC\xrightarrow{\on{coact}} D(K)\otimes \sC\xrightarrow{f^!\otimes \on{id}_{\sC}} D(H)\otimes \sC,
\]

\noindent where the first map is given by the coaction map for $K\curvearrowright \sC$. We need to show that for all $\phi\in \sC^{\vee}$, the composition
\[
\sC\xrightarrow{\on{coact}} D(K)\otimes \sC\xrightarrow{f^!\otimes \on{id}_{\sC}} D(H)\otimes \sC\xrightarrow{\on{id}_{D(H)}\otimes \phi} D(H)
\]

\noindent lands in $D_Z(H)$, and moreover that $Z\subset \fh^*$ is the smallest closed, conical $H$-stable subset so that this holds. However, this simply follows from the singular support estimate for pullback via smooth maps, cf. §\ref{s:functsing}.
\end{proof}

\subsubsection{} Next, let $i: H\into K$ be a closed embedding of connected algebraic groups. For an $H$-category $\sC$, we write $\on{Ind}_H^K(\sC):=D(K)\underset{D(H)}{\otimes} \sC$ equipped with its natural $K$-action. Let $di^{\bullet}: \fk^*\onto \fh^*$ be the induced map of dual Lie algebras. Note that for any $\phi\in\sC^{\vee}$, we obtain an $H$-linear map
\[
\sC\xrightarrow{\on{coact}} D(H)\otimes \sC\xrightarrow{\on{id}\otimes \phi} D(H).
\]

\noindent By abuse of notation, we denote the above composition by $\phi$ as well. Similarly, for any $\widetilde{\phi}\in D(K)^{\vee}$, we obtain an $H$-linear map
\[
\widetilde{\phi}: D(K)\to D(H).
\]
\begin{lem}\label{l:oneinc}
Let $\sC$ be a dualizable $H$-category. Then $\on{Ad}_K((di^{\bullet})^{-1}(\on{SS}(\sC)))\subset\on{SS}(\on{Ind}_H^K(\sC))$.
\end{lem}

\begin{proof}
Since $\on{SS}(\on{Ind}_H^K(\sC))$ is $K$-stable by definition, it suffices to show that
\[
(di^{\bullet})^{-1}(\on{SS}(\sC))\subset\on{SS}(\on{Ind}_H^K(\sC)).
\]

\noindent We need to show that for all $c\in \sC, \phi\in\sC^{\vee}$, we have
\[
(di^{\bullet})^{-1}\on{SS}(\phi(c))\subset \on{SS}(\on{Ind}_H^K(\sC)).
\]

\noindent From the singular support estimate of closed embeddings (cf. §\ref{s:functsing}), this is equivalent to showing that the map (see Remark \ref{r:MC}):
\[
\sC\otimes \sC^{\vee}\xrightarrow{\on{MC}_{\sC}} D(H)\xrightarrow{i_{*,\on{dR}}} D(K)
\]

\noindent lands in $D_{\on{SS}(\on{Ind}_H^K(\sC))}(K)$. Consider the functors
\[
i_{*,\on{dR}}: \sC\to D(K)\underset{D(H)}{\otimes} \sC=\on{Ind}_H^K(\sC),
\]
\[
(i^!)^{\vee}: \sC^{\vee}\to (D(K)\underset{D(H)}{\otimes} \sC)^{\vee}=\on{Ind}_H^K(\sC)^{\vee}.
\]

\noindent We have a commutative diagram
\[\begin{tikzcd}
	{\sC\otimes \sC^{\vee}} && {D(H)} \\
	\\
	\on{Ind}_H^K(\sC)\otimes \on{Ind}_H^K(\sC)^{\vee} && {D(K).}
	\arrow["{\mathrm{MC}_{\sC}}", from=1-1, to=1-3]
	\arrow["{i_{*,\mathrm{dR}}\otimes (i^!)^{\vee}}"', from=1-1, to=3-1]
	\arrow["{i_{*,\mathrm{dR}}}", from=1-3, to=3-3]
	\arrow["{\mathrm{MC}_{\on{Ind}_H^K(\sC)}}", from=3-1, to=3-3]
\end{tikzcd}\]

\noindent By definition, the lower horizontal maps factors through $D_{\on{SS}(\on{Ind}_H^K(\sC))}(K)$.

\end{proof}

\subsubsection{} We may now prove the main theorem of this subsection. Resuming the previous notation, we let $G$ be a connected reductive group and $M$ a Levi subgroup.
\begin{thm}\label{t:pind}
Let $\sC$ be a nilpotent $M$-category. Then $\on{SS}(\on{Ind}_M^G(\sC))=\on{Ind}_M^G(\on{SS}(\sC))$.
\end{thm}

\begin{proof}
Combining Lemma \ref{l:smoothpullback} and Lemma \ref{l:oneinc}, we see that we have an inclusion
\[
\on{Ind}_M^G(\on{SS}(\sC))\subset \on{SS}(\on{Ind}_M^G(\sC)),
\]

\noindent and so it remains to prove the reverse inclusion.

By Theorem \ref{t:nilp}, we may assume that $\sC$ is $(B_M,\chi)$-generated for some $\chi\in \ft^*/X^{\bullet}(T)(k)$. That is:
\[
D(M/B_M,\chi)\underset{\sH_{M,\chi}}{\otimes} \sC^{B_M,\chi}\simeq\sC. 
\]

\noindent Consider $\sC$ as a $P$-category via the projection $P\to M$. Note that $\sC\simeq \sC^{U_P}$ as a $P$-category. As such, we obtain
\begin{equation}\label{eq:paradescrip}
D(P/B,\chi)\underset{D(B,\chi\backslash P/B,\chi)}{\otimes} \sC^{B,\chi}\simeq \sC.
\end{equation}

\noindent In particular:
\[
\on{Ind}_M^G(\sC)=D(G)\underset{D(P)}{\otimes} \sC\simeq D(G/B,\chi)\underset{D(B,\chi\backslash P/B,\chi)}{\otimes} \sC^{B,\chi}.
\]

\noindent This shows that $\on{Ind}_M^G(\sC)$ is $(B,\chi)$-generated as a $G$-category (since $D(G/B,\chi)$ is). It suffices to show that the $\sH_{G,\chi}$-category
\begin{equation}\label{eq:pind}
\on{Ind}_M^G(\sC)^{B,\chi}\simeq \sH_{G,\chi}\underset{D(B,\chi\backslash P/B,\chi)}{\otimes} \sC^{B,\chi}
\end{equation}

\noindent is supported on $\sH_{G,\chi,\on{Ind}_M^G(\on{SS}(\sC))}$. Let us denote by
\[
D_{\on{SS}(\sC)}(B,\chi\backslash P/B,\chi)
\]

\noindent the subcategory of $D(B,\chi\backslash P/B,\chi)$ consisting of objects whose singular support (considered as a D-module on $P$) is contained in $\on{SS}(\sC)\subset \fm^*\subset \fp^*$. Note that we have a canonical monoidal equivalence
\[
D(B,\chi\backslash P/B,\chi)\simeq D(B_M,\chi\backslash M/B_M,\chi).
\]

\noindent We may rewrite (\ref{eq:pind}) as
\[
\sH_{G,\chi}\underset{D(B,\chi\backslash P/B,\chi)}{\otimes}\big(D_{\on{SS}(\sC)}(B,\chi\backslash P/B,\chi)\underset{D(B,\chi\backslash P/B,\chi)}{\otimes} \sC^{B,\chi}\big).
\]

\noindent The theorem now follows from the fact that the convolution functor
\[
\sH_{G,\chi}\underset{D(B,\chi\backslash P/B,\chi)}{\otimes}D_{\on{SS}(\sC)}(B,\chi\backslash P/B,\chi)\to \sH_{G,\chi}
\]

\noindent factors through $\sH_{G,\chi,\on{Ind}_M^G(\on{SS}(\sC))}$, cf. §\ref{s:functsing}.
\end{proof}

\subsection{Filtration on nilpotent $G$-categories}
In this subsection, we introduce a natural filtration by nilpotent orbits on nilpotent $G$-categories. Moreover, we show that if $\sC$ is $B$-generated, then the filtration is determined by special nilpotent orbits.

\subsubsection{}\label{s:canfiltdef} Throughout, let $\sC$ be a nilpotent $G$-category. Let $\sC_Z\subset \sC$ be the full subcategory generated under colimits by the essential image of the convolution functor
\[
\underset{\chi\in\ft^*/X^{\bullet}(T)(k)}{\bigoplus} D(G/B,\chi)\underset{\sH_{G,\chi}}{\otimes}\sH_{G,\chi,Z}\underset{\sH_{G,\chi}}{\otimes} \sC^{B,\chi}\to \sC.
\]

\noindent Note that $\sC_Z$ is $G$-stable. It is easy to see that $\sC_Z$ is the largest $G$-stable subcategory of $\sC$ with $\on{SS}(\sC_Z)\subset Z$. Moreover, by Theorem \ref{t:nilp}, $\sC_{\cN}=\sC$. As such, we obtain a filtration of our category $\sC$ by $G$-subcategories indexed by nilpotent orbits.

\begin{prop}\label{l:stable}
The subcategory $\sC_Z$ is dualizable. Moreover, the inclusion
\[
\sC_Z\into \sC
\]

\noindent admits a left adjoint.
\end{prop}

\begin{proof}
By Theorem \ref{t:nilp}, we may assume that $\sC$ is $(B,\chi)$-generated for some $\chi\in \ft^*/X^{\bullet}(T)(k)$. 

We claim that the inclusion
\[
D(G/B,\chi)\underset{\sH_{G,\chi}}{\otimes}\sH_{G,\chi,Z}\into D(G/B,\chi)
\]

\noindent admits a left adjoint. Indeed, we saw in the proof of Theorem \ref{t:van1} that $\on{Ker}(\on{Av}_*^{\psi_e})=\on{Ker}(\on{Av}_!^B\circ \on{Av}_*^{\psi_e})$, and that $\on{Av}_!^B\circ \on{Av}_*^{\psi_e}$ commuted with cofiltered limits. It then follows from Theorem \ref{t:van} that the inclusion

\[
D(G/B,\chi)\underset{\sH_{G,\chi}}{\otimes}\sH_{G,\chi,Z}\to D(G/B,\chi)
\]

\noindent commutes with cofiltered limits. In particular, it admits a left adjoint:
\[
D(G/B,\chi)\to D(G/B,\chi)\underset{\sH_{G,\chi}}{\otimes}\sH_{G,\chi,Z}.
\]

\noindent This implies that the right-hand side is compactly generated, and hence dualizable. It follows formally that the inclusion 

\[\sC_Z\simeq D(G/B,\chi)\underset{\sH_{G,\chi}}{\otimes}\sH_{G,\chi,Z}\underset{\sH_{G,\chi}}{\otimes} \sC^{B,\chi}\into D(G/B,\chi)\underset{\sH_{G,\chi}}{\otimes} \sC^{B,\chi}\simeq \sC
\]

\noindent similarly admits a left adjoint. Moreover, from the above presentation of $\sC$, we see that $\sC_Z$ is dualizable.

\end{proof}

\subsubsection{Special nilpotent orbits}\label{s:sospecial} There are various (non-obviously) equivalent definitions of special nilpotent orbits. For the original due to Lusztig, see \cite{lusztig1979class}. Spaltenstein \cite[Chap. III]{spaltenstein2006classes} later gave a definition in terms of a duality of nilpotent orbits of $\fg$ and $\check{\fg}$, the Lie algebra of the Langlands dual group of $G$.

For us, we need the following characterization due to Barbasch-Vogan \cite[Thm. 1.1]{barbasch1983primitive} and Duflo \cite[Thm. 4.3]{duflo1977classification}: let $L_w=L_{w\cdot 0}$ be the simple highest weight module with central character $0$ corresponding to $w\in W$. Then the singular support (or associated variety) of the annihilator of $L_w$ in $U(\fg)$ is the closure of a special nilpotent orbit. Moreover, every special nilpotent orbit arises this way. We write $\fX(\fg)^{\on{sp}}$ for the set of special nilpotent orbits of $\fg$.

\subsubsection{} Let $Z\subset \cN$ be a closed $G$-stable subset of the nilpotent cone. Let
\[
Z^{\on{sp}}=\bigcup_{\substack{\bO\in\fX(\fg)^{\on{sp}} \\ \bO\subset Z}}\overline{\bO}\subset Z
\]

\noindent be the union of the closures of all special nilpotent orbits contained in $Z$.
\begin{prop}\label{p:specialneeds}
Suppose $\sC$ is a $B$-generated (and in particular nilpotent) dualizable $G$-category. Then the inclusion
\[
\sC_{Z^{\on{sp}}}\subset \sC_Z
\]

\noindent is an equivalence.
\end{prop}

\begin{proof}

We have:
\[
\sC_Z\simeq D(G/B,\chi)\underset{\sH_{G,\chi}}{\otimes}\sH_{G,\chi,Z}\underset{\sH_{G,\chi}}{\otimes} \sC^{B,\chi},
\]
\[
\sC_{Z^{\on{sp}}}\simeq D(G/B,\chi)\underset{\sH_{G,\chi}}{\otimes}\sH_{G,\chi,Z^{\on{sp}}}\underset{\sH_{G,\chi}}{\otimes} \sC^{B,\chi}.
\]

\noindent As such, it suffices to show that the inclusion
\[
\sH_{G,Z^{\on{sp}}}\into\sH_{G,Z}
\]

\noindent is an equivalence. However, this follows from the definition of special nilpotent orbits.
\end{proof}

\begin{rem}\label{r:specialgen}
If $\sC$ is nilpotent but not $B$-generated, Proposition \ref{p:specialneeds} fails in general, because the singular support of the annihilator of a highest weight module with non-zero central character need not be the closure of a special nilpotent orbit. Instead, if $\sC^N=\sC^{(B,\chi)\on{-mon}}$ for some $\chi$, there is an obvious analogue replacing special nilpotent orbits with orbits maximal in the singular support of the annihilator of highest weight modules with central character $\lambda\in\ft^*$, where $\lambda$ is a regular lift of $\chi$.
\end{rem}

\subsection{Whittaker characterization of singular support of $G$-categories}\label{s:whitchar}

We now prove the main result of this section. 

\subsubsection{} For a nilpotent element $e\in\bO$, write
\[
\on{Whit}_{e}(\sC):=\sC^{U^-,\psi_e}.
\]

\noindent Since we have a (non-canonical) equivalence
\[
\on{Whit}_e(\sC)\simeq \on{Whit}_{e'}(\sC)
\]

\noindent for Ad-conjugate elements $e,e'$,\footnote{That is, $\on{Whit}_e(\sC)$ is non-canonically independent of the choice of an $\mathfrak{sl}_2$-triple and a Lagrangian in the construction of $U^{-}$. Indeed, this follows from the fact that for a Heisenberg group $H$ with a Lagrangian $L\subset H$ and a categorical representation $D(H)\curvearrowright \sC$, one has a canonical equivalence $\sC^{L}\simeq \sC^{H,\on{w}}$, whenever the center of $H$ (which is isomorphic to $\bG_a$) acts by an exponential character. The centralizer of the $\mathfrak{sl}_2$-triple acts non-trivially on $\on{Whit}_e(\sC)$ in general.} we also denote this common DG-category by $\on{Whit}_{\bO}(\sC)$.

\begin{thm}\label{t:main}
Let $\sC$ be a nilpotent $G$-category. Then 
\[
\on{Whit}_{\bO}(\sC)=0
\]

\noindent whenever $\bO\cap\on{SS}(\sC)=\emptyset$. Moreover, if $\bO$ is a nilpotent orbit maximal with respect to the property that $\bO\cap \on{SS}(\sC)\neq\emptyset$, then
\[
\on{Whit}_{\bO}(\sC)\neq 0.
\]
\end{thm}

\begin{proof}
By Theorem \ref{t:nilp}, we may assume that $\sC$ is $(B,\chi)$-generated for some $\chi\in \ft^*/X^{\bullet}(k)$. In particular, any $G$-subcategory of $\sC$ is also $(B,\chi)$-generated, cf. \cite[Cor. 3.2.5]{campbell2021affine}.

\step Let us prove the first part of the theorem. Thus, let $e$ be a nilpotent element with $e\notin\on{SS}(\sC)$. By Proposition \ref{p:heckedescrip}, it follows that
\[
\on{Whit}_{\bO}(\sC)\simeq \big(D(U^{-},\psi_e\backslash G/B,\chi)\underset{\sH_{G,\chi}}{\otimes}\sH_{G,\chi,Z}\big) \underset{\sH_{G,\chi}}{\otimes} \sC^{B,\chi}.
\]

\noindent However by Theorem \ref{t:van1'}, the category $D(U^{-},\psi_e\backslash G/B,\chi)\underset{\sH_{G,\chi}}{\otimes}\sH_{G,\chi,Z}$ is zero.

\step Next, we prove the second statement. Let $\bO$ be a maximal nilpotent orbit with $\bO\subset \on{SS}(\sC)$. Since $\sC$ is $(B,\chi)$-generated, applying Proposition \ref{p:heckedescrip} again, we can find $c\in \sC^{B,\chi}, \phi\in \sC^{\vee}, \cG\in D(G/B,\chi)$ such that the image of $c$ (denoted $\cF$) under
\[
\sC^{B,\chi}\xrightarrow{\on{coact}^{B,\chi}} D(B,\chi\backslash G)\otimes \sC\xrightarrow{\on{id}\otimes \phi} D(B,\chi\backslash G)\xrightarrow{-\star\cG} D(B,\chi\backslash G/B,\chi)
\]

\noindent satisfies $\bO\cap\on{SS}(\cF)\neq\emptyset$. Since $\on{SS}(\cF)\subset \on{SS}(\sC)$, we conclude that $\bO$ is a maximal orbit intersecting $\on{SS}(\cF)$. By Theorem \ref{t:van2}, we may find $e\in\bO$ such that 
\[
0\neq \on{Av}_*^{\psi_e}\circ(-\star\cG)\circ(\on{id}\otimes\phi)\circ\on{coact}^{B,\chi}(c)\simeq (-\star\cG)\circ(\on{id}\otimes\phi)\circ\on{coact}^{U^-,\psi_e}\on{Av}_*^{\psi_e}(c)\in D(U^-,\psi_e\backslash G/B,\chi).
\]

\noindent Here, we consider $\on{Av}_*^{\psi_e}$ as averaging on the left. In particular, $\on{Av}_*^{\psi_e}(c)\neq 0$, which implies $\on{Whit}_{\bO}(\sC)\neq 0$.

\end{proof}

\section{Whittaker coefficients of character sheaves}\label{s:5}
In this section, we study generalized Whittaker coefficients of character sheaves. Our main result is Theorem \ref{t:finalboss?}, which realizes the Hochschild homology of the Whittaker models of nilpotent $G$-categories (at maximal orbits) in terms of the microstalk of its corresponding character sheaf.

\subsection{Hochschild homology of DG-categories}\label{S:5.1}
We review some basic constructions of categorical traces and Hochschild homology. We refer to \cite{ben2009character}, \cite{ben2013nonlinear}, \cite{gaitsgory2022toy} for details.

\subsubsection{} 

Recall that if $\sD$ is a dualizable DG-category, we may talk about its Hochschild homology, denoted $\on{HH}_*(\sC)\in\on{Vect}$. By definition, $\on{HH}_*(\sC)$ is the composition
\[
\on{Vect}\xrightarrow{\eta_{\sD}} \sD\otimes \sD^{\vee}\xrightarrow{\epsilon_{\sD}} \on{Vect},
\]

\noindent where $\eta_{\sD}$ (resp. $\epsilon_{\sD}$) is the unit (resp. counit) of the duality datum of $\sD$. Since $\on{End}_{\on{DGCat}_{\on{cont}}}(\on{Vect})\simeq \on{Vect}$, this composition is given by tensoring with a vector space $\on{HH}_*(\sC)$.

\subsubsection{}\label{s:HHendo} More generally, if $F\in\on{End}_{\on{DGCat}_{\on{cont}}}(\sD)$, we define $\on{HH}_*(\sD,F)\in\on{Vect}$ to be the composition
\[
\on{Vect}\xrightarrow{\eta_{\sD}} \sD\otimes \sD^{\vee}\xrightarrow{F\otimes \on{id}}\sD\otimes \sD^{\vee}\xrightarrow{\epsilon_{\sD}} \on{Vect}.
\]

\noindent We refer to $\on{HH}_*(\sD,F)$ as the Hochschild homology of $\sD$ with coefficients in $F$. Note that taking Hochschild homology is functorial in the natural sense: if $F\to G$ is a natural transformation of continuous endofunctors of $\sD$, we obtain a map of vector spaces
\[
\on{HH}_*(\sD,F)\to \on{HH}_*(\sD,G).
\]

\subsubsection{} Let $F_{1,2}: \sD_1\to \sD_2, \; F_{2,1}: \sD_2\to \sD_1$ be continuous functors. Hochschild homology is \emph{cyclic} in the sense that we have a natural isomorphism
\[
\on{HH}_*(\sD_1,F_{2,1}\circ F_{1,2})\simeq \on{HH}_*(\sD_2, F_{1,2}\circ F_{2,1})
\]

\subsubsection{}\label{s:514} Let $F:\sD_1\to \sD_2$ be a continuous functor which admits a continuous right adjoint, $G=F^R$. We obtain a map $\on{HH}_*(\sD_1)\to \on{HH}_*(\sD)$ as the composition
\[
\on{HH}_*(\sD_1)=\on{HH}_*(\sD_1,\on{id}_{\sD_1})\to \on{HH}_*(\sD_1, G\circ F)\simeq \on{HH}_*(\sD_2, F\circ G)\to \on{HH}_*(\sD_2, \on{id}_{\sD_2})=\on{HH}_*(\sD_2).
\]

\noindent Here, the first (resp. last) arrow is the unit (resp. counit) of the adjunction between $F$ and $G$.

\subsubsection{}\label{s:quaterlyrigid} Let $\sA$ be a dualizable monoidal DG-category such that the monoidal structure map
\[
\sA\otimes \sA\to \sA
\]

\noindent admits a continuous right adjoint. The functoriality of Hochschild homology described in the previous paragraph exhibits $\on{HH}_*(\sA)$ as an associative algebra.

If $\sM$ is a dualizable DG-category equipped with a module structure for $\sA$ such that the action map
\[
\sA\otimes \sM\to \sM
\]

\noindent admits a continuous right adjoint, we similary obtain an action
\[
\on{HH}_*(\sA)\curvearrowright \on{HH}_*(\sM).
\]

\subsubsection{Example} Let $\sA=\sH_{G,\chi}=D(B,\chi\backslash G/B,\chi)$. The convolution map
\[
\sH_{G,\chi}\otimes \sH_{G,\chi}\to \sH_{G,\chi}
\]

\noindent admits a continuous right adjoint, and hence $\on{HH}_*(\sH_{G,\chi})$ is an associative algebra. In fact, $\sH_{G,\chi}$ is \emph{semi-rigid} in the sense of \cite{ben2009character}. As such, if $\sM$ is a dualizable DG-category equipped with a module structure for $\sH_{G,\chi}$, then $\sM$ is dualizable as a $\sH_{B,\chi}$-category, and the action map
\[
\sH_{G,\chi}\otimes \sM\to \sM
\]

\noindent admits a continuous right adjoint. In particular, we obtain an action
\[
\on{HH}_*(\sH_{G,\chi})\curvearrowright \on{HH}_*(\sM).
\]

\subsection{Character sheaves of $G$-categories} 

\subsubsection{} Let $\sC$ be a dualizable $G$-category. We have unit and counit maps
\[
\eta_{\sC}: D(G)\to \sC\otimes \sC^{\vee},\;\; \epsilon_{\sC}: \sC\otimes \sC^{\vee}\to D(G).
\]

\noindent The functor $\eta_{\sC}: D(G)\to \sC\otimes \sC^{\vee}\simeq\on{End}_{\on{DGCat}_{\on{cont}}}(\sC)$ is nothing more than the module structure of $\sC$, while $\epsilon_{\sC}=\on{MC}_{\sC}$ is the functor of matrix coefficients. The composition
\[
D(G)\to \sC\otimes \sC^{\vee}\to D(G)
\]

\noindent is naturally $G\times G$-linear. Thus, it gives rise to an object
\[
\chi_{\sC}\in \on{End}_{D(G\times G)\on{\mathbf{-mod}}}(D(G))\simeq D(G)\underset{D(G)\otimes D(G)}{\otimes} D(G)\simeq D(G\overset{\on{ad}}{/}G).
\]

\noindent Here, $G\overset{\on{ad}}{/}G$ denotes the quotient of $G$ by itself under the adjoint action.
\begin{defin}
$\chi_{\sC}\in D(G\overset{\on{ad}}{/}G)$ is the \emph{character sheaf} associated to $\sC$.
\end{defin}
\begin{rem}
We note that $\chi_{\sC}$ is nothing but $\on{MC}_{\sC}(\on{id}_{\sC})$.
\end{rem}

\subsubsection{Compatibility with parabolic restriction and induction}\label{s:compindpres} Let $P\subset G$ be a parabolic subgroup with Levi $M$. Consider the correspondence
\[\begin{tikzcd}
	& {P\overset{\mathrm{ad}}{/}P} \\
	{G\overset{\mathrm{ad}}{/}G} && {M\overset{\mathrm{ad}}{/}M.}
	\arrow["p"', from=1-2, to=2-1]
	\arrow["q", from=1-2, to=2-3]
\end{tikzcd}\]

\noindent The following is clear from the definitions:
\begin{lem}\label{l:compwithpind}
Let $\sC$ be a dualizable $G$-category. Then
\[
\chi_{\on{Res}_M^G(\sC)}=q_{*,\on{dR}}\circ p^!(\chi_{\sC}).
\]

\noindent Similarly, if $\sC_M$ is a dualizable $M$-category, then
\[
\chi_{\on{Ind}_M^G(\sC_M)}=p_{*,\on{dR}}\circ q^!(\chi_{\sC_M}).
\]
\end{lem}

\subsubsection{Example}\label{e:Springer} Let $\sC=D(G/B,\chi)$ with its left $G$-action. Consider the action
\[
D(T)\curvearrowright D(T/T,\chi).
\]

\noindent Then
\[
D(G/B,\chi)\simeq \on{Ind}_T^G(D(T/T,\chi)).
\]

\noindent As such, $\chi_{\sC}$ is the (Grothendieck-)Springer sheaf, which we denote by $\on{Spr}_{\chi}$. That is, 
$\on{Spr}_{\chi}$ is the image of $\chi\in D(T\overset{\on{ad}}{/}T)$ under $!$-pull, $*$-push along the correspondence
\[\begin{tikzcd}
	& {B\overset{\mathrm{ad}}{/}B} \\
	{G\overset{\mathrm{ad}}{/}G} && {T\overset{\mathrm{ad}}{/}T.}
	\arrow["p"', from=1-2, to=2-1]
	\arrow["q", from=1-2, to=2-3]
\end{tikzcd}\]

\subsubsection{}\label{s:bimodulestructure} Let $\sC$ be a $(D(G),\sH_{G,\chi})$-bimodule category. In this case, $\chi_{\sC}$ naturally acquires an action of $\on{HH}_*(\sH_{G,\chi})$. Informally, the action is given as follows. Consider
\[
\sC\otimes \sH_{G,\chi}
\]

\noindent as a $G$-category in the natural way. We have
\[
\chi_{\sC\otimes \sH_{G,\chi}}\simeq \chi_{\sC}\otimes \on{HH}_*(\sH_{G,\chi}).
\]

\noindent Since the action map
\[
\sC\otimes \sH_{G,\chi}\to \sC
\]

\noindent is $G$-linear and admits a continuous right adjoint, we get a map
\[
\chi_{\sC}\otimes \on{HH}_*(\sH_{G,\chi})\to \chi_{\sC},
\]

\noindent which is the desired action of $\on{HH}_*(\sH_{G,\chi})$ on $\chi_{\sC}$.

More formally, let $D(G)\textrm{\textbf{-mod}}^r$ be the $2$-category of dualizable $G$-categories, where functors are required to be $G$-linear and admits continuous right adjoints.\footnote{Which we remind are automatically $G$-linear.} The functor
\begin{equation}\label{eq:charsheaf}
\chi: D(G)\textrm{\textbf{-mod}}^r\to D(G\overset{\on{ad}}{/}G),\;\; \sC\mapsto \chi_{\sC}
\end{equation}

\noindent then upgrades to a functor
\[
\sH_{G,\chi}\on{-mod}(D(G)\textrm{\textbf{-mod}}^r)\to  \on{HH}_*(\sH_{G,\chi})\on{-mod}(D(G\overset{\on{ad}}{/}G)).
\]

\subsubsection{} From Example \ref{e:Springer}, we obtain an action
\[
\on{Spr}_{\chi}\curvearrowleft \on{HH}_*(\sH_{G,\chi}).
\]

\begin{lem}\label{l:charreg}
Let $\sC$ be a nilpotent $G$-category. Then $\chi_{\sC}\in D(G\overset{\on{ad}}{/} G)$ is regular holonomic.
\end{lem}

\begin{proof}
By assumption, $\chi_{\sC}$ is a character sheaf in the sense of \S \ref{S:5.4} below. In particular, it is regular holonomic.
\end{proof}

\subsection{Character sheaves on the Lie algebra}\label{s:charliealg}
In this subsection, we realize Whittaker coefficients of character sheaves on $\fg$ as microstalks. We will ultimately deduce Theorem \ref{t:catstate} from this situation.

\subsubsection{} We say that a D-module on $\fg$ has nilpotent singular support if the singular support is contained in $\fg\times \cN\subset T^*\fg$.

\begin{defin}
A character sheaf $\cF\in D(\fg/G)$ is D-module with nilpotent singular support. We remind that these are all regular holonomic.
\end{defin}

\subsubsection{} Consider the scaling action of $\bG_m$ on $\fg$. The following is well-known (see e.g. \cite[Cor. 4.8+Thm. 5.3]{mirkovic2004character}):
\begin{prop}\label{p:scalinginv}
Any character sheaf $\cF\in D(\fg/G)$ is $\bG_m$-monodromic.
\end{prop}

\subsubsection{} Consider the variety
\[
\Lambda:=\lbrace (x,y)\in \fg\times \cN\;\vert\; [x,y]=0\rbrace.
\]

\noindent More generally, for a closed, conical Ad-invariant $Z\subset \cN$, we may consider the subvariety
\[
\Lambda_Z:=\lbrace (x,y)\in \fg\times Z\;\vert\; [x,y]=0\rbrace\subset \Lambda.
\]

\noindent If $\bO\subset Z$ is a nilpotent orbit, the corresponding subvariety
\[
\Lambda_{\bO}:=\Lambda_Z\underset{Z}{\times}\bO\subset \Lambda_Z
\]

\noindent is smooth and irreducible, since it is the conormal bundle of $\bO$. Thus, $\Lambda_Z$ is Lagrangian. Moreover, the closures of $\Lambda_{\bO}$ for $\bO\subset Z$ are exactly the irreducible components of $\Lambda_Z$.

\subsubsection{} The underlying classical stack of $T^*(\fg/G)$ naturally identifies with 
\[
T^*(\fg/G)=\lbrace (x,y)\in \fg\times \fg^*\;\vert\; [x,y]=0\rbrace/G.
\]

\noindent In particular,
\[
T^*(\fg/G)\underset{\fg^*/G}{\times}\cN/G=\Lambda/G.
\]

\subsubsection{} For a coherent character sheaf $\cF$, considered as a sheaf on $\fg$, its characteristic cycle takes the form
\[
\on{CC}(\cF)=\underset{\bO\subset \cN}{\sum} c_{\bO,\cF}[\bO],
\]

\noindent where $c_{\bO,\cF}$ denotes the multiplicity of $\overline{\Lambda_{\bO}}$ in the characteristic variety of $\cF$.

\subsubsection{Whittaker coefficients}

For a nilpotent element $e\in \bO$, consider the map
\[
\fp_e: \fu^{-}/U^{-}\to \fg/G.
\]

\noindent We get a corresponding Whittaker functional
\[
\on{coeff}_{\bO}: D(\fg/G)\to \on{Vect},\;\; \cF\mapsto C_{\on{dR}}(\fu^{-}/U^{-}, \fp_e^!(\cF)\overset{!}{\otimes} \psi_e^!(\on{exp})).
\]

\noindent Note that the functional is independent of the choice of $e\in\bO$ and an $\mathfrak{sl}_2$-triple containing it. It is also not difficult to show that it is independent of the Lagrangian $\ell\subset \fg(-1)$.

\subsubsection{} For a character sheaf $\cF\in D(\fg/G)$, we say that a nilpotent orbit $\bO$ is maximal in the singular support of $\cF$ if either $\bO$ is a maximal orbit satisfying $\fg\times \bO\cap \on{SS}(\cF)\neq\emptyset$, or if $\fg\times\bO\cap \on{SS}(\cF)=\emptyset$.
\begin{thm}\label{t:whitcoefflie}
Let $\cF\in D(\fg/G)^{\heartsuit}$ be a perverse character sheaf, and let $\bO$ be a maximal orbit in the singular support of $\cF$. Then
\[
\on{coeff}_{\bO}(\cF)\in\on{Vect}
\]

\noindent is concentrated in cohomological degree zero, and moreover
\[
\on{dim}\on{coeff}_{\bO}(\cF)=c_{\bO,\cF}.
\]
\end{thm}

\begin{proof}
Consider the Fourier-Deligne transform:
\[
\on{FT}=\on{FT}_{\fg}: D(\fg)\to D(\fg^*),\;\; \cF\mapsto (\on{pr}_2)_{*,\on{dR}}(\on{pr}_1^!(\cF)\overset{!}{\otimes} Q^!(\on{exp})).
\]

\noindent Here:
\begin{itemize}
    \item $Q$ denotes the canonical pairing $\fg\times \fg^*\to \bA^1$.

    \item $\on{exp}$ denotes the exponential sheaf on $\bA^1$ concentrated in perverse degree $-1$.

    \item $\on{pr}_1$ (resp. $\on{pr}_2$) is the projection map $\fg\times \fg^*\to \fg$ (resp. $\fg\times \fg^*\to \fg^*$).

\end{itemize}

\noindent By \cite[Corollary 1.7.9]{gaitsgory2013functors}, the functor $\on{FT}[-\on{dim} G]$ is t-exact. The Fourier transform descends to a functor
\[
D(\fg/G)\to D(\fg^*/G).
\]

\noindent We similarly have a Fourier transform
\[
\on{FT}_{\fu^{-}}: D(\fu^{-}/U^{-})\to D((\fu^{-})^*/U^{-}).
\]

\noindent Let $i_e: \on{Spec}(k)\to (\fu^{-})^*/U^{-}$ be the $k$-point corresponding to $\psi_e$. We have a tautological isomorphism of functors
\begin{equation}\label{eq:FTu}
i_e^!\circ \on{FT}_{\fu^{-}}[-2\on{dim}U^{-}]\simeq C_{\on{dR}}(\fu^{-}/U^{-}, -\overset{!}{\otimes} \psi_e^!(\on{exp})).
\end{equation}

\noindent Denote by $\pi$ the projection
\[
\fg^*/U^{-}\to \fg^*/G.
\]

\noindent The following diagram commutes (see \cite[Lemma 5.1.9]{beraldo2017loop}):
\[\begin{tikzcd}
	{D(\mathfrak{g}/G)} && {D(\mathfrak{g}^*/G)} && {D(\mathfrak{g}^*/U^{-})} \\
	\\
	{D(\mathfrak{u}^{-}/U^{-})} &&&& {D((\mathfrak{u}^{-})^*/U^{-}).}
	\arrow["{\mathfrak{p}_e^!}"', from=1-1, to=3-1]
	\arrow["{\mathrm{FT}_{\mathfrak{g}}}", from=1-1, to=1-3]
	\arrow["{\mathrm{FT}_{\mathfrak{u}^{-}}}"', from=3-1, to=3-5]
	\arrow["{\pi^!}", from=1-3, to=1-5]
	\arrow["{(-\mathfrak{p}_e^*)_{*,\mathrm{dR}}[2(\mathrm{dim} U^{-}-\mathrm{dim} G)]}", from=1-5, to=3-5]
\end{tikzcd}\]

\noindent Here, $-p_e^*$ is the negative of the dual to the map $p_e: \fu^{-}/U^{-}\to \fg/U^{-}$.

Combining (\ref{eq:FTu}) and the above diagram, we have an isomorphism
\[
\on{coeff}_{\bO}(\cF)\simeq i_e^!\circ (-\mathfrak{p}_e^*)_{*,\mathrm{dR}}\circ \pi^!\circ\on{FT}_{\fg}(\cF)[-2\on{dim}G].
\]

\noindent Identifying $\fg\simeq \fg^*$, we consider $(\fu^{-})^{\perp}$ as a subspace of $\fg$. Denote by $\fp^{\perp}_{-e}: -e+(\fu^{-})^{\perp}/U^{-}\into \fg/U^{-}$ the corresponding embedding. By base-change, we obtain an isomorphism
\begin{equation}\label{eq:coeffO}
\on{coeff}_{\bO}(\cF)\simeq C_{\on{dR}}(-e+(\fu^{-})^{\perp}/U^{-}, (\fp^{\perp}_{-e})^!\circ \pi^!\circ\on{FT}_{\fg}(\cF))[2(\on{dim} U^{-}-\on{dim} G)].
\end{equation}

\noindent By Proposition \ref{p:scalinginv}, $\cF$ is $\bG_m$-monodromic. As such, the characteristic cycle of $\on{FT}_{\fg}(\cF)$ coincides with the characteristic cycle of $\cF$ under the identification $T^*\fg=\fg\times \fg^*=(\fg^*)^*\times \fg^*=T^*\fg^*$. Recall that $-e+(\fu^{-})^{\perp}/U^{-}\simeq \bS_{-e}:=-e+\on{ker}(\on{ad}(f))$ is the Slodowy slice of $-e$ (cf. \cite[Lemma 2.1]{gan2002quantization}).

We have for any closed conical Ad-stable $Z\subset \cN$ in which $\bO$ is maximal:
\[
\bS_{-e}\cap Z=\bS_{-e}\cap \bO=\on{Ad}_{U^{-}}(-e).
\]

\noindent Indeed, this follow from \cite[Prop. 1.3.3 (iii)]{ginzburg2009harish} and the fact that if the Slodowy slice  meets a nilpotent orbit $\bO'$, then $\bO\subset \overline{\bO'}$. Moreover by \emph{loc.cit}, the action of $U^{-}$ on $-e$ is free. Since $\bO$ is maximal in the singular support of $\cF$, we may rewrite the right hand side of (\ref{eq:coeffO}) as
\begin{equation}\label{eq:rewrite1}
i_{-e}^!\circ \on{FT}_{\fg}(\cF)[2(\on{dim} U^{-}-\on{dim} G)],
\end{equation}

\noindent where $i_{-e}: \on{Spec}(k)\to \fg/G$ corresponds to $-e$. The restriction of $\on{FT}_{\fg}(\cF)$ to $\bO$ is lisse, since $\on{FT}_{\fg}(\cF)$ is $G$-equivariant. Moreover, the pullback of $\on{FT}_{\fg}(\cF)$ to $\bO$ is still concentrated in perverse degree $-\on{dim} G$ because $\bO$ is open in the support of $\on{FT}_{\fg}(\cF)$.

Thus, we see that the vector space (\ref{eq:rewrite1}) is concentrated in degree
\[
-\on{dim} G-2(\on{dim} U^{-}-\on{dim} G)+(\on{dim}\bO-\on{dim} G)=0.
\]

\noindent Here, we used that $\on{dim}\bO=2\on{dim}U^{-}$ (see \cite[§1.3]{ginzburg2009harish}). Moreover, since $\on{coeff}_{\bO}(\cF)$ coincides with (\ref{eq:rewrite1}), and $\bO$ is maximal in the support of $\on{FT}_{\fg}(\cF)$, we see that the dimension of $\on{coeff}_{\bO}(\cF)$ is exactly given by $c_{\bO,\cF}$.

\end{proof}

\subsection{Character sheaves on the group}\label{S:5.4}

We now prove Theorem \ref{t:whitcoefflie} for character sheaves on the group.

\subsubsection{} We say that a coherent D-module $\cF\in D(G\overset{\on{ad}}{/}G)$ has nilpotent singular support if its singular support (as a D-module on $G$) is contained in $G\times \cN\subset T^*G$.

\begin{defin}
A character sheaf $\cF\in D(G\overset{\on{ad}}{/}G)$ is a D-module with nilpotent singular support. As before, these are automatically regular holonomic.
\end{defin}

\subsubsection{} Consider the Lagrangian variety
\[
\Lambda_G:=\lbrace (g,x)\in G\times \cN\;\vert \on{Ad}_g(x)=x\rbrace.
\]

\noindent For an Ad-invariant subscheme $Z\subset \cN$, we let
\[
\Lambda_Z=\lbrace (g,x)\in G\times Z\;\vert\; \on{Ad}_g(x)=x\rbrace.
\]

\subsubsection{} For a nilpotent orbit $\bO$, note that $\Lambda_{\bO}$ is smooth. It is not connected in general, however, because the cenralizer of a nilpotent element in $G$ need not be connected. We let $\Lambda_{\bO,0}$ be the connected component of $\Lambda_{\bO}$ containing $(1,\bO)\subset \Lambda_{\bO}$.

\subsubsection{} For a character sheaf $\cF\in D(G\overset{\on{ad}}{/}G)$, we say that an orbit $\bO$ is maximal in the singular support of $\cF$ if either $G\times \bO\cap \on{SS}(\cF)=\emptyset$, or if $\bO$ is maximal with respect to the property that $G\times \bO\cap\on{SS}(\cF)\neq \emptyset$.

In this case, we let $c_{\bO,0,\cF}$ be the multiplicity of $\Lambda_{\bO,0}$ in the characteristic cycle of $\cF$.

\subsubsection{Exponential}\label{ss:exp} Assume now that our ground field $k$ is the complex numbers. We continue to write $\fg,G$ etc. for their underlying complex analytic varieties.\footnote{As opposed to $\fg^{\on{an}},G^{\on{an}}$, say.}

In this case, we have an exponential map\footnote{We write the exponential in bold letters to distinguish it from the exponential D-module $\on{exp}\in D(\bA^1)$.}
\[
\mathbf{exp}: \fg\to G,
\]

\noindent which is $G$-equivariant and a biholomorphism near $0$. Restricting the exponential map to $\cN\subset \fg$, we obtain an isomorphism
\[
\mathbf{exp}: \cN\xrightarrow{\simeq} \cU,
\]

\noindent of algebraic varieties, where $\cU\subset G$ denotes the subvariety of unipotent elements. Henceforth, we will refer to $\cU\subset G$ as the nilpotent cone for convenience.

\subsubsection{} For $e\in\bO$, we denote by $p_e$ the map $U^{-}\overset{\on{ad}}{/}U^{-}\to G\overset{\on{ad}}{/}G$. As in the Lie algebra case, we have a Whittaker functional
\[
\on{coeff}_{\bO}: D(G\overset{\on{ad}}{/}G)\to\on{Vect},\;\; \cF\mapsto C_{\on{dR}}(U^{-}\overset{\on{ad}}{/}U^{-}, p_e^!(\cF)\overset{!}{\otimes} \psi_e^!(\on{exp})).
\]

\noindent When we want to emphasize the difference between the coefficient functor $\on{coeff}_{\bO}$ on $D(\fg/G)$ and $D(G\overset{\on{ad}}{/}G)$, we write $\on{coeff}_{\bO}^{\fg},\on{coeff}_{\bO}^G$, respectively.

\subsubsection{Digression: constructible sheaves} For an algebraic stack $\sY$, we let 
\[
\on{Shv}(\sY)
\]

\noindent be the category of ind-constructible sheaves of $\bC$-vector spaces as in \cite[§F]{arinkin2020stack}. Concretely, for a finite type scheme $S$, we let
\[
\on{Shv}(S):=\on{Ind}(\on{Shv}(S)^{\on{constr}}),
\]

\noindent where $\on{Shv}(S)^{\on{constr}}$ denotes the usual (small) category of constructible sheaves on $S$. If $\sY$ is an algebraic stack, we define
\[
\on{Shv}(\sY)=\underset{S\to \sY}{\on{lim}} \on{Shv}(S),
\]

\noindent where the limit is taken over the category of finite type schemes mapping to $\sY$.

\subsubsection{} Let us define variants of the Whittaker coefficient functors that do not involve the exponential D-module (which we remind is not available in the constructible setting). Let $\bG_m$ act on $\bA^1$ by scaling. The point is to use the fact that $\on{Av}_*^{\bG_m}(\on{exp})$ is regular holonomic to define such Whittaker coefficients for $\bG_m$-monodromic sheaves.

\subsubsection{} Consider the open embedding
\[
j:\bG_m\into \bA^1.
\]

\noindent Let $\pi: \bG_m\to \on{pt}$ be the projection. Denote by $C^*(\bG_m)$ the cohomology complex of $\bG_m$. Since
\[
C^*(\bG_m)\simeq \on{End}_{D(\bG_m)}(\underline{\bC}_{\bG_m}),
\]

\noindent it naturally acts on the constant sheaf $\underline{\bC}_{\bG_m}$. In particular, we get an action
\[
C^*(\bG_m)\curvearrowright j_{*,\on{dR}}(\omega_{\bG_m})\simeq j_{*,\on{dR}}(\underline{\bC}_{\bG_m})[2].
\]

\subsubsection{} Define the functor
\begin{equation}\label{eq:whitconstr}
\widetilde{\on{coeff}}_{\bO}^G: D(G\overset{\on{ad}}{/}G)\to\on{Vect},\;\; \cF\mapsto \on{Hom}_{D(\bA^1)}(j_{*,\on{dR}}(\omega_{\bG_m}),(\psi_e)_{*,\on{dR}}\circ p_e^!(\cF))\underset{C^*(\bG_m)}{\otimes} \bC[1].
\end{equation}

\noindent We have a similarly defined functor
\begin{equation}\label{eq:whitlconstr}
\widetilde{\on{coeff}}_{\bO}^{\fg}: D(\fg/G)\to \on{Vect}
\end{equation}

\noindent defined by replacing $G$ with $\fg$.

\begin{lem}\label{l:widetildetonormal}
We have canonical isomorphisms of functors
\[
\widetilde{\on{coeff}}_{\bO}^G\simeq \on{coeff}_{\bO}^G: D(G\overset{\on{ad}}{/}G)\to \on{Vect,}
\]
\[
\widetilde{\on{coeff}}_{\bO}^{\fg}\simeq \on{coeff}_{\bO}^{\fg}: D(\fg/G)\to \on{Vect.}
\]

\end{lem}

\begin{proof}
Denote by $D(\bA^1)^{\bG_m\on{-mon}}$ the full subcategory of $D(\bA^1)$ generated under colimits by objects in the image of the forgetful functor
\[
D(\bA^1)^{\bG_m}\to D(\bA^1).
\]

\noindent We have a natural map
\[
-\!\on{exp}\to j_{*,\on{dR}}(\omega_{\bG_m})[-1], 
\]

\noindent where $-\!\on{exp}=\bD(\on{exp})$. This induces a functor
\begin{equation}
\on{Hom}_{D(\bA^1)}(j_{*,\on{dR}}(\omega_{\bG_m}),-)\underset{C^*(\bG_m)}{\otimes}\bC[1]\to \on{Hom}_{D(\bA^1)}(-\!\on{exp},-).
\end{equation}

\noindent It it easy to see that the above natural transformation is an isomorphism when restricted to $D(\bA^1)^{\bG_m\on{-mon}}$. Indeed, it suffices to check that they agree on $\omega_{\bG_m}$ (which is killed) and the delta sheaf $\delta_0$ (which is sent to $\bC[-1]$).

Next, let $\cF\in D(G\overset{\on{ad}}{/}G)$. Observe that we have a canonical isomorphism
\[
\on{coeff}_{\bO}^G(\cF)\simeq \on{Hom}_{D(\bA^1)}(-\on{exp},(\psi_e)_{*,\on{dR}}\circ p_e^!(\cF)).
\]

\noindent Thus, we obtain a natural transformation
\begin{equation}\label{eq:widetonormal}
\widetilde{\on{coeff}}_{\bO}^G\to \on{coeff}_{\bO}^G.
\end{equation}

\noindent By the above, the above two functors agree on objects $\cF\in D(G\overset{\on{ad}}{/}G)$ that satisfy that
\[
(\psi_e)_{*,\on{dR}}\circ p_e^!(\cF)\in D(\bA^1)
\]

\noindent is $\bG_m$-monodromic. However, any such $\cF$ has this property. Indeed, $p_e$ factors through $\cN/G$ and any sheaf on $\cN/G$ is $\bG_m$-monodromic. This shows that the natural transformation (\ref{eq:widetonormal}) is an isomorphism.

In exactly the same way, we see that the similarly defined natural transformation
\[
\widetilde{\on{coeff}}_{\bO}^{\fg}\to \on{coeff}_{\bO}^{\fg}
\]

\noindent is an isomorphism.
\end{proof}

\subsubsection{} Note that the functor $\widetilde{\on{coeff}}_{\bO}^G$ is well-defined as a functor on the category $\on{Shv}(G\overset{\on{ad}}{/}G)$, and similarly $\widetilde{\on{coeff}}_{\bO}^{\fg}$ makes sense on $\on{Shv}(\fg/G)$.
\begin{lem}\label{l:exppull}
We have a canonical isomorphism
\[
\widetilde{\on{coeff}}_{\bO}^{\fg}\circ \mathbf{exp}^!\simeq \widetilde{\on{coeff}}_{\bO}^G: \on{Shv}(G\overset{\on{ad}}{/}G)\to \on{Vect.}
\]
\end{lem}

\begin{proof}
This is an immediate consequence of the commutative diagram
\[\begin{tikzcd}
	{\mathbb{A}^1} && {U^{-}\overset{\mathrm{ad}}{/}U^{-}} && {G\overset{\mathrm{ad}}{/}G} \\
	\\
	{\mathbb{A}^1} && {\mathfrak{u}^{-}/U^{-}} && {\mathfrak{g}/G}
	\arrow["{\psi_e}"', from=1-3, to=1-1]
	\arrow["{p_e}", from=1-3, to=1-5]
	\arrow["{\mathrm{id}}", from=3-1, to=1-1]
	\arrow["{\textbf{exp}}", from=3-3, to=1-3]
	\arrow["{d\psi_e}"', from=3-3, to=3-1]
	\arrow["{\mathfrak{p}_e}"', from=3-3, to=3-5]
	\arrow["{\textbf{exp}}"', from=3-5, to=1-5]
\end{tikzcd}\]

\noindent and the fact that the middle vertical arrow is an isomorphism.

\end{proof}

\subsubsection{} We will need the following lemma which follows from Lusztig's generalized Springer correspondence:

\begin{lem}\label{l:reductiontoliealg}
Let $\cF\in D(G\overset{\on{ad}}{/}G)$ be an irreducible perverse character sheaf whose support intersects the nilpotent cone of $G$. There exists a perverse character sheaf $\cF^{\fg}\in D(\fg/G)$ and an open ball $j: V\subset \fg$ containing $0$ such that (under Riemann-Hilbert) $(\mathbf{exp}\circ j)^!(\cF)\simeq j^!(\cF^{\fg})$.

\end{lem}

\begin{proof}
Recall that a character sheaf is cuspidal if parabolic restriction to any proper parabolic subgroup kills the sheaf. There exists a (not necessarily proper) parabolic $P\subset G$ with Levi $M$ and an irreducible perverse cuspidal character sheaf $\cF_M\in D(M\overset{\on{ad}}{/}M)$ such that $\cF$ appears as a subquotient of $\on{Ind}_M^G(\cF_M)$. Here, $\on{Ind}_M^G$ is the functor of pull-push along the correspondence
\[\begin{tikzcd}
	& P\overset{\on{ad}}{/}P \\
	G\overset{\on{ad}}{/}G && M\overset{\on{ad}}{/}M.
	\arrow["p"', from=1-2, to=2-1]
	\arrow["q", from=1-2, to=2-3]
\end{tikzcd}\]

\noindent We remind that $\on{Ind}_M^G$ is t-exact when restricted to character sheaves (see \cite{lusztig1985character}).\footnote{In fact, $\on{Ind}_M^G$ is t-exact on all equivariant sheaves, see \cite{bezrukavnikov2021parabolic}[Thm. 5.4].} Note that by assumption on $\cF$ the support of $\cF_M$ intersects $\cN_M$, the nilpotent cone of $M$. Since $\cF_M$ is irreducible, perverse and cuspidal, there exists a nilpotent orbit $\bO\subset \cN_M$ and a local system $\sL$ on $k:Z(M)^{\circ}\cdot \bO\subset M$ such that $\cF_M\simeq k_{!*}(\sL)$ (see e.g. \cite{mars1989ta}[Thm. 6.3.1(i)]). Here, $Z(M)^{\circ}$ is the connected component of the identity of the center of $M$.

Let $k^{\fm}: Z(\fm)+\bO\into \fm$ denote the inclusion map. Let $\sL^{\fm}$ be the pullback of $\sL$ under the exponential map $\mathbf{exp}: Z(\fm)+\bO\to Z(M)^{\circ}\cdot\bO$. Then $\cF_M^{\fm}:=k^{\fm}_{!*}(\sL^{\fm})\in D(\fm/M)$ is a perverse character sheaf. Indeed, we may find a small ball around $0\in \fm$ such that the preimage of $Z(M)^{\circ}\cdot \bO$ under the exponential map is contained in $Z(\fm)+\bO$. As such, $\cF_M^{\fm}$ is a $\bG_m$-equivariant D-module whose singular support intersects the cotangent fiber at $0\in \fm$ in the nilpotent locus. This implies that the entire singular support of $\cF_M^{\fm}$ is nilpotent (e.g., by observing that its Fourier transform is supported on the nilpotent locus).

Let $'\cF^{\fg}:=\on{Ind}_{\fm}^{\fg}(\cF_M^{\fm})\in D(\fg/G)$ be the parabolic induction of $\cF_M^{\fm}$. Choose $V$ to be a sufficiently small ball around $0$ such that the exponential map is a biholomorphism and such that $\mathbf{exp}^!$ commutes with parabolic induction upon restriction to $V$. Then $(\mathbf{exp}\circ j)^!(\cF)$ is a subquotient of $j^!('\cF^{\fg})$.

Let $\bR_{0<r\leq 1}$ be the positive real numbers in the interval $(0,1]$ considered as a monoid under multiplication. Since $j^!$ gives an equivalence between $\bR_{0<r\leq 1}$-equivariant sheaves on $\fg$ and on $V$, we see that there exists a subquotient $\cF^{\fg}$ of $'\cF^{\fg}$ such that $(\mathbf{exp}\circ j)^!(\cF)\simeq j^!(\cF^{\fg})$.

\end{proof}

\subsubsection{} Let now $k$ again be an arbitrary algebraically closed field of characteristic zero. The following is the main theorem of this subsection.
\begin{thm}\label{t:whitismicrostalkG}
Let $\cF\in D(G\overset{\on{ad}}{/}G)$ be a perverse character sheaf, and let $\bO$ be a maximal orbit in the singular support of $\cF$. Then
\[
\on{coeff}_{\bO}^G(\cF)\in\on{Vect}^{\heartsuit}
\]

\noindent is a finite-dimensional vector space concentrated in degree zero, and moreover
\[
\on{dim}\on{coeff}_{\bO}^G(\cF)=c_{\bO,0,\cF}.
\]

\end{thm}

\begin{proof}
By Lefschetz principle, we may assume that $k=\bC$ is the complex numbers. In particular, we may work with constructible sheaves using Riemann-Hilbert.

\step 

Suppose first that the support of $\cF$ does not intersect the nilpotent cone. Then $\on{coeff}_{\bO}(\cF)$ tautologically vanishes, and moreover $c_{\bO,0,\cF}=0$. Hence the theorem is trivial in this case.

\step Thus, we may assume that the support of $\cF$ intersects the nilpotent cone. Let $\cF^{\fg}\in D(\fg/G)$ be the sheaf and $0\in V\subset \fg$ the ball appearing in Lemma \ref{l:reductiontoliealg}. In particular, the restriction of $\mathbf{exp}$ to $V$ is a biholomorphism. We have a commutative diagram:
\[\begin{tikzcd}
	{\cU\overset{\on{ad}}{/}G} && {G\overset{\mathrm{ad}}{/}G} \\
	\\
	{\cN/G} && {\mathfrak{g}/G.}
	\arrow["{\mathbf{exp}}", from=3-1, to=1-1]
	\arrow[from=1-1, to=1-3]
	\arrow["{\mathbf{exp}}"', from=3-3, to=1-3]
	\arrow[from=3-1, to=3-3]
\end{tikzcd}\]

\noindent By the isomorphism
\[
(\mathbf{exp}\circ j)^!(\cF)\simeq j^!(\cF^{\fg}),
\]

\noindent we get an equality $c_{\bO,0,\cF}=c_{\bO,\cF^{\fg}}$. Since $\cF$ satisfies that $p_e^!(\cF)\in \on{Shv}(U^{-}\overset{\on{ad}}{/}U^{-})\simeq \on{Shv}(\fu^{-}/U^{-})$ is $\bG_m$-monodromic for the natural scaling action on $\fu^{-}$ (because the map $U^-\overset{\on{ad}}{/}U^-\to G\overset{\on{ad}}{/}G$ factors through $\cN/G$), we see that we may compute $\on{coeff}_{\bO}^G(\cF)$ locally around $1\in G$ (and similarly for $\on{coeff}_{\bO}^{\fg}(\cF^{\fg})$). Thus, it follows from Lemma \ref{l:exppull} that
\[
\on{coeff}_{\bO}^G(\cF)\simeq \on{coeff}_{\bO}^{\fg}(\cF^{\fg}).
\]

\noindent We conclude by Theorem \ref{t:whitcoefflie}.

\end{proof}

\subsection{Microstalks} 

\subsubsection{Pseudo-microstalks}\label{s:pseudom} 

Recall that our goal is to show that Whittaker coefficients compute microstalks. That is, while the results in Section \ref{s:charliealg} and \ref{S:5.4} realize the main properties of being a microstalk, we now show that the Whittaker coefficient are themselves microstalks.

\subsubsection{} For this, it is convenient to introduce the notion of a \emph{pseudo-microstalk}. Let $X$ be a smooth stack. Let $\Lambda\subset T^*X$ be a closed conical subset such that all objects of $D_{\Lambda}(X)$ are regular holonomic.

Let us say that a functional
\[
\mu: D_{\Lambda}(X)\to\on{Vect}
\]

\noindent is a \emph{pseudo-microstalk} at $(x,\xi)\in \Lambda^{\on{sm}}$ if:
\begin{itemize}
    \item $\mu$ sends coherent D-modules to (finite complexes of) finite-dimensional vector spaces.

    \item $\mu$ is t-exact and commutes with arbitrary limits.

    \item For coherent $\cF\in D_{\Lambda}(X)$, the Euler characteristic of $\mu(\cF)$ coincides with the multiplicity of $(x,\xi)$ in the characteristic cycle of $\cF$.
    
\end{itemize}

\subsubsection{} Let $Z\subset \cN$ be closed conical and Ad-invariant. Let
\[
D_{\Lambda_Z}(\fg/G)
\]

\noindent be the category of character sheaves on $\fg$ with singular support contained in $\Lambda_Z$. Recall that these are regular holonomic.

Similarly, let
\[
D_{\Lambda_Z}(G\overset{\on{ad}}{/}G)
\]

\noindent be the category of character sheaves on $G$ with singular support contained in $\Lambda_Z$.

\subsubsection{} We claim that $\on{coeff}_{\bO}^{\fg}: D_{\Lambda_Z}(\fg/G)\to \on{Vect}$ commutes with limits (and similarly for $\on{coeff}^G_{\bO}$). By Lemma \ref{l:widetildetonormal}, it suffices to show that $\widetilde{\on{coeff}_{\bO}^{\fg}}: D_{\Lambda_Z}(\fg/G)\to \on{Vect}$ commutes with limits. By Riemann-Hilbert, we have an embedding $D_{\Lambda_Z}(\fg/G)\into \on{Shv}^{\on{all}}(\fg/G)$, where the latter denotes the category of all Betti sheaves; see e.g., \cite{arinkin2020stack}[\S G.7] for the definition of Betti sheaves for stacks. By \emph{loc.cit}, this embedding commutes with limits. Moreover, $\widetilde{\on{coeff}_{\bO}^{\fg}}$ extends to a functional on $\on{Shv}^{\on{all}}(\fg/G)$, and this functional evidently commutes with limits.

\subsubsection{} Let $\bO$ be a maximal nilpotent orbit in $\bO$, and let $e\in\bO$. We may summarize the results of Section \ref{s:charliealg} and \ref{S:5.4} as follows:
\begin{thm}\label{t:summ5354}
The Whittaker coefficient functors
\[
\on{coeff}_{\bO}^{\fg}: D_{\Lambda_Z}(\fg/G)\to \on{Vect},
\]
\[
\on{coeff}_{\bO}^{G}: D_{\Lambda_Z}(G\overset{\on{ad}}{/}G)\to \on{Vect}
\]

\noindent are pseudo-microstalks at $(0,e)$ and $(1,e)$, respectively.
\end{thm}

\subsubsection{} The following allows us to prove that pseudo-microstalks (non-canonically) compute microstalks:
\begin{prop}\label{p:noncaniso}
Let $\mu_1,\mu_2$ be two pseudo-microstalks
\[
\mu_1,\mu_2: D_{\Lambda_Z}(\fg/G)\to \on{Vect.}
\]

\noindent Then there is a non-canonical isomorphism $\mu_1\simeq \mu_2$.

Similarly, any two pseudo-microstalks
\[
\mu_1',\mu_2': D_{\Lambda_Z}(G\overset{\on{ad}}{/}G)\to \on{Vect}
\]

\noindent are non-canonically isomorphic.
\end{prop}

\begin{proof}

\step 
We prove the assertion for $D_{\Lambda_Z}(\fg/G)$, the argument in the group case being similar.

We first define some variants of the category $D_{\Lambda_Z}(\fg/G)$. Consider the small category
\[
D_{\Lambda_Z}(\fg/G)^{\on{access},c}:=D_{\Lambda_Z}(\fg/G)\cap D(\fg/G)^c
\]

\noindent of objects in $D_{\Lambda_Z}(\fg/G)$ that are compact as objects of $D_{\Lambda_Z}(\fg/G)$. Let
\[
D_{\Lambda_Z}(\fg/G)^{\on{access}}=\on{Ind}(D_{\Lambda_Z}(\fg/G)^{\on{access},c}).
\]

\noindent Similarly, let
\[
D_{\Lambda_Z}(\fg/G)^{\on{access},\on{coh}}:=D_{\Lambda_Z}(\fg/G)\cap D(\fg/G)^{\on{coh}}
\]

\noindent be the category of objects $D_{\Lambda_Z}(\fg/G)$ the are coherent as objects of $D(\fg/G)$ (that is, become compact after pulling back along $\fg\to \fg/G$). Let
\[
D_{\Lambda_Z}(\fg/G)^{\on{access},\on{ren}}=\on{Ind}(D_{\Lambda_Z}(\fg/G)^{\on{access},\on{coh}}).
\]

\noindent This category carries a natural t-structure making the fully faithful embedding
\[
D_{\Lambda_Z}(\fg/G)^{\on{access}}\into D_{\Lambda_Z}(\fg/G)^{\on{access},\on{ren}}
\]

\noindent t-exact. We may ind-extend $\mu_1,\mu_2$ to functors
\[
D_{\Lambda_Z}(\fg/G)^{\on{access},\on{ren}}\to\on{Vect.}
\]

\noindent We have a fully faithful embedding
\[
D_{\Lambda_Z}(\fg/G)^{\on{access}}\into D_{\Lambda_Z}(\fg/G)
\]

\noindent with a continuous right adjoint:
\begin{equation}\label{eq:rightadjacc}
D_{\Lambda_Z}(\fg/G)\to D_{\Lambda_Z}(\fg/G)^{\on{access}}.
\end{equation}

\step  In this step, we provide an isomorphism $\mu_1\simeq \mu_2$ as functionals on $D_{\Lambda_Z}(\fg/G)^{\on{access,ren}}$ and hence as functionals on $D_{\Lambda_Z}(\fg/G)^{\on{access}}$.

The Fourier transform provides an equivalence (up to a cohomological shift):
\[
\on{FT}_{\fg}: D_{\Lambda_Z}(\fg/G)\to D(Z/G).
\]

\noindent In particular, the abelian category $D_{\Lambda_Z}(\fg/G)^{\heartsuit}$ only has finitely many simple objects. Let $\sP$ be the direct sum over all simples. By construction, $\sP$ is a compact generator of $D_{\Lambda_Z}(\fg/G)^{\on{access},\on{ren}}$. Let 
\[
A:=\on{End}_{D_{\Lambda_Z}(\fg/G)}(\sP).
\]

\noindent Then $A\in\on{Vect}^{\geq 0}$ is a coconnective dga. To provide an isomorphism $\mu_1\simeq \mu_2$, it suffices to provide an isomorphism $\mu_1(\sP)\simeq \mu_2(\sP)$ of $A$-modules. Since $\mu_1$ and $\mu_2$ are pseudo-microstalks, it follows that the vector spaces $\mu_1(\sP),\mu_2(\sP)$ are both concentrated in degree zero and have the same dimension.

Let $A^0:=H^0(A)$. Note that as an algebra, $A^0$ is just a bunch of copies of the ground field $k$ (one for each simple). By \cite[Prop. 3.9]{dwyer2006duality}, it follows that $\mu_1(\sP)$ and $\mu_2(\sP)$ are isomorphic as $A$-modules if and only if they are so as $A^0$-modules, which is immediate.

\step By step 2, we have an isomorphism $\mu_1\simeq \mu_2$ as functionals on $D_{\Lambda_Z}(\fg/G)^{\on{access}}$. By construction, $D_{\Lambda_Z}(\fg/G)$ identifies with the left completion with respect to the natural t-structure on $D_{\Lambda_Z}(\fg/G)^{\on{access}}$ (see \cite{arinkin2020stack}[\S E.5.5]). Since each $\mu_i$ is t-exact and commutes with limits, we may upgrade the above isomorphism to an isomorphism as functionals on $D_{\Lambda_Z}(\fg/G)$.
\end{proof}

\subsubsection{} For the rest of this subsection, we assume our ground field is the complex numbers.

Recall the notion of microstalks, cf. §\ref{s:microstalk}. Since $(0,e)\in\Lambda_Z$ is a smooth point, we may consider the microstalk functor at $(0,e)$:
\[
\mu_{\bO}^{\fg}: D_{\Lambda_Z}(\fg/G)\to \on{Vect.}
\]

\noindent Similarly, we may consider the microstalk functor at $(1,e)$:
\[
\mu_{\bO}^G: D_{\Lambda_Z}(G\overset{\on{ad}}{/}G)\to \on{Vect.}
\]

\subsubsection{} Unsurprisingly, microstalks are pseudo-microstalks. Combining Theorem \ref{t:summ5354} and Proposition \ref{p:noncaniso}, we obtain:
\begin{thm}\label{t:mainmicro1}
We have isomorphisms
\[
\on{coeff}_{\bO}^{\fg}\simeq \mu_{\bO}^{\fg}: D_{\Lambda_Z}(\fg/G)\to \on{Vect},
\]
\[
\on{coeff}_{\bO}^{G}\simeq \mu_{\bO}^G: D_{\Lambda_Z}(G\overset{\on{ad}}{/}G)\to \on{Vect.}
\]

\noindent That is, the Whittaker coefficient functors compute the microstalk at $(0,e)$ and $(1,e)$, respectively.
\end{thm}

\subsection{Hochschild homology of Whittaker model as a microstalk}\label{S:5.5}
In this section, we prove Theorem \ref{t:catstate} realizing Hochschild homology of the Whittaker model of a nilpotent $G$-category as the Whittaker coefficient of its character sheaf.

\subsubsection{} First, we start on general grounds. Let $\sC$ be a dualizable $G$-category. For a $G$-equivariant endofunctor
\[
F: \sC\to \sC,
\]

\noindent we may consider its image, denoted $\chi_{F}$, under the map
\[
\on{MC}_{\sC}: \on{End}_{D(G)\on{-\textbf{mod}}}(\sC)\simeq \sC\underset{D(G)}{\otimes} \sC^{\vee}\to D(G\overset{\on{ad}}{/}G).
\]

\subsubsection{} On the other hand, $F$ induces an endofunctor
\[
F_{\bO}:=F_{\vert \on{Whit}_{\bO}(\sC)}: \on{Whit}_{\bO}(\sC)\to \on{Whit}_{\bO}(\sC).
\]

\noindent As in §\ref{s:HHendo}, we may consider its Hochschild homology
\[
\on{HH}_*(\on{Whit}_{\bO}(\sC),F_{\bO})\in\on{Vect}.
\]
\begin{lem}\label{l:HHiscoeff}
We have a canonical isomorphism of vector spaces
\[
\on{HH}_*(\on{Whit}_{\bO}(\sC),F_{\bO})\simeq \on{coeff}_{\bO}(\chi_{F}).
\]

\noindent In particular, for $F=\on{id}_{\sC}$, we get an isomorphism
\[
\on{HH}_*(\on{Whit}_{\bO}(\sC))\simeq \on{coeff}_{\bO}(\chi_{\sC}).
\]
\end{lem}

\begin{proof}
We need to show that the following diagram commutes:
\[\begin{tikzcd}
	{\mathrm{End}_{D(G)\mathrm{-\textbf{mod}}}(\sC)} && {D(G\overset{\mathrm{ad}}{/}G)} \\
	\\
	{\mathrm{End}_{\mathrm{DGCat}_{\mathrm{cont}}}(\mathrm{Whit}_{\mathbb{O}}(\sC))} && {\mathrm{Vect}.}
	\arrow["{\mathrm{MC}_{\sC}}", from=1-1, to=1-3]
	\arrow["{\mathrm{coeff}_{\mathbb{O}}}", from=1-3, to=3-3]
	\arrow["{\mathrm{HH}_*(\mathrm{Whit}_{\mathbb{O}}(\sC),-)}"', from=3-1, to=3-3]
	\arrow["{F\mapsto F_{\bO}}"', from=1-1, to=3-1]
\end{tikzcd}\]

\noindent Note that we have a commutative diagram
\[\begin{tikzcd}
	{\mathrm{End}_{D(G)\mathrm{-\textbf{mod}}}(\sC)} && {\mathrm{End}_{D(U^{-})\mathrm{-\textbf{mod}}}(\sC)} \\
	\\
	{D(G\overset{\mathrm{ad}}{/}G)} && {D(U^{-}\overset{\mathrm{ad}}{/}U^{-}).}
	\arrow["{\mathrm{Res}}", from=1-1, to=1-3]
	\arrow["{\mathrm{MC}_{\sC}}", from=1-3, to=3-3]
	\arrow["{\mathrm{MC}_{\sC}}"', from=1-1, to=3-1]
	\arrow["{p_e^!}"', from=3-1, to=3-3]
\end{tikzcd}\]

\noindent Thus, it suffices to show that the following diagram commutes:
\begin{equation}\label{d:final}
\begin{tikzcd}
	{\mathrm{End}_{D(U^{-})\mathrm{-\textbf{mod}}}(\sC)} && {D(U^{-}\overset{\mathrm{ad}}{/}U^{-})} \\
	\\
	{\mathrm{End}_{\mathrm{DGCat}_{\mathrm{cont}}}(\mathrm{Whit}_{\mathbb{O}}(\sC))} && {\mathrm{Vect}.}
	\arrow["{\mathrm{MC}_{\sC}}", from=1-1, to=1-3]
	\arrow["{C_{\on{dR}}(U^{-}\overset{\on{ad}}{/}U^{-}, \psi_e^!(\on{exp})\overset{!}{\otimes}-)}", from=1-3, to=3-3]
	\arrow["{\mathrm{HH}_*(\mathrm{Whit}_{\mathbb{O}}(\sC),-)}"', from=3-1, to=3-3]
	\arrow["{F\mapsto F_{\bO}}"', from=1-1, to=3-1]
\end{tikzcd}
\end{equation}

\noindent Note that under the canonical self-duality
\[
(\mathrm{End}_{D(U^{-})\mathrm{-\textbf{mod}}}(\sC))^{\vee}\simeq (\sC^{\vee}\underset{D(U^{-})}{\otimes} \sC)^{\vee}\simeq \sC^{\vee}\underset{D(U^{-})}{\otimes}(\sC^{\vee})^{\vee}\simeq \mathrm{End}_{D(U^{-})\mathrm{-\textbf{mod}}}(\sC),
\]

\noindent the dual to the functor
\[
\on{MC}_{\sC}: \mathrm{End}_{D(U^{-})\mathrm{-\textbf{mod}}}(\sC)\to D(U^{-}\overset{\on{ad}}{/}U^{-})
\]

\noindent goes to the functor
\[
\on{MC}_{\sC}^{\vee}: D(U^{-}\overset{\on{ad}}{/}U^{-})\simeq (D(U^{-}\overset{\on{ad}}{/}U^{-}))^{\vee}\to \mathrm{End}_{D(U^{-})\mathrm{-\textbf{mod}}}(\sC),\;\; \cF\mapsto (c\mapsto \cF\star c).
\]

\noindent Moreover, the dual to the functor
\[
\mathrm{End}_{D(U^{-})\mathrm{-\textbf{mod}}}(\sC)\xrightarrow{F\mapsto F_{\bO}}\mathrm{End}_{\mathrm{DGCat}_{\mathrm{cont}}}(\mathrm{Whit}_{\mathbb{O}}(\sC))
\]

\noindent goes (under the canonical self-dualities of the source and target) to the functor
\[
\mathrm{End}_{\mathrm{DGCat}_{\mathrm{cont}}}(\mathrm{Whit}_{\mathbb{O}}(\sC))\to \mathrm{End}_{D(U^{-})\mathrm{-\textbf{mod}}}(\sC),\;\; F'\mapsto \on{oblv}^{\psi_e}\circ F'\circ \on{Av}_*^{\psi_e}.
\]

\noindent Here, $\on{oblv}^{\psi_e}$ denotes the forgetful functor $\on{Whit}_{\bO}(\sC)\to \sC$. Indeed, this follows from the fact that under the $U^{-}\times U^{-}$-action on $\sC\otimes \sC^{\vee}$, we have
\[
\mathrm{End}_{D(U^{-})\mathrm{-\textbf{mod}}}(\sC)\simeq (\sC\otimes \sC^{\vee})^{\Delta U^{-}},\;\; \mathrm{End}_{\mathrm{DGCat}_{\mathrm{cont}}}(\mathrm{Whit}_{\mathbb{O}}(\sC))\simeq (\sC\otimes \sC^{\vee})^{U^{-}\times U^{-}, \psi_e\times \psi_e}.
\]

\noindent Passing to dual functors in the diagram (\ref{d:final}), it suffices to show that the following diagram commutes:
\[\begin{tikzcd}
	{\mathrm{End}_{D(U^{-})\mathrm{-\textbf{mod}}}(\sC)} && {D(U^{-}\overset{\mathrm{ad}}{/}U^{-})} \\
	\\
	{\mathrm{End}_{\mathrm{DGCat}_{\mathrm{cont}}}(\mathrm{Whit}_{\mathbb{O}}(\sC))} && {\mathrm{Vect}.}
	\arrow["{(\mathrm{MC}_{\sC})^{\vee}}"', from=1-3, to=1-1]
	\arrow["{k\mapsto \psi_e^!(\mathrm{exp})}"', from=3-3, to=1-3]
	\arrow["{k\mapsto \mathrm{id_{\mathrm{Whit}_{\mathbb{O}}(\sC)}}}", from=3-3, to=3-1]
	\arrow["{F'\mapsto \mathrm{oblv}^{\psi_e}\circ F'\circ\mathrm{Av}_*^{\psi_e}}", from=3-1, to=1-1]
\end{tikzcd}\]

\noindent 

\noindent However, the commutativity of the above diagram is clear, since the endofunctor $\on{oblv}^{\psi_e}\circ \on{Av}_*^{\psi_e}$ coincides with acting by $\psi_e^!(\on{exp})$.

\end{proof}

\subsubsection{} Putting it together, we get:
\begin{thm}\label{t:finalboss?}
Let $\sC$ be a nilpotent $G$-category and $\bO$ a maximal nilpotent orbit contained in $\on{SS}(\sC)$. Then 
\[
\on{HH}_*(\on{Whit}_{\bO}(\sC))=\on{coeff}_{\bO}(\chi_{\sC}).
\]

\noindent Moreover, $\on{coeff}_{\bO}(\chi_{\sC})$ coincides with the microstalk of $\chi_{\sC}$ at $(1,e)$.
\end{thm}

\begin{proof}
By Theorem \ref{t:nilp}, we may assume that $\sC$ is $(B,\chi)$-generated. Since $\bO$ is also maximal in $\on{SS}(\chi_{\sC})$, the theorem now follows from a combination Lemma \ref{l:HHiscoeff} and Theorem \ref{t:mainmicro1}.
\end{proof}

\section{Finite-dimensional modules of $\sW$-algebras}\label{S:6}

In this section, we identify the rationalized Grothendieck group of finite-dimensional representations of a $\sW$-algebra with fixed regular central character as a certain subspace of the top cohomology of the Springer fiber. This was originally proved by Losev-Ostrik \cite[Thm. 7.4(iii)]{losev2014classification} for trivial central character and by Bezrukavnikov-Losev \cite[Thm. 5.2(2)]{bezrukavnikov2018dimension} for arbitrary central character.

Our proof is self-contained. The idea is to use that the category of finite-dimensional representations of a $\sW$-algebra admits a characterization in terms of singular support and then leverage the results we have obtained so far. In particular, we do not appeal to methods in positive characteristic.

\subsection{Statement of theorem}

\subsubsection{Renormalized Hecke categories}

We will work with a renormalized version of the Hecke category. Fix a character sheaf on $T$ corresponding to $\chi\in \ft^*/X^{\bullet}(T)(k)$. Let
\[
\sH_{G,\chi}^{\on{coh}}\subset \sH_{G,\chi}:=D(B,\chi\backslash G/B,\chi)
\]

\noindent be the small subcategory of  coherent objects. That is, objects that become compact after applying the forgetful functor
\[
\sH_{G,\chi}\xrightarrow{\on{oblv}^{B,\chi}} D(G/B,\chi).
\]

\subsubsection{} We let
\[
\sH_{G,\chi}^{\on{ren}}:=\on{Ind}(\sH_{G,\chi}^{\on{coh}})
\]

\noindent be its ind-completion and refer to it as the \emph{renormalized Hecke category}.
\begin{example}
The monoidal unit in $\sH_{G,\chi}$ is not compact but defines a compact object in $\sH_{G,\chi}^{\on{ren}}$.
\end{example}

\subsubsection{} Since compact objects of $\sH_{B,\chi}$ are coherent, we have a fully faithful embedding
\[
\sH_{G,\chi}\into \sH_{G,\chi}^{\on{ren}}.
\]

\noindent The functor admits a continuous right adjoint 
\begin{equation}\label{eq:killnilp}
\sH_{G,\chi}^{\on{ren}}\to \sH_{G,\chi}
\end{equation}

\noindent defined by ind-extending the embedding $\sH_{G,\chi}^{\on{coh}}\into \sH_{G,\chi}$.

Note that $\sH_{G,\chi}^{\on{ren}}$ inherits a t-structure making the functor (\ref{eq:killnilp}) t-exact.

\subsubsection{}\label{s:chern} We want to consider the Grothendieck group of $\sH_{G,\chi}^{\on{ren}}$ and its relation to Hochschild homology. Let us consider this on general grounds.

For a dualizable DG-category $\sC$, write $\sC^c\subset \sC$ for its compact objects. If $c\in \sC^c$, we get a functor
\[
\on{Vect}\to \sC,\;\; k\mapsto c.
\]

\noindent Since $c$ is compact, the above functor admits a continuous right adjoint. Thus, from the functoriality of traces, we get a map of vector spaces
\[
k=\on{HH}_*(\on{Vect})\to \on{HH}_*(\sC).
\]

\noindent We denote by $[c]\in \on{HH}_0(\sC)$ the corresponding element and refer to it as the \emph{Chern character} of $c$. 

\subsubsection{} If $\sC$ is moreover compactly generated, we may consider its Grothendieck group 
\[
K_0(\sC):=K_0(\sC^c).
\]

\noindent We write $K_0(\sC)_k:=K_0(\sC)\underset{\bZ}{\otimes} k$. In this case, we obtain the Chern character map
\begin{equation}\label{eq:chern''}
K_0(\sC)_k\to \on{HH}_0(\sC), \;\; c\mapsto [c].
\end{equation}

\subsubsection{} Let $\sC=\sH_{G,\chi}^{\on{ren}}$ and consider the map
\begin{equation}\label{eq:chernabc}
K_0(\sH_{G,\chi}^{\on{ren}})_k\to \on{HH}_0(\sH_{G,\chi}^{\on{ren}}).
\end{equation}

\noindent Since $\sH_{G,\chi}^{\on{ren}}$ is rigid, $\on{HH}_*(\sH_{G,\chi}^{\on{ren}})$ is a dga, and the map (\ref{eq:chernabc}) is a map of algebras. 

\subsubsection{} Let $W$ be the Weyl group of $G$, and let 
\[
W_{[\chi]}=\lbrace w\in W\;\vert\; w\chi-\chi=0\in \ft^*/X^{\bullet}(T)\rbrace
\]

\noindent be the integral Weyl group of $\chi$. We remind (see e.g. \cite{lusztig2020endoscopy}) that we have an isomorphism of algebras
\begin{equation}\label{eq:simofalgs}
k[W_{[\lambda]}]\simeq K_0(\sH_{G,\chi}^{\on{ren}})_k,
\end{equation}

\noindent sending $w$ to $j_{w,\chi,!}$ (the latter being defined in §\ref{s:jwchi}).

On the other hand, we are interested in the Hochschild homology of $\sH_{G,\chi}^{\on{ren}}$:
\begin{lem}\label{l:HHisgroupalg}
As a vector space, we have:
\[
\on{HH}_*(\sH_{G,\chi}^{\on{ren}})=k[W_{[\chi]}]\otimes \on{Sym}(\ft^*[-1]\oplus \ft^*[-2]).
\]

\noindent Moreover, the Chern character map
\[
k[W_{[\chi]}]\simeq K_0(\sH_{G,\chi}^{\on{ren}})_k\to \on{HH}_*(\sH_{G,\chi}^{\on{ren}})
\]

\noindent induces an isomorphism $k[W_{[\chi]}]\simeq \on{HH}_0(\sH_{G,\chi}^{\on{ren}})$.

\end{lem}

\begin{rem}\label{r:celinedionatolympics?}
Although we will not need it for our purposes, one can show that there is a canonical isomorphism
\[
\on{HH}_*(\sH_{G,\chi}^{\on{ren}})=W_{[\chi]}\ltimes \on{Sym}(\ft^*[-1]\oplus \ft^*[-2])
\]

\noindent of dg-algebras, where $W_{[\chi]}$ acts on $\on{Sym}(\ft^*[-1]\oplus \ft^*[-2])$ via the diagonal action of $W_{[\chi]}$ on $\ft^*[-1]\oplus \ft^*[-2]$ in the usual way. We include a sketch of this proof in Section \ref{S:actiononspringerfibrelaks}, see Lemma \ref{l:strongversion}. An alternative proof appears in \cite{li2023derivedyan}.

\end{rem}

\begin{proof}

\step Let $D(T,\chi\backslash T/T,\chi)^{\on{ren}}$ be the renormalized Hecke category for the torus. By construction, $\chi^{\vee}\in D(T,\chi\backslash T/T,\chi)^{\on{ren}}$ is a compact generator,\footnote{Note that it is the dual sheaf $\chi^{\vee}$ that appears here, and not $\chi$ itself. This is analogous to the fact that for a multiplicative function $\chi: B\rightarrow k$, the function $\chi^{\vee}(x)=\chi(x)^{-1}$ is right-invariant for the action $B: \on{Fun}(B)\rightarrow \mathrm{Fun}(B),\;\; (fb)(x):=f(xb^{-1})$.} and so we get an equivalence
\[
D(T,\chi\backslash T/T,\chi)^{\on{ren}}\simeq \on{End}_{D(T,\chi\backslash T/T,\chi)^{\on{ren}}}(\chi^{\vee})\on{-mod}\simeq C^*(\bB T)\on{-mod}\simeq H^*(\bB T)\on{-mod.}
\]

\noindent The last equivalence follows by formality of $C^*(\bB T)$.

\step Note that for $w\in W$, the similarly defined renormalized category
\[
D(B,\chi\backslash BwB/B,\chi)^{\on{ren}}
\]

\noindent is non-zero if and only if $w\in W_{[\chi]}$. In this case, we have an equivalence (canonical up to tensoring with a $!$-fibre of $\chi$, see §\ref{s:noncanequiv3}):
\[
D(B,\chi\backslash BwB/B,\chi)^{\on{ren}}\simeq D(T,\chi\backslash T/T,\chi)^{\on{ren}}.
\]

\noindent Thus, for each $w\in W_{[\chi]}$, let $j_{w,\chi,!}\in\sH_{G,\chi}^{\on{ren}}$ be the $!$-pushforward of $\chi^{\vee}\in D(T,\chi\backslash T/T,\chi)^{\on{ren}}\simeq D(B,\chi\backslash BwB/B,\chi)^{\on{ren}}$ under the locally closed embedding
\[
B\backslash BwB/B\into B\backslash G/B.
\]

\noindent Let $\sC_{\leq w}$ be the full subcategory of $\sH_{G,\chi}^{\on{ren}}$ generated by $j_{w',\chi,!}$ for $w'\leq w$ in the Bruhat order. Then the $\sC_{\leq w}$ stratify $\sH_{G,\chi}^{\on{ren}}$. By definition, the inclusion $\sC_{\leq w}\into \sH_{G,\chi}$ admits a continuous right adjoint. Moreover, for $w'\leq w$, the left orthogonal to $\sC_{\leq w'}$ in $\sC_{\leq w}$ is generated by $j_{w'',\chi,!}$ for all $w''\leq w$ with $w''\nleq w'$. Thus, the inclusion $\sC_{\leq w}\into \sH_{G,\chi}$ admits a left adjoint as well. It follows that as a vector space,
\[
\on{HH}_*(\sH_{G,\chi}^{\on{ren}})\simeq \underset{w\in W_{[\chi]}}{\bigoplus} \on{HH}_*(D(T,\chi\backslash T/T,\chi)^{\on{ren}}).
\]

\noindent By Step 1, we have $\on{HH}_*(D(T,\chi\backslash T/T,\chi)^{\on{ren}})=\on{HH}_*(H^*(\bB T))$. Since $H^*(\bB T)\simeq \on{Sym}(\ft^*[-2])$, we get
\[
\on{HH}_*(H^*(\bB T))\simeq \on{Sym}(\ft^*[-1]\oplus \ft^*[-2]).
\]

\noindent Thus, we obtain an isomorphism of vector spaces
\[
\on{HH}_*(\sH_{G,\chi}^{\on{ren}})\simeq k[W_{[\chi]}]\otimes \on{Sym}(\ft^*[-1]\oplus \ft^*[-2]).
\]

\noindent Moreover, observe that the embedding
\[
k[W_{[\chi]}]\into k[W_{[\chi]}]\otimes \on{Sym}(\ft^*[-1]\oplus \ft^*[-2])\simeq \on{HH}_*(\sH_{G,\chi}^{\on{ren}})
\]

\noindent canonically identifies with the Chern character map under the isomorphism (\ref{eq:simofalgs}).

\end{proof}

\subsubsection{Finite-dimensional modules of $\sW$-algebras} Fix a nilpotent element $e\in\bO$ along with an $\mathfrak{sl}_2$-triple $(e,h,f)$ and a Lagrangian subspace $\ell\subset \fg(-1)$ as in §\ref{s:genwhit}. We let $\sW_e$ be the corresponding $\sW$-algebra (see \cite[§2.3]{losev2010finite}). The center of $\sW_e$ identifies with the center of $U(\fg)$. As such, for $\lambda\in \ft^*$, we may consider the DG-category
\[
\sW_{e,\lambda}\on{-mod}
\]

\noindent of $\sW_e$-modules with central character $\lambda$. Henceforth, we assume that $\lambda$ is regular.

\subsubsection{} We have a full subcategory
\[
\sW_{e,\lambda}\on{-mod}^{\on{fin}}\subset \sW_{e,\lambda}\on{-mod}
\]

\noindent consisting of objects $M\in \sW_{e,\lambda}\on{-mod}$ such that $H^n(M)$ is a filtered colimit of finite-dimensional representations for each $n\in \bZ$. We write 
\[
\sW_{e,\lambda}\on{-mod}^{\on{fin},\heartsuit}
\]

\noindent for the heart of the natural t-structure on $\sW_{e,\lambda}\on{-mod}^{\on{fin}}$.

\subsubsection{} Let
\[
K_0(\sW_{e,\lambda}\on{-mod})_k=K_0(\sW_{e,\lambda}\on{-mod})\underset{\bZ}{\otimes}k.
\]

\noindent We have a Chern character map
\begin{equation}\label{eq:cherncba}
K_0(\sW_{e,\lambda}\on{-mod})_k\to \on{HH}_0(\sW_{e,\lambda}\on{-mod}).
\end{equation}

\subsubsection{}\label{s:justificationfin} We let
\[
K_0(\sW_{e,\lambda}\on{-mod}^{\on{fin},\heartsuit})_k=K_0(\sW_{e,\lambda}\on{-mod}^{\on{fin},\heartsuit})\underset{\bZ}{\otimes} k,
\]

\noindent where $K_0(\sW_{e,\lambda}\on{-mod}^{\on{fin},\heartsuit})$ is the Grothendieck group of the abelian category of finite-dimensional $\sW_{e,\lambda}$-modules.\footnote{We note that it is not clear that $K_0(\sW_{e,\lambda}\on{-mod}^{\on{fin},\heartsuit})$ should coincide with $K_0(\sW_{e,\lambda}\on{-mod}^{\on{fin}})$, since the category $\sW_{e,\lambda}\on{-mod}^{\on{fin}}$ contains compact objects that are unbounded from below.}

If $M\in \sW_{e,\lambda}\on{-mod}$ is a bounded complex each of whose cohomologies are finitely generated, it defines a compact object in $\sW_{e,\lambda}\on{-mod}$.\footnote{This follows from a combination of Skryabin's equivalence and Beilinson-Bernstein localization, using that $\lambda$ is regular.} Thus, we get a natural map
\begin{equation}\label{eq:fintononfin}
K_0(\sW_{e,\lambda}\on{-mod}^{\on{fin},\heartsuit})_k\to K_0(\sW_{e,\lambda}\on{-mod})_k.
\end{equation}

\subsubsection{} Consider the Springer resolution:
\[
\pi: \widetilde{\cN}\to \cN.
\]

\noindent Denote by $\cB_e:=\pi^{-1}(e)$ the Springer fiber of $e\in \cN$. We let $H^*(\cB_e)$ be the cohomology of $\cB_e$ and write $H^{\on{top}}(\cB_e):=H^{2\on{dim} \cB_e}(\cB_e)$.

We have the following lemma whose proof we provide at the end of this subsection:
\begin{lem}\label{l:HHiscoho}
There is a canonical identification
\[
\on{HH}_*(\sW_{e,\lambda}\on{-mod})=H^*(\cB_e)[2\on{dim}\cB_e].
\]

\noindent In particular
\[
\on{HH}_0(\sW_{e,\lambda}\on{-mod})=H^{\on{top}}(\cB_e).
\]
\end{lem}

\begin{rem}
The second statement was proven in \cite{etingof2010traces}. Our proof, while similar in spirit, is essentially an immediate consequence of Lemma \ref{l:HHiscoeff}.
\end{rem}

\subsubsection{} Let $\chi=\chi_{\lambda}$ be the character sheaf on $T$ corresponding to $\lambda$. We write $W_{[\lambda]}$ for the corresponding integral Weyl group. By Skryabin's equivalence and Beilinson-Bernstein localization, we get an equivalence
\begin{equation}\label{eq:Skryabin+BB}
\sW_{e,\lambda}\on{-mod}\simeq D(U^{-},\psi_e\backslash G/B,\chi).
\end{equation}

\noindent In particular, $\sW_{e,\lambda}\on{-mod}$ carries a right action of $\sH_{G,\chi}^{\on{ren}}$. Moreover, the action map
\[
\sW_{e,\lambda}\on{-mod}\otimes \sH_{G,\chi}^{\on{ren}}\to \sW_{e,\lambda}\on{-mod}
\]

\noindent admits a continuous right adjoint. Passing to Hochschild homology and applying Lemma \ref{l:HHisgroupalg} and Lemma \ref{l:HHiscoho}, we get an action:
\begin{equation}\label{eq:10/09:1716}
H^*(\cB_e)\simeq\on{HH}_*(\sW_{e,\lambda}\on{-mod})[-2\on{dim}\cB_e]\curvearrowleft \on{HH}_*(\sH_{G,\chi})\overset{\on{Lem.}\;\ref{l:HHisgroupalg}}{\leftarrow}k[W_{[\lambda]}].
\end{equation}

\noindent In particular, we get an action:
\begin{equation}\label{eq:actHH_Ow-mod}
H^{\on{top}}(\cB_e)\curvearrowleft k[W_{[\lambda]}].
\end{equation}

\noindent By Corollary \ref{c:finalspringerboss} below, this coincides with the usual Springer action, normalized so that for $e=0$, we get the trivial representation.

\subsubsection{} We claim that we also have an action
\begin{equation}
K_0(\sW_{e,\lambda}\on{-mod}^{\on{fin},\heartsuit})\curvearrowleft k[W_{[\lambda]}],
\end{equation}

\noindent which we now define.

Let $\bO$ be the nilpotent orbit containing $e$. Let
\[
\sH_{G,\chi,\overline{\bO}}^{\on{ren}}:=\sH_{G,\chi}^{\on{ren}}\underset{\sH_{G,\chi}}{\times}\sH_{G,\chi,\overline{\bO}}
\]

\noindent be the full subcategory of $\sH_{G,\chi}^{\on{ren}}$ consisting of objects that land in $\sH_{G,\chi,\overline{\bO}}$ under the map $\sH_{G,\chi}^{\on{ren}}\to \sH_{G,\chi}$. Note that the inclusion
\[
\sH_{G,\chi,\overline{\bO}}^{\on{ren}}\into \sH_{G,\chi}^{\on{ren}}
\]

\noindent commutes with limits, and hence admits a left adjoint. Indeed, this follows from the corresponding fact for the embedding $\sH_{G,\chi,\overline{\bO}}\into \sH_{G,\chi}$ (cf. Proposition \ref{l:stable}) and the fact that the map $\sH_{G,\chi}^{\on{ren}}\to \sH_{G,\chi}$ commutes with limits, being a right adjoint.

\subsubsection{} At the end of this subsection, we will show that:
\begin{prop}\label{p:keytoeverythingrs}
We have an identification
\[
\sW_{e,\lambda}\on{-mod}^{\on{fin}}\simeq \sW_{e,\lambda}\on{-mod}\underset{\sH_{G,\chi}^{\on{ren}}}{\otimes} \sH_{G,\chi,\overline{\bO}}^{\on{ren}}
\]

\noindent as subcategories of $\sW_{e,\lambda}\on{-mod}$.

\end{prop}

\noindent As such, $\sW_{e,\lambda}\on{-mod}^{\on{fin}}$ similarly carries a right $\sH_{G,\chi}^{\on{ren}}$-action. Using that $K_0(\sW_{e,\lambda}\on{-mod}^{\on{fin},\heartsuit})_k$ identifies with $K_0(-)_k$ of the small category consisting of bounded complexes of finite-dimensional $\sW_{e,\lambda}$-modules, we get an action
\begin{equation}\label{eq:intweylfin}
K_0(\sW_{e,\lambda}\on{-mod}^{\on{fin},\heartsuit})_k\curvearrowleft K_0(\sH_{G,\chi}^{\on{ren}})_k\overset{(\ref{eq:simofalgs})}{\simeq} k[W_{[\lambda]}].
\end{equation}

\subsubsection{} Let 
\[
K_0(\sH_{G,\chi}^{\on{ren}})_{\overline{\bO},k}
\]

\noindent be the Grothendieck group of the abelian subcategory of $\sH_{G,\chi,\overline{\bO}}^{\on{ren},\heartsuit}$ consisting of coherent D-modules, base-changed to $k$. That is, $K_0(\sH_{G,\chi}^{\on{ren}})_{\overline{\bO},k}$ has a basis consisting of simple sheaves in $\sH_{G,\chi}^{\on{ren},\heartsuit}$ whose singular support is contained in $\overline{\bO}$.

\subsubsection{} Consider the subcategory 
\[
\sH_{G,\chi,\partial\bO}^{\on{ren}}\subset \sH_{G,\chi,\overline{\bO}}^{\on{ren}},
\]

\noindent where $\partial\bO=\overline{\bO}\setminus\bO$. We define the vector space
\[
K_0(\sH_{G,\chi}^{\on{ren}})_{\partial\bO,k}
\]

\noindent similarly. We have an injection
\[
K_0(\sH_{G,\chi}^{\on{ren}})_{\partial\bO,k}\into K_0(\sH_{G,\chi}^{\on{ren}})_{\overline{\bO},k}.
\]

\noindent Let
\[
K_0(\sH_{G,\chi}^{\on{ren}})_{\bO,k}:=K_0(\sH_{G,\chi}^{\on{ren}})_{\overline{\bO},k}/K_0(\sH_{G,\chi}^{\on{ren}})_{\partial\bO,k}.
\]

\subsubsection{} We have now come to the statement of the main theorem of this section. Using the isomorphism $k[W_{[\lambda]}]=K_0(\sH_{G,\chi}^{\on{ren}})$ and the action (\ref{eq:actHH_Ow-mod}), we may consider the subrepresentation
\[
H^{\on{top}}(\cB_e)\underset{k[W_{[\lambda]}]}{\otimes} K_0(\sH_{G,\chi}^{\on{ren}})_{\bO,k}\subset H^{\on{top}}(\cB_e)\underset{k[W_{[\lambda]}]}{\otimes} K_0(\sH_{G,\chi}^{\on{ren}})_k=H^{\on{top}}(\cB_e).
\]

\noindent Moreover, combining Lemma \ref{l:HHiscoho} with the Chern character (\ref{eq:cherncba}), we get a map:
\begin{equation}\label{eq:1009}
K_0(\sW_{e,\lambda}\on{-mod}^{\on{fin},\heartsuit})_k\xrightarrow{(\ref{eq:fintononfin})} K_0(\sW_{e,\lambda}\on{-mod})_k\to H^{\on{top}}(\cB_e).
\end{equation}
\begin{thm}\label{t:classification}
The map
\[
K_0(\sW_{e,\lambda}\on{-mod}^{\on{fin},\heartsuit})_k\to H^{\on{top}}(\cB_e)
\]

\noindent maps isomorphically onto 
\[
H^{\on{top}}(\cB_e)\underset{k[W_{[\lambda]}]}{\otimes} K_0(\sH_{G,\chi}^{\on{ren}})_{\bO,k.}
\]

\noindent Moreover, the resulting isomorphism
\[
K_0(\sW_{e,\lambda}\on{-mod}^{\on{fin}})_k\simeq H^{\on{top}}(\cB_e)\underset{k[W_{[\lambda]}]}{\otimes} K_0(\sH_{G,\chi}^{\on{ren}})_{\bO,k}
\]

\noindent is $k[W_{[\lambda]}]$-linear for their respective right actions.

\end{thm}

\begin{rem}\label{r:Dodd}
A map
\[
K_0(\sW_{e,\lambda}\on{-mod})_k\to H^{\on{top}}(\cB_e)
\]

\noindent was constructed by Dodd in \cite{dodd2014injectivity} by associating to a $W_{e,\lambda}$-module a coherent sheaf on the resolution $\widetilde{\cS}_e$ of the nilpotent Slodowy slice via localization onto a quantization of $\widetilde{\cS}_e$ and then degenerating. To get a class in $H^{\on{top}}(\cB_e)$, one takes the corresponding characteristic cycle. Dodd showed that the resulting map
\[
K_0(\sW_{e,\lambda}\on{-mod}^{\on{fin},\heartsuit})_k\to K_0(\sW_{e,\lambda}\on{-mod})_k\to H^{\on{top}}(\cB_e)
\]

\noindent is injective by reducing mod $p$. We expect our map (\ref{eq:1009}) to coincide with that of \cite{dodd2014injectivity}.
\end{rem}

\subsubsection{} We end this subsection by proving Lemma \ref{l:HHiscoho} and Proposition \ref{p:keytoeverythingrs}:

\vspace{2mm}

\begin{proof}[Proof of Lemma \ref{l:HHiscoho}]

\noindent We use the notation of Section \ref{S:5.5}. Recall the equivalence (\ref{eq:Skryabin+BB}):
\[
\sW_{e,\lambda}\on{-mod}\simeq D(U^{-},\psi_e\backslash G/B,\chi).
\]

\noindent By Lemma \ref{l:HHiscoeff} and Example \ref{e:Springer}, we get
\[
\on{HH}_*(\sW_{e,\lambda}\on{-mod})\simeq \on{coeff}^G_{\bO}(\on{Spr}_{\chi}).
\]

\noindent We remind that $\on{Spr}_{\chi}\in D(G\overset{\on{ad}}{/}G)$ is the $\chi$-twisted Grothendieck Springer sheaf. Let $\on{Spr}^{\fg}$ be the Grothendieck Springer sheaf on $\fg/G$. That is, the image of $\omega_{\fb/B}$ under pushforward along
\[
\fb/B\to \fg/G.
\]

\noindent Since the restrictions of $\on{Spr}_{\chi}$ and $\on{Spr}^{\fg}$ to the unipotent/nilpotent locus agree, we obtain a canonical isomorphism

\begin{equation}\label{eq:whitspringerequality}
\on{coeff}_{\bO}^G(\on{Spr}_{\chi})\simeq \on{coeff}_{\bO}^{\fg}(\on{Spr}^{\fg}).
\end{equation}

\noindent By (\ref{eq:coeffO}), we have an isomorphism
\begin{equation}\label{eq:nudetnok!}
\on{coeff}_{\bO}(\on{Spr}^{\fg})\simeq C_{\on{dR}}(\cS_{-e}, f_{-e}^!\circ\on{FT}_{\fg}(\on{Spr}^{\fg}))[2(\on{dim} U^{-}-\on{dim} G)].
\end{equation}

\noindent Here:
\begin{itemize}
    \item The sheaf $\on{FT}_{\fg}(\on{Spr}^{\fg})$ naturally lives on $\cN/G$, since $\on{Spr}^{\fg}$ has nilpotent singular support.

    \item $\cS_{-e}=(-e+\on{ker}(\on{ad}(f)))\cap \cN$ is the nilpotent Slodowy slice of $-e$.

    \item $f_{-e}$ is the natural map $\cS_{-e}\into \cN\to \cN/G$.
\end{itemize}

\noindent Let $i_{\cN}:\cN/G\into \fg/G$ be the inclusion, and let $r$ denote the rank of $G$. Recall (see e.g. \cite[Lemma 2.2]{achar2014weyl})\footnote{Note that there is no cohomological shift in the statement of \emph{loc.cit}. This is because their Fourier transform differs from our by a shift.} that we have an equivalence
\[
\on{FT}_{\fg}(\on{Spr}^{\fg})\simeq i_{\cN}^!(\on{Spr}^{\fg})[r+\on{dim}G].
\]

\noindent Let $\pi:\widetilde{\cN}\to \cN$ be the Springer resolution, and let $g_{-e}$ be the composition 
\[
\widetilde{\cS}_{-e}:=\pi^{-1}(\cS_{-e})\into \widetilde{\cN}\to \widetilde{\cN}/G\simeq \fn/B\into \fb/B.
\]

\noindent Applying base change, we get:
\begin{equation}\label{eq:resoslodowy}
\on{coeff}_{\bO}(\on{Spr}_{\chi}^{\fg})\simeq C_{\on{dR}}(\widetilde{\cS}_{-e}, \omega_{\widetilde{\cS}_{-e}})[2\on{dim} U^{-}-\on{dim} G+r].
\end{equation}

\noindent Note that $\omega_{\widetilde{\cS}_{-e}}=\underline{k}_{\widetilde{\cS}_{-e}}[2(\on{dim}\cN-\on{dim}\bO)]$, since $\widetilde{\cS}_{-e}$ is smooth of dimension $\on{dim}\cN-\on{dim}\bO$ (see \cite[Prop. 2.1.2]{ginzburg2009harish}). Thus, the right hand side of (\ref{eq:resoslodowy}) identifies with
\begin{equation}\label{eq:gettingcloser}
H^*(\widetilde{\cS}_{-e})[2\on{dim} U^{-}-\on{dim} G+r+2(\on{dim}\cN-\on{dim}\bO)].
\end{equation}

\noindent Recall that $2\on{dim} U^{-}=\on{dim}\bO$. Thus the above shift reduces to:
\[
2\on{dim} U^{-}-\on{dim} G+r+2(\on{dim}\cN-\on{dim}\bO)=\on{dim}\cN-\on{dim}\bO=2\on{dim}\cB_{-e}.
\]

\noindent Finally, $\cS_{-e}$ carries the Kac-Kazhdan $\bG_m$-action contracting it onto $-e$. This lifts to an action on $\widetilde{\cS}_{-e}$ contracting it onto $\cB_{-e}$. Thus, we may rewrite (\ref{eq:gettingcloser}) as
\[
H^*(\cB_{-e})[2\on{dim}\cB_{-e}]\simeq H^*(\cB_{e})[2\on{dim}\cB_{e}],
\]

\noindent as desired.

\end{proof}

\begin{rem}\label{r:Springeridentification}
The following technical remark will be useful in Section \ref{S:actiononspringerfibrelaks}.

For $x\in \cN$, let $i_x:\on{Spec}(k)\to \cN/G$ be the corresponding $k$-point of $\cN/G$. The nilpotent Slodowy slice $\cS_{-e}$ intersects the regular nilpotent locus (see \cite[§2]{ginzburg2009harish}), and so the map $f_{-e}$ is smooth of relative dimension $\on{dim}G-\on{dim}\bO$. As such, the right-hand side of (\ref{eq:nudetnok!}) coincides with:
\[
C_{\on{dR}}(\cS_{-e}, f_{-e}^{*,\on{dR}}\circ\on{FT}_{\fg}(\on{Spr}^{\fg}))[-\on{dim}\bO].
\]

\noindent Moreover, the sheaf $f_{-e}^{*,\on{dR}}\circ\on{FT}_{\fg}(\on{Spr}^{\fg})$ is $\bG_m$-equivariant for the Kac-Kazhdan action on $\cS_{-e}$. By the contraction principle, equation (\ref{eq:nudetnok!}) thus yields:
\begin{equation}\label{eq:odetillisa}
\on{coeff}_{\bO}(\on{Spr}^{\fg})\simeq i_{-e}^{*,\on{dR}}\circ \on{FT}_{\fg}(\on{Spr}^{\fg})[-\on{dim}\bO]\simeq i_{e}^{*,\on{dR}}\circ \on{FT}_{\fg}(\on{Spr}^{\fg})[-\on{dim}\bO].
\end{equation}

\end{rem}

\begin{proof}[Proof of Proposition \ref{p:keytoeverythingrs}]

Let $M\in \sW_{e,\lambda}\on{-mod}^{\heartsuit}$ be coherent. Let $M'\in \fg\on{-mod}_{\lambda}$ be the image of $M$ under Skryabin's equivalence. Then according to \cite{ginzburg2009harish}[Cor. 4.1.6], $M$ is finite-dimensional if and only if $M'$ satisfies that $\on{SS}(U(\fg)_{\lambda}/\on{Ann}_{U(\fg)_{\lambda}}(M'))\subset \overline{\bO}$. Thus, the proposition follows from Corollary \ref{c:boundonann}.

\end{proof}

\subsection{Springer action}\label{s:Spr}
In this subsection, we show that the realization of the ($\chi$-twisted) Grothendieck Springer sheaf $\on{Spr}_{\chi}$ as the character sheaf associated to $D(G/B,\chi)$ leads to the usual Springer action of $k[W_{[\lambda]}]$ on $\on{Spr}_{\chi}$. This in turn allows us to identify our categorical action (\ref{eq:actHH_Ow-mod}) of $k[W_{[\lambda]}]$ on the cohomology of the Springer fiber with the usual Springer action, see Corollary \ref{c:finalspringerboss} below.

\subsubsection{Categorical Springer action}
By §\ref{s:bimodulestructure} and Example \ref{e:Springer}, we have an action
\[
\on{Spr}_{\chi}\otimes \on{HH}_*(\sH_{G,\chi}^{\on{ren}})\to \on{Spr}_{\chi.}
\]

\noindent Since
\[
\on{HH}_*(\sH_{G,\chi}^{\on{ren}})\overset{\on{Lem.} \ref{l:HHisgroupalg}}{\simeq} k[W_{[\lambda]}]\otimes \on{Sym}(\ft^*[-1]\oplus \ft^*[-2]),
\]

\noindent we get an induced action
\begin{equation}\label{eq:springeract1}
\on{Spr}_{\chi}\curvearrowleft k[W_{[\lambda]}].
\end{equation}

\subsubsection{}\label{s:noncanequiv3} More generally, recall (cf. for example \cite{lusztig2020endoscopy}) that for any $w,y\in W$, we have an equivalence
\begin{equation}\label{eq:noncanequiv}
D(N\backslash NwB/B,y(\chi))\simeq \on{Vect}
\end{equation}

\noindent such that the image of $k$ is concentrated in perverse degree $-\ell(w)=-\on{dim}(NwB)/B$ as an object in $D(NwB/B,y(\chi))$.

We remind that $y(\chi)=y_{*,\on{dR}}(\chi)$ is the pushforward of $\chi$ along $y: T\to T.$ The equivalence (\ref{eq:noncanequiv}) is not canonical but depends on tensoring with a $!$-fibre of $y(\chi)$. Let $\phi_{wy,y}$ be the image of $k$ under the equivalence. It naturally descends to an object of $D(B,wy(\chi)\backslash NwB/B,y(\chi))^{\on{ren}}$, which we similarly denote by $\phi_{wy,y}$. Let
\begin{equation}\label{eq:jwyy}
j_{wy,y}\in D(B,wy(\chi)\backslash G/B,y(\chi))^{\on{ren}}
\end{equation}

\noindent be the $!$-pushforward of $\phi_{wy,y}$ under the locally closed embedding $NwB\into G$.

\subsubsection{} The $G$-linear functor
\[
D(G/B,\chi)\xrightarrow{-\star j_{1,w}} D(G/B,w(\chi))
\]

\noindent induces a map of character sheaves:
\begin{equation}\label{eq:springermap2}
\on{Spr}_{\chi}\to \on{Spr}_{w(\chi).}
\end{equation}

\noindent When $w\in W_{[\lambda]}$, we recover the action (\ref{eq:springeract1}).

\subsubsection{Usual Springer action}\label{s:USA} Recall that we also have a usual Springer action of $k[W_{[\lambda]}]$ on $\on{Spr}_{\chi}$ defined as follows. Let
\[
\pi: \widetilde{G}=G\overset{B}{\times}B\to G
\]

\noindent be the Grothendieck Springer resolution. Let $j:G_{\on{rs}}\into G$ be the open locus of regular semisimple elements. Recall that the map
\[
\pi_{\on{rs}}:\widetilde{G}_{\on{rs}}:=\pi^{-1}(G_{\on{rs}})\to G_{\on{rs}}
\]

\noindent is a regular $W$-covering. The map $\pi$ is $G$-equivariant, and so we get an induced map
\[
\pi_{\on{rs}}/G: \widetilde{G}_{\on{rs}}/G\to G\overset{\on{ad}}{/}G.
\]

\subsubsection{}\label{s:chirs} Write $\chi_{\on{rs}}$ for the $!$-pullback of $\chi$ under the composition
\[
\widetilde{G}_{\on{rs}}/G\into \widetilde{G}/G\simeq B\overset{\on{ad}}{/}B\to T\overset{\on{ad}}{/}T.
\]

\noindent Since $W_{[\lambda]}$ is the stabilizer of $\chi\in\ft^*/X^{\bullet}(T)$ in $W$, we get an action
\[
(\pi_{\on{rs}}/G)_{*,\on{dR}}(\chi_{\on{rs}})\curvearrowleft k[W_{[\lambda]}].
\]

\noindent In particular, we obtain an action
\begin{equation}\label{eq:springeract2}
\on{Spr}_{\chi}=j_{!*}\circ (\pi_{\on{rs}}/G)_{*,\on{dR}}(\chi_{\on{rs}})\curvearrowleft k[W_{[\lambda]}].
\end{equation}

\subsubsection{} More generally, for $w\in W$, we get a map
\begin{equation}\label{eq:springermap1}
w: \on{Spr}_{\chi}\to \on{Spr}_{w(\chi).}
\end{equation}

\noindent For $w\in W_{[\lambda]}$, this recovers the action (\ref{eq:springeract2}).

\subsubsection{} The goal of this subsection is to prove:
\begin{prop}\label{p:Springeract3}
The maps (\ref{eq:springermap2}) and (\ref{eq:springermap1}) coincide. In particular, the actions (\ref{eq:springeract1}) and (\ref{eq:springeract2}) coincide.
\end{prop}

\subsubsection{} 
Before the proof of Proposition \ref{p:Springeract3}, we record some basic assertions regarding Hochschild homology of categories of twisted D-modules.

Fix a regular semisimple element $t\in B$. Denote by $t_{*,\on{dR}}: D(G)\to D(G)$ pushforward along the left action map by $t$. For an element $x\in T$, we write $\ell_x(\chi)\in\on{Vect}$ for the $!$-fiber of $\chi$ at $x$. For $w\in W$, we write $w(-): T\to T$ for the action of $w$ on $T$.

\subsubsection{}\label{s:6210}
For $w,y\in W$, let $\chi_{w,y}$ be the generator of the category $D(wB/B,y(\chi))\subset D(NwB/B,y(\chi))$. With our conventions, this generator is $(wy)_{*,\on{dR}}(\chi^{\vee})$. Note that we have a canonical isomorphism
\[
t_{*,\on{dR}}(\chi_{w,y})\simeq \chi_{w,y}\otimes \ell_{(wy)^{-1}(t)}(\chi).
\]

\noindent We may depict the above isomorphism by the commutative diagram
\[\begin{tikzcd}
	{\mathrm{Vect}} && {D(NwB/B,y(\chi))} \\
	\\
	{\mathrm{Vect}} && {D(NwB/B,y(\chi)).}
	\arrow["{\chi_{w,y}}", from=1-1, to=1-3]
	\arrow["{-\otimes \ell_{(wy)^{-1}(t)}(\chi)}"', from=1-1, to=3-1]
	\arrow["{t_{*,\mathrm{dR}}}", from=1-3, to=3-3]
	\arrow["{\chi_{w,y}}"', from=3-1, to=3-3]
\end{tikzcd}\]

\noindent By functoriality of traces (see §\ref{S:5.1}), we get a map of vector spaces:
\begin{equation}\label{eq:haGOTHEEM}
\ell_{(wy)^{-1}(t)}(\chi)\simeq \on{HH}_*(\on{Vect}, -\otimes \ell_{(wy)^{-1}(t)}(\chi))\xrightarrow{[\chi_{w,y}]} \on{HH}_*(D(NwB/B,y(\chi)),t_{*,\on{dR}}).
\end{equation}

\noindent Similarly, by considering $\chi_{w,y}$ as an object in $D(G/B,y(\chi))$, we get a map:
\begin{equation}\label{eq:wirsindnachItaloBrothers}
\ell_{(wy)^{-1}(t)}(\chi)\xrightarrow{[\chi_{w,y}]}  \on{HH}_*(D(G/B,y(\chi)),t_{*,\on{dR}}).
\end{equation}

\subsubsection{}\label{s:6211'} The sheaf $j_{wy,y}$ (see §\ref{s:noncanequiv3}) also satisfies that
\[
t_{*,\on{dR}}(j_{wy,y})\simeq j_{wy,y}\otimes\ell_{(wy)^{-1}(t)}(\chi),
\]

\noindent and so by the same construction, we get maps:
\begin{equation}\label{eq:jwyyclass}
\ell_{(wy)^{-1}(t)}(\chi)\xrightarrow{[j_{wy,y}]} \on{HH}_*(D(NwB/B,y(\chi)),t_{*,\on{dR}}),
\end{equation}
\begin{equation}\label{eq:987987987123123123}
\ell_{(wy)^{-1}(t)}(\chi)\xrightarrow{[j_{wy,y}]} \on{HH}_*(D(G/B,y(\chi)),t_{*,\on{dR}}).
\end{equation}

\noindent We want to show (see Lemma \ref{l:mainlemmactually} below) that the two maps (\ref{eq:haGOTHEEM}) and (\ref{eq:jwyyclass}) coincide. Before doing so, we will need to consider the duality datum associated to the category of twisted D-modules on a variety a bit closer.

\subsubsection{}\label{s:6212} Let $X$ be a variety equipped with a right action of an algebraic group $H$. Suppose $\chi$ is a character sheaf on $H$ and $\phi: X\rightarrow X$ an $H$-equivariant endomorphism.
Let us write $\on{HH}_*(D(X/H,\chi),\phi_{*,\on{dR}})$ in explicit terms. By definition, tensoring with this vector space recovers the composition
\[\begin{tikzcd}
	&&&& {D(X/H)} & {\mathrm{Vect.}} \\
	\\
	{\mathrm{Vect}} && {D(X/H)} && {D(X/H,\chi)\otimes D(X/H,\chi^{\vee})}
	\arrow["{C_{\mathrm{dR}}(X/H,-)}", from=1-5, to=1-6]
	\arrow["{\omega_{X/H}}", from=3-1, to=3-3]
	\arrow["{\Gamma_{\phi,*,\mathrm{dR}}}", from=3-3, to=3-5]
	\arrow["{\Delta^!}"', from=3-5, to=1-5]
\end{tikzcd}\]

\noindent Here, $\Gamma_{\phi,*,\on{dR}}$ denotes the functor given by the composition of:
\begin{itemize}
    \item $*$-pushforward along the map $X/H\to X\times X/H$ given by the graph of $\phi$, where the action of $H$ on $X\times X$ is the diagonal action.

    \item Relative averaging 
    \[
    \on{Av}_*^{\Delta H\to (H\times H,\chi\boxtimes \chi^{\vee})}: D(X\times X/H)\to D(X/H,\chi)\otimes D(X/H,\chi^{\vee}),
    \]

    \noindent where we consider $H$ as a subgroup of $H\times H$ via the diagonal embedding.
\end{itemize}

\noindent When $\phi=\on{id}$, we write $\Delta_{*,\on{dR}}:=\Gamma_{\on{id},*,\on{dR}}$.

The functor $\Delta^!: D(X/H,\chi)\otimes D(X/H,\chi^{\vee})\to D(X/H)$ is defined as the composition of the forgetful functor $\on{oblv}^{(H\times H,\chi\boxtimes \chi^{\vee})\to \Delta H}: D(X/H,\chi)\otimes D(X/H,\chi^{\vee})\to D(X\times X/H)$, followed by $!$-pullback along the diagonal map $X/H\to X\times X/H$.

\subsubsection{}\label{eq:isoesction}
Let
\[
Z=\lbrace (x,h)\in X\times H\;\vert\; \phi(x)=xh\rbrace\subset X\times H.
\]

\noindent This scheme carries a natural right $H$-action given by $(x,h_1)\cdot h_2=(xh_2,h_2^{-1}h_1h_2)$. Let
\[
\on{Iso}^{\phi}:=Z/H
\]

\noindent be the corresponding (twisted) isotropy stack. We have natural projection maps
\[
p: \on{Iso}^{\phi}\to H\overset{\on{ad}}{/}H, \;\; q: \on{Iso}^{\phi}\to X/H.
\]
\begin{lem}\label{l:isoinput}
We have a canonical isomorphism
\[
\on{HH}_*(D(X/H,\chi),\phi_{*,\on{dR}})\simeq C_{\on{dR}}(\on{Iso}^{\phi},p^!(\chi)).
\]
\end{lem}

\begin{proof}
We know that
\[
\on{HH}_*(D(X/H,\chi),\phi_{*,\on{dR}})\simeq C_{\on{dR}}(X/H, \Delta^!\circ\Gamma_{\phi,*,\on{dR}}(\omega_{X/H})).
\]

\noindent Thus, it suffices to provide a canonical isomorphism
\begin{equation}\label{eq:turndownforwhaaat}
\Delta^!\circ\Gamma_{\phi,*,\on{dR}}(\omega_{X/H})\simeq q_{*,\on{dR}}\circ p^!(\chi).
\end{equation}

\noindent Let
\[
\widetilde{Z}=\lbrace (x,h,k)\in X\times H\times H\;\vert\; xh=\phi(x)k\rbrace\subset X\times H\times H.
\]

\noindent We have a Cartesian diagram
\[\begin{tikzcd}
	{\widetilde{Z}} && {X\times H} \\
	\\
	{X\times H} && {X\times X.}
	\arrow["{(x,h,k)\mapsto (x,h)}", from=1-1, to=1-3]
	\arrow["{(x,h,k)\mapsto (\phi(x),k)}"', from=1-1, to=3-1]
	\arrow["{(x,h)\mapsto (xh,\phi(x))}", from=1-3, to=3-3]
	\arrow["{(x,h)\mapsto (xh,x)}"', from=3-1, to=3-3]
\end{tikzcd}\]

\noindent This diagram is $H\times H$-equivariant, where:
\begin{itemize}
    \item $H\times H$ acts on $X\times X$ by $(x,y)\cdot (h,k)=(xh,yk)$;

    \item $H\times H$ acts on $X\times H$ by $(x,h)\cdot (k,l)=(xl,l^{-1}hk)$;

    \item $H\times H$ acts on $\widetilde{Z}$ by $(x,h,k)\cdot (l,m)=(xm,m^{-1}hl,m^{-1}kl)$.
\end{itemize}

\noindent Passing to $(B\times B,\chi\boxtimes \chi^{\vee})$-invariants, we identify:
\[
D(X/H)\simeq D(X\times H)^{(B\times B,\chi\boxtimes \chi^{\vee})},\;\; \omega_{X/H}\mapsto \omega_X\boxtimes \chi,
\]
\[
D(\on{Iso}^{\phi})\simeq D(\widetilde{Z})^{(B\times B,\chi\boxtimes \chi^{\vee})}.
\]

\noindent Moreover, under these identifications:
\begin{itemize}

\item The leftmost vertical map in the above diagram goes to $q$;

\item Pushforward (resp. pullback) along the rightmost vertical map (resp. bottom horizontal map) goes to $\Gamma_{\phi,*,\on{dR}}$ (resp. $\Delta^!$).

\end{itemize}

\noindent The lemma now follows by base change.

\end{proof}

\subsubsection{}\label{ss:628butformodn} Next, we examine the maps (\ref{eq:haGOTHEEM}) and (\ref{eq:jwyyclass}).

Fix $w,y\in W, t\in T^{\on{rs}}$. Write $\ell:=\ell_{(wy)^{-1}(t)}(\chi)$ and $\underline{\ell}:=\underline{k}_{NwB/B}\otimes \ell$, where $\underline{k}_{NwB/B}$ is the constant sheaf on $NwB/B$. For any compact $\cF\in D(NwB/B,y(\chi))$, we have maps
\begin{equation}\label{eq:havemaps1}
\underline{k}_{NwB/B}\xrightarrow{u_{\cF}} \cF\overset{!}{\otimes}\bD(\cF)\xrightarrow{\epsilon_{\cF}} \omega_{NwB/B}.
\end{equation}

\noindent Here, $u_{\cF}$ is given by the image of $\on{id}_{\cF}$ under the identification
\[
\on{Hom}_{D(NwB/B,y(\chi))}(\cF, \cF)\simeq \on{Hom}_{D(NwB/B)}(\underline{k}_{NwB/B}, \cF\overset{!}{\otimes}\bD(\cF)).
\]

\noindent The map $\epsilon_{\cF}$ is given by applying $\Delta^!$ to the canonical map
\begin{equation}\label{eq:maptorulethemall}
\cF\boxtimes \bD(\cF)\to \Delta_{*,\on{dR}}(\omega_{NwB/B})
\end{equation}

\noindent given by the image of $\on{id}_{\cF}$ under the identification
\[
\on{Hom}_{D(NwB/B,y(\chi))}(\cF, \cF)\simeq \on{Hom}_{D(NwB/B)}(\cF\overset{*}{\otimes}\bD(\cF),\omega_{NwB/B})
\]
\[
\simeq \on{Hom}_{D(G/B,\chi)\otimes D(G/B,\chi^{\vee})}(\cF\boxtimes\bD(\cF), \Delta_{*,\on{dR}}(\omega_{G/B})).
\]

\noindent Tensoring (\ref{eq:havemaps1}) with the line $\ell$, we obtain:
\begin{equation}\label{eq:havemaps2}
\underline{\ell}\to (\cF\otimes \ell)\overset{!}{\otimes} \bD(\cF)\to \omega_{NwB/B}\otimes \ell.
\end{equation}

\noindent If we further have the datum of an isomorphism
\[
t_{*,\on{dR}}(\cF)\simeq \cF\otimes \ell,
\]

\noindent we may consider the composition:
\[
\underline{\ell}\to (\cF\otimes \ell)\overset{!}{\otimes} \bD(\cF)\simeq t_{*,\on{dR}}(\cF)\overset{!}{\otimes} \bD(\cF)\simeq \Delta^!\circ (t_{*,\on{dR}}\otimes\on{id})(\cF\boxtimes\bD(\cF))\to 
\]
\begin{equation}\label{eq:havemaps3}
\xrightarrow{\Delta^!\circ (t_{*,\on{dR}}\otimes\on{id})(\epsilon_{\cF})}\Delta^!\circ (t_{*,\on{dR}}\otimes \on{id})\circ\Delta_{*,\on{dR}}(\omega_{NwB/B})\simeq \Delta^!\circ\Gamma_{t,*,\on{dR}}(\omega_{NwB/B}).
\end{equation}

\noindent By adjunction, we get a map
\begin{equation}\label{eq:havemaps4}
\ell\to C_{\on{dR}}(NwB/B,\Delta^!\circ \Gamma_{t,*,\on{dR}}(\omega_{NwB/B}))=\on{HH}_*(D(NwB/B,y(\chi)),t_{*,\on{dR}}).
\end{equation}

\noindent Taking $\cF = \chi_{w,y}, j_{wy,y}$, we exactly compute the maps (\ref{eq:haGOTHEEM}), (\ref{eq:jwyyclass}).\footnote{We remark that the above computation holds on very general grounds: we did not use the geometry of $NwB$.}

\begin{lem}\label{l:mainlemmactually}
The maps (\ref{eq:haGOTHEEM}) and (\ref{eq:jwyyclass}) coincide.
\end{lem}

\begin{proof}
For ease of notation, assume that $y=1$. We continue to write $\ell=\ell_{w^{-1}(t)}(\chi)$ for our fixed $w\in W.$

With the notation of §\ref{eq:isoesction}, note that the map 
\[
q: \on{Iso}^t\to NwB/B
\]

\noindent induces an isomorphism
\[
\on{Iso}^t\simeq wB/B.
\]

\noindent Indeed, this follows from the fact that $\pi_{\on{rs}}:\widetilde{G}_{\on{rs}}\to G_{\on{rs}}$ is a regular $W$-covering, cf. §\ref{s:USA}. With this identification, the map
\[
\on{pt}=wB/B\into \on{Iso}^t\xrightarrow{p} B\overset{\on{ad}}{/}B
\]

\noindent corresponds to the point $w^{-1}(t)$. Let $i_{w}: wB/B\into NwB/B$ be the inclusion, and let $\delta_{w}:=i_{w,*,\on{dR}}(\ell)$. 
By (\ref{eq:turndownforwhaaat}), we have a canonical isomorphism
\[
\Delta^!\circ \Gamma_{t,*,\on{dR}}(\omega_{NwB/B})\simeq q_{*,\on{dR}}\circ p^!(\chi)\simeq \delta_{w}\in D(NwB/B).
\]

\noindent By adjunction, the map (\ref{eq:havemaps3}) factors as:
\[
\underline{\ell}\to i_{w,*,\on{dR}}\circ i_{w}^{*,\on{dR}}(\underline{\ell})\to i_{w,*,\on{dR}}\circ i_{w}^{*,\on{dR}}((j_{w,1}\otimes \ell)\overset{!}{\otimes}\bD(j_{w,1}))\to \delta_{w}.
\]

\noindent Note that
\[
i_w^{*,\on{dR}}((j_{w,1}\otimes \ell)\overset{!}{\otimes}\bD(j_{w,1}))\simeq (\chi_{w,1}\otimes\ell)\overset{!}{\otimes} \bD(\chi_{w,1}).
\]

\noindent This is immediate from the fact that $j_{w,1}$ is lisse on $NwB$ and that $i_w^{*,\on{dR}}(j_{w,1})\simeq \chi_{w,1}[2\ell(w)]$, where $\ell(w)$ denotes the length of $w$. Thus, we see that (\ref{eq:havemaps3}) canonically identifies with the map (\ref{eq:haGOTHEEM}), as desired.

\end{proof}

\subsubsection{} We are now in a position to prove Proposition \ref{p:Springeract3}.

\begin{proof}[Proof of Proposition \ref{p:Springeract3}] We will use the following notation. Given a map of algebraic groups $J \rightarrow K$, we write the associated induction and restriction of categorical representations by
\[\begin{tikzcd}
	{D(J)\textrm{\textbf{-mod}}} && {D(K)\textrm{\textbf{-mod}}.}
	\arrow["{\mathrm{ind}_J^K}", shift left, from=1-1, to=1-3]
	\arrow["{\mathrm{res}_J^K}", shift left, from=1-3, to=1-1]
\end{tikzcd}\]
That is, lower case first letters will denote these functors, whereas upper case first letters will, as before, denote parabolic induction and restriction.

In addition, since we will consider both character sheaves of general categorical representations $\chi_{\sC}$ and our fixed character sheaf on $B$, we denote the latter by $\theta$.

\step We first claim we may assume that our Weyl group element is a simple reflection $s$. Indeed, the case of general $w$ then follows from a straightforward induction on the length of $w$.

Let us next reduce to the case of $G$ being of semisimple rank one. To do so, write $B \hookrightarrow P \hookrightarrow G$ for the standard minimal parabolic of $G$ corresponding to $s$. In particular, we may consider its Levi quotient $L = P/U_P$, which is of semisimple rank one and comes equipped with its standard Borel $B_L = B/U_P$.

Note that $\theta$ descends uniquely to a character sheaf on $B_L$, which we again denote by $\theta$. Moreover, we have a tautological isomorphism 
$$\on{Ind}_L^G D(L/B_L, \theta) \simeq D(G/B, \theta).$$
As the constructions of the categorical and usual Springer action of $s$ are straightforwardly seen to commute with parabolic induction of categorical representations and character sheaves (see Lemma \ref{l:compwithpind}), we may reduce to the case of $G = L$. 

\iffalse 
\step We next further reduce to the case when the the derived subgroup of $G$ is of adjoint type, i.e., $[G,G] \simeq PGL_2$. Indeed, more generally, given a finite (necessarily central) isogeny $q: G \rightarrow G'$, write $q: B \rightarrow B'$ for the corresponding Borel subgroup of $G'$. 

It is standard that we may choose a (non-unique) descent of $\theta$, i.e., a pair $(\theta', \alpha)$, where $\theta'$ is a character sheaf on $B'$ and $\alpha$ is an isomorphism of character sheaves $\alpha: q^!(\theta') \simeq \theta$.

With this, in evident notation, we have a canonical isomorphism of $D(G)$-modules
$$D(G/B, \theta) \simeq \mathrm{res}_{G}^{G'} D(G'/B', \theta').$$
As the constructions of the categorical and usual Springer action of $s$ are straightforwardly seen to commute with pullback along finite isogenies, we may reduce to the case of $G = G'$. As we can take $G'$ explicitly to be the quotient of $G$ by the center of its derived subgroup, this completes the reduction step. 
\fi

\step Choose a splitting $T \hookrightarrow B$, write $T_{rs} \hookrightarrow T$ for the regular semisimple locus, and consider the induced map 
$$i: T_{rs}\overset{\on{ad}}{/}T \rightarrow G\overset{\on{ad}}{/}G.$$
It suffices to see that the categorical and usual Springer actions agree after pullback along $i^!$, and indeed even upon further $!$-restricting to a closed point $t: \Spec k \rightarrow T_{rs}/T$.

Denote the normalizer  of $T$ in $G$  by $T^+ := N_G(T)$, and note that we have a canonical natural transformation 
$$\eta: \on{res}_{T}^{T^+} \circ \on{ind}_{T}^{T^+}(-)  \rightarrow \on{res}_{T}^{G} \circ \on{Ind}_{T}^G(-),$$
corresponding to the map $\iota: T^+ \rightarrow G/N$, and the induced pushforward $\iota_{*, \dR}: D(T^+) \rightarrow D(G/N)$.  As $\iota$ is a closed embedding, this is right adjointable, whence for any dualizable $D(T)$-module $\sC$ we obtain a canonical map on character sheaves:
$$\on{HH}_*(\eta): \chi_{\on{res}_{T}^{T^+} \circ \on{ind}_{T}^{T^+}(\sC)} \rightarrow \chi_{\on{res}_{T}^G \circ \on{Ind}_{T}^G(\sC)}.$$
We note that for $\sC \simeq D(T/T, \theta)$, upon restriction to the open subset $T^{rs}/T$, the map $\on{HH}_*(\eta)$ is an isomorphism. 

\step To proceed, let us now interpret the usual Springer action in terms of categorical induction, albeit to $T^+$ rather than $G$. Namely, if we consider the conjugation action $s: D(T) \simeq D(T)$, in particular we can pushforward categorical representations $\sC$ along it; denote this by $\sC \mapsto s_* \sC := \on{res}_{T}^{T} \sC$.  

Moreover, if we consider the induced representations 
$$\on{ind}_{T}^{T^+}(\sC) \quad \text{and} \quad \on{ind}_{T}^{T^+}(s_* \sC),$$
for a choice of coset representative $\dot{s}$ in $T^+$, we have an associated isomorphism 
\begin{equation}\label{eq:nokia}
\dot{s}: \on{ind}_{T}^{T^+}(\sC) \simeq \on{ind}_{T}^{T^+}(s_* \sC)
\end{equation}

\noindent of $T^+$-categories. Namely, it is induced by adjunction from the map $D(T) \rightarrow D(T^+)$ given by left convolution with the delta D-module $\delta_{\dot{s}}$. 

For $\sC = D(T/T, \theta)$, note that $s_*\sC\simeq D(T/T,s(\theta))$. In this case, the isomorphism (\ref{eq:nokia}) is independent of the choice of $\dot{s}$ up to tensoring with a $!$-fiber $\ell$ of $\theta$. In particular, the induced map of character sheaves 
\begin{equation} \label{dots}\on{HH}_*(\dot{s}): \chi_{ \on{ind}_T^{T^+}(\sC)} \rightarrow \chi_{\on{ind}_{T}^{T^+}(s_*\sC)}\end{equation}
is canonically independent of the choice of lift $\dot{s}$. 

Recall that induction and restriction of dualizable categorical representations yields $*$-pushforward and $!$-pullback of character sheaves, respectively. If we consider the adjoint quotient $T^+/T^+$, note that its components are given by preimages of conjugacy classes along the canonical projection $$\pi: T^+\overset{\on{ad}}{/}T^+ \rightarrow W\overset{\on{ad}}{/}W.$$As the tautological map $T\overset{\on{ad}}{/}T \rightarrow T^+\overset{\on{ad}}{/}T^+$ factors through the neutral component $(T^+\overset{\on{ad}}{/}T^+)^\circ \simeq T\overset{\on{ad}}{/}T^+$, the isomorphism \eqref{dots} is between two sheaves supported on $T\overset{\on{ad}}{/}T^+$. Recalling that the tautological map $T^{rs} \overset{\on{ad}}{/} T^+ \rightarrow G^{rs}\overset{\on{ad}}{/}G$ is an isomorphism, it follows from the definitions that $\on{HH}_*({\dot{s}})$ agrees with the usual Springer action. 

\step We continue to take $\sC := D(T/T, \theta)$ so that $\on{Ind}_T^G(\sC)\simeq D(G/B,\theta)$ and $\on{Ind}_T^G(s_*\sC)\simeq D(G/B,s(\theta))$. Consider the diagram:
\[
\begin{tikzcd}\label{d:diaagraaamz}
	\ind_T^{T^+}(\sC)  & \on{Ind}_T^G(\sC) \\ \ind_T^{T^+}(s_*\sC)  & \on{Ind}_T^G(s_*\sC).
	\arrow["\dot{s}",  from=1-1, to=2-1]
	\arrow["-\star j_{1,s}",  from=1-2, to=2-2]
     \arrow["\eta",  from=1-1, to=1-2 ]
      \arrow["\eta",  from=2-1, to=2-2 ]
\end{tikzcd}
\]

\noindent While this diagram of categories is lax commutative, it is not strictly commutative. Nonetheless, we must show that, for $t \in T_{rs}$, the induced diagram of vector spaces obtained after applying $\on{HH}(-,t_{*,\on{dR}})$ does commute.

Using that $\on{HH}_*(\eta)$ is an isomorphism over $T^{\on{rs}}/T$, we need to show that the two maps
\begin{equation}\label{d:twomapsoneaction}
\begin{tikzcd}
	\on{HH}_*(D(G/B,\theta),t_{*,\on{dR}}) & \on{HH}_*(D(G/B,s(\theta)),t_{*,\on{dR}})
	\arrow[shift left, from=1-1, to=1-2]
	\arrow[shift right, from=1-1, to=1-2]
\end{tikzcd}
\end{equation}

\noindent induced by the diagram coincide. Recall that $\on{HH}_*(D(G/B,\theta),t_{*,\on{dR}})$ identifies with the $!$-fiber of $\chi_{D(G/B,\theta)}\simeq \on{Spr}_{\theta}$ at $t$. Similarly, $\on{HH}_*(D(G/B,s(\theta)),t_{*,\on{dR}})$ identifies with the $!$-fiber of $\chi_{D(G/B,s(\theta))}\simeq \on{Spr}_{s(\theta)}$ at $t$. In particular, both identify with $\ell_t\oplus \ell_{s(t)}$. Concretely, the embeddings
\[
\ell_t\to \on{HH}_*(D(G/B,\theta),t_{*,\on{dR}}),\;\; \ell_t \to \on{HH}_*(D(G/B,s(\theta)),t_{*,\on{dR}})
\]

\noindent are induced by the maps $[\theta_{1,1}], [\theta_{s,s}]$, respectively, cf. (\ref{eq:wirsindnachItaloBrothers}). Similarly, the embeddings
\[
\ell_{s(t)}\to \on{HH}_*(D(G/B,\theta),t_{*,\on{dR}}),\;\; \ell_{s(t)} \to \on{HH}_*(D(G/B,s(\theta)),t_{*,\on{dR}})
\]

\noindent are induced by the maps $[\theta_{s,1}], [\theta_{1,s}]$. By construction, the compositions
\[
\ell_t\xrightarrow{[\theta_{1,1}]} \on{HH}_*(D(G/B,\theta),t_{*,\on{dR}})\xrightarrow{\on{HH}_*(\dot{s})} \on{HH}_*(D(G/B,s(\theta)),t_{*,\on{dR}}),
\]
\[
\ell_{s(t)}\xrightarrow{[\theta_{s,1}]} \on{HH}_*(D(G/B,\theta),t_{*,\on{dR}})\xrightarrow{\on{HH}_*(\dot{s})} \on{HH}_*(D(G/B,s(\theta)),t_{*,\on{dR}})
\]

\noindent coincide with $[\theta_{s,s}]$ and $[\theta_{1,s}]$, respectively. Thus, we must show that the same is true when replacing $\on{HH}_*(\dot{s})$ with $\on{HH}_*(-\star j_{1,s})$.

\step 

The two maps (\ref{d:twomapsoneaction}) are both induced by taking the fiber at $t$ of isomorphisms 
\[\begin{tikzcd}\label{eq:twomapsoneaction}
	\on{Spr}_{\theta} & \on{Spr}_{s(\theta).}
	\arrow[shift left, from=1-1, to=1-2]
	\arrow[shift right, from=1-1, to=1-2]
\end{tikzcd}\]

\noindent Note that $\on{Spr}_{\theta}$ is either irreducible or the direct sum of two non-isomorphic simple summands. In particular, on each simple summand, the two maps differ by a constant. As such, the composition 
\[
\ell_t\xrightarrow{[\theta_{1,1}]} \on{HH}_*(D(G/B,\theta),t_{*,\on{dR}})\xrightarrow{\on{HH}_*(-\star j_{1,s})} \on{HH}_*(D(G/B,s(\theta)),t_{*,\on{dR}})
\]

\noindent differs from $[\theta_{s,s}]$ by a constant $c\in k$. On the other hand, the composition is $[j_{s,s}]$, cf. (\ref{eq:987987987123123123}). To see that $c=1$, it suffices to check that $[\theta_{s,s}]$ and $[j_{s,s}]$ agree after composing with the map
\[
\on{HH}_*(D(G/B,\theta),t_{*,\on{dR}})\to \on{HH}_*(D(NsB/B,s(\theta)),t_{*,\on{dR}})
\]

\noindent induced by applying $\on{HH}_*(-)$ to pullback along $NsB\to G$. However, this is the content of Lemma \ref{l:mainlemmactually}.

It remains to check that the composition
\[
\ell_{s(t)}\xrightarrow{[\theta_{s,1}]} \on{HH}_*(D(G/B,\theta),t_{*,\on{dR}})\xrightarrow{\on{HH}_*(-\star j_{1,s})} \on{HH}_*(D(G/B,s(\theta)),t_{*,\on{dR}})
\]

\noindent coincides with $[\theta_{1,s}]$. We have (see \cite[Lemma 3.5]{lusztig2020endoscopy}):\footnote{To apply \emph{loc.cit}, we are using that in the Grothendieck group of $\sH_{G,\theta}^{\on{ren}}$, $j_{s,1}$ coincides with its analogue where we $*$-pushforward $\phi_{s,1}$, cf. §\ref{s:noncanequiv3}.}
\[
\on{HH}_*((-\star j_{1,s})\circ (-\star j_{s,1}))=\on{HH}_*(j_{1,1})=\on{id.}
\]

\noindent As such, it suffices to show that the composition
\[
\ell_{s(t)}\xrightarrow{[\theta_{1,s}]} \on{HH}_*(D(G/B,s(\theta)),t_{*,\on{dR}})\xrightarrow{\on{HH}_*(-\star j_{s,1})} \on{HH}_*(D(G/B,\theta),t_{*,\on{dR}})
\]

\noindent coincides with $[\theta_{s,1}]$. However, this follows from the same reasoning as above, noting that our argument was symmetric in $\theta$ and $s(\theta)$.

\end{proof}

\subsection{Action on cohomology of Springer fiber}\label{S:actiononspringerfibrelaks}

\subsubsection{} In this subsection, we identify the action (\ref{eq:10/09:1716}) of $k[W_{[\lambda]}]$ on $H^*(\cB_e)$ with the usual Springer action.\footnote{There are two standard Springer actions $H^*(\cB_e)\curvearrowleft k[W_{[\lambda]}]$, and they differ by tensoring with the sign representation. The version we consider here is the one normalized so that the trivial representation corresponds to $e=0$.}

Recall that the usual Springer action of $k[W_{[\lambda]}]$ on $H^*(\cB_e)$ is obtained as follows (see e.g. \cite{achar2014weyl}). The algebra $k[W_{[\lambda]}]$ acts on $\on{Spr}^{\fg}$ by the same construction as in §\ref{s:USA}, replacing $G$ by $\fg$. By functoriality, we obtain an action on the vector space:
\[
i_e^{*,\on{dR}}\circ \on{FT}_{\fg}(\on{Spr}^{\fg})\curvearrowleft W_{[\lambda]}.
\]

\noindent Here, $i_{e}: \on{Spec}(k)\to \cN/G$ denotes the $k$-point corresponding to $e$. The sheaf $\on{FT}_{\fg}(\on{Spr}^{\fg})$ is supported on $\cN/G$ and identifies with the $*$-pullback of $\on{Spr}^{\fg}$ along $\cN/G\to \fg/G$, up to a shift.\footnote{We remind that this is spelled out in the proof of Lemma \ref{l:HHiscoho}.} By base-change, the vector space $i_e^{*,\on{dR}}\circ \on{FT}_{\fg}(\on{Spr}^{\fg})$ identifies with $H^*(\cB_e)$ (up to a shift), thus providing the Springer representation:
\begin{equation}\label{eq:finalspringeractionvect}
H^*(\cB_e)\curvearrowleft W_{[\lambda]}.
\end{equation}

\subsubsection{} We want to compare the above action to the categorical action (\ref{eq:actHH_Ow-mod}). We remind that the latter was given by applying $\on{HH}_*(-)$ to the action
\[
\sW_{e,\lambda}\on{-mod}\overset{(\ref{eq:Skryabin+BB})}{\simeq} D(U^{-},\psi_e\backslash G/B,\chi)\curvearrowleft \sH_{G,\chi}^{\on{ren}},
\]

\noindent and then using the identification
\[
H^*(\cB_e)[2\on{dim}\cB_e]\overset{\on{Lem.} \ref{l:HHiscoho}}{\simeq}\on{HH}_*(\sW_{e,\lambda}\on{-mod})\curvearrowleft \on{HH}_*(\sH_{G,\chi}^{\on{ren}}).
\]

\noindent Since $k[W_{[\lambda]}]$ injects into $\on{HH}_*(\sH_{G,\chi}^{\on{ren}})$, cf. Lemma \ref{l:HHisgroupalg}, we get an action
\begin{equation}\label{eq:finalspringeryaaz}
H^*(\cB_e)\curvearrowleft W_{[\lambda]}.
\end{equation}
\begin{cor}\label{c:finalspringerboss}
The actions (\ref{eq:finalspringeractionvect}) and (\ref{eq:finalspringeryaaz}) coincide.
\end{cor}

\begin{proof}
\step

By Lemma \ref{l:HHiscoeff}, we may rewrite the action (\ref{eq:finalspringeryaaz}) as follows. We have the categorical Springer action:
\[
\on{Spr}_{\chi}\overset{(\ref{eq:springeract1})}{\curvearrowleft} k[W_{[\lambda]}].
\]

\noindent Applying the functor $\on{coeff}_{\bO}^G$, we obtain an action
\begin{equation}\label{eq:lastoneploxx}
H^*(\cB_e)[2\on{dim}\cB_e]\simeq\overset{\on{Lem.}\ref{l:HHiscoeff}}{\simeq}\on{coeff}_{\bO}^G(\on{Spr}_{\chi})\curvearrowleft k[W_{[\lambda]}],
\end{equation}

\noindent which (up to a shift) coincides with the action (\ref{eq:finalspringeryaaz}). By Proposition \ref{p:Springeract3}, the action (\ref{eq:lastoneploxx}) coincides with applying $\on{coeff}_{\bO}^G$ to the usual Springer action:
\[
\on{Spr}_{\chi}\overset{(\ref{eq:springeract2})}{\curvearrowleft} k[W_{[\lambda]}].
\]

\step Let $i_{\cN}: \cN/G\into \fg/G$ be the inclusion. We denote by $'i_{\cN}$ the inclusion $\cN/G\into G\overset{\on{ad}}{/}G$ under the identification of the nilpotent locus of $\fg$ with the unipotent locus of $G$. The above gives an action
\begin{equation}\label{eq:iseminem'snewalbumanygood?}
'i_{\cN}^!(\on{Spr}_{\chi})\overset{'i_{\cN}^!(\ref{eq:springeract2})}{\curvearrowleft} k[W_{[\lambda]}].
\end{equation}

\noindent On the other hand, applying $i_{\cN}^!$ to
\[
\on{Spr}^{\fg}\curvearrowleft k[W_{[\lambda]}],
\]

\noindent we get an action:
\begin{equation}\label{eq:deltaairlinesdoesnthaveenoughgoodmusic}
i_{\cN}^!(\on{Spr}^{\fg})\curvearrowleft k[W_{[\lambda]}].
\end{equation}

\noindent By construction, under the canonical identification
\[
'i_{\cN}^!(\on{Spr}_{\chi})\simeq i_{\cN}^!(\on{Spr}^{\fg}),
\]

\noindent the two actions (\ref{eq:iseminem'snewalbumanygood?}) and (\ref{eq:deltaairlinesdoesnthaveenoughgoodmusic}) coincide.\footnote{If $k=\bC$, this follows from Riemann-Hilbert as in §\ref{ss:exp}: locally around $1\in G$, $\mathbf{exp}: \fg\to G$ is a biholomorphism and on this neighbourhood, pullback takes $\on{Spr}_{\chi}$ to $\on{Spr}^{\fg}$. In general, we can reduce to $k=\bC$ by Lefschetz principle.}

\step 

Combining the above two steps, we get that the action (\ref{eq:finalspringeryaaz}) is obtained by applying $\on{coeff}_{\bO}^{\fg}$ to the Springer action
\[
\on{Spr}^{\fg}\curvearrowleft k[W_{[\lambda]}].
\]

\noindent The proposition now follows from Remark \ref{r:Springeridentification}.

\end{proof}

\subsubsection{} We end this subsection by sketching the proof of the following strengthening of Lemma \ref{l:HHisgroupalg}:
\begin{lem}\label{l:strongversion}
We have a canonical isomorphism
\[
\on{HH}_*(\sH_{G,\chi}^{\on{ren}})=W_{[\chi]}\ltimes \on{Sym}(\ft^*[-1]\oplus \ft^*[-2])
\]

\noindent of dg-algebras, where $W_{[\chi]}$ acts on $\on{Sym}(\ft^*[-1]\oplus \ft^*[-2])$ via the diagonal action of $W_{[\chi]}$ on $\ft^*[-1]\oplus \ft^*[-2]$.
\end{lem}

\begin{proof}[Proof sketch]
\step 
From Lemma \ref{l:HHisgroupalg}, we know that
\[
\on{HH}_*(\sH_{G,\chi}^{\on{ren}})=W_{[\chi]}\ltimes \on{Sym}(\ft^*[-1]\oplus \ft^*[-2])
\]

\noindent for some action
\begin{equation}\label{eq:1109:1123}
W_{[\lambda]}\curvearrowright \on{Sym}(\ft^*[-1]\oplus \ft^*[-2]).
\end{equation}

\noindent Concretely, the action is given as follows: for $w\in W$, consider the endofunctor:
\[
D(B,\chi\backslash B/B,\chi)^{\on{ren}}\xrightarrow{j_{w,1}\star - } D(B,w(\chi)\backslash G/B,\chi)^{\on{ren}}\xrightarrow{-\star j_{1,w}} D(B,w(\chi)\backslash B/B,w(\chi))^{\on{ren}}\simeq D(B,\chi\backslash B/B,\chi)^{\on{ren}}.
\]

\noindent Here, the last equivalence comes from the canonical identifications $D(B,\chi\backslash B/B,\chi)^{\on{ren}}\simeq C^*(\bB T)\on{-mod}\simeq D(B,w(\chi)\backslash B/B,w(\chi))^{\on{ren}}$. Passing to Hochschild homology, we obtain the endomorphism
\begin{equation}\label{eq:1109:1747}
w: \on{Sym}(\ft^*[-1]\oplus \ft^*[-2])\to \on{Sym}(\ft^*[-1]\oplus \ft^*[-2]).
\end{equation}

\noindent This provides an action of $W$ on $\on{Sym}(\ft^*[-1]\oplus \ft^*[-2])$, and restricting this action to $W_{[\lambda]}$ recovers (\ref{eq:1109:1123}). Note that the multiplication map
\[
\on{Sym}(\ft^*[-1]\oplus \ft^*[-2])\otimes \on{Sym}(\ft^*[-1]\oplus \ft^*[-2])\to \on{Sym}(\ft^*[-1]\oplus \ft^*[-2])
\]

\noindent is $k[W]$-linear for the diagonal action (\ref{eq:1109:1747}) on the left-hand side. Thus, it suffices to show that for every simple reflection $s\in W$, the corresponding endomorphism (\ref{eq:1109:1747}) acts on $\ft^*[-1]\oplus \ft^*[-2]$ in the usual way.

\step 

Let $P_s$ be the parabolic associated to $s$ and $G_s$ its Levi quotient. By applying the above construction, replacing $G$ by $G_s$, we similarly get an action
\[
s: \on{Sym}(\ft^*[-1]\oplus \ft^*[-2])\to \on{Sym}(\ft^*[-1]\oplus \ft^*[-2]),
\]

\noindent where we consider $\ft$ as a maximal torus in $\fg_s$. By construction (see e.g. §\ref{s:compindpres}), this action coincides with (\ref{eq:1109:1747}). Thus, we may assume that $G=G_s$ is of semisimple rank $1$. Let $Z=Z(G)$ be the center of $G$ with Lie algebra $\fz$. Consider the isogeny $Z\times \bG_m\to T$, where the map $\bG_m\to T$ is the coroot of $G$. Then
\begin{equation}\label{eq:1109:2010}
D(B,\chi\backslash B/B,\chi)\simeq D(T,\chi\backslash T/T,\chi)\simeq D(Z,\chi\backslash Z/Z,\chi)\otimes D(\bG_m,\chi\backslash \bG_m/\bG_m,\chi),
\end{equation}

\noindent and similarly for its renormalized versions. It follows from the definition that $s$ acts on 
\[
\on{Sym}(\ft^*[-1]\oplus \ft^*[-2])\simeq \on{Sym}(\fz^*[-1]\oplus \fz^*[-2])\otimes \on{Sym}(k[-1]\oplus k[-2])
\]

\noindent by acting trivially on the first factor. Thus, it suffices to show that $s$ acts on $k[-1]\oplus k[-2]$ by $-1$.

\step

Let us first show that $s$ acts by $-1$ on $k[-2]$. If it did not, $s$ would act trivially on $\on{Sym}(\ft^*[-2])$. We claim that this is impossible. Indeed, suppose not. Then $\on{Sym}(\ft^*[-2])$ is central in $\on{HH}_*(\sH_{G,\chi}^{\on{ren}})$. On the other hand, acting by $k[-2]$ provides an isomorphism
\[
H^0(\bP^1)\to H^2(\bP^1),
\]

\noindent of $W$-representations, where:
\begin{itemize}
    \item We identify $\bP^1=G/B$;

    \item We realize $H^*(G/B)$ as the $!$-fiber at $1$ of the Springer sheaf $\on{Spr}_{\chi}$ with its $\on{HH}_*(\sH_{G,\chi})$-module structure.
\end{itemize}

\noindent However, this is a contradiction: by Proposition \ref{p:Springeract3}, $H^0(\bP^1)$ identifies with the sign representation of $W$, while $H^2(\bP^1)$ identifies with the trivial representation.

\step

Finally, let us show that the action of $s$ on $k[-1]$ is given by $-1$. The $S^1$-action on $\on{Sym}(\ft^*[-1]\oplus \ft^*[-2])$ (cf. Section \ref{S: PCH} below) respects the decomposition (\ref{eq:1109:2010}) and is $W$-linear. Moreover, it maps $\ft^*[-2]$ isomorphically onto $\ft^*[-1]$.\footnote{This amounts to showing that the $S^1$-action on $\on{HH}_*(\on{Sym}(k[-2]))$ is non-trivial in degree two. This is easily verified interpreting the $S^1$-action as the Connes operator.} We conclude by Step 3.

\end{proof}

\subsection{Periodic cyclic homology}\label{S: PCH}

In this subsection, we compute periodic cyclic homology of certain categories of interest.

\subsubsection{} Let $\sC$ be a dualizable DG category. As explained in \cite[§3]{ben2012loop}, \cite[§4.3]{raskin2018dundas}, the Hochschild homology of $\sC$ comes equipped with an $S^1=\bB \bZ$-action:
\[
S^1\curvearrowright \on{HH}_*(\sC).
\]

\noindent This is equivalent to an action of the algebra (see \cite[Cor. 3.14]{ben2012loop})
\[
C_{\bullet}(S^1)=k[\epsilon]/(\epsilon^2),\;\; \vert\epsilon\vert=-1.
\]

\subsubsection{}\label{s:hp} Taking invariants, we obtain an action
\[
k[u]=\on{End}_{C_{\bullet}(S^1)\on{-mod}}(k)\curvearrowright  \on{HH}_*(\sC)^{S^1},\;\; \vert u\vert=2.
\]

\noindent Let
\[
\on{HP}_*(\sC):=\on{HH}_*(\sC)^{tS^1}=\on{HH}_*(\sC)^{S^1}\underset{k[u]}{\otimes}k[u,u^{-1}].
\]

\noindent We refer to $\on{HP}_*(\sC)$ as the periodic cyclic homology of $\sC$.

\subsubsection{} The construction of $\on{HP}_*$ is functorical under continuous functors that admit continuous right adjoints. In particular, if $c\in \sC$ is a compact object, the Chern character $[c]\in\on{HH_0}$ (see (\ref{eq:chern''})) naturally yields an element of $\on{HP}_0(\sC)$.

As such, if $\sC$ is compactly generated, we get a map of $k$-vector spaces:
\[
K_0(\sC)_k\to \on{HP}_0(\sC).
\]
\begin{example}
When $\sC=A\on{-mod}$ for a dga $A$, $\on{HP}_*(A\on{-mod})$ coincides with $\on{HP}_*(A)$, the usual periodic cyclic homology of $A$.
\end{example}
\begin{lem}\label{l:finhp}
Let $A\in\on{AssocAlg}(\on{Vect}^{\heartsuit})$ be a finite-dimensional associative algebra over $k$. Then we have a natural isomorphism
\[
\underset{V\;\on{simple}}{\bigoplus}k \simeq \on{HP}_0(A),
\]

\noindent where the direct sum is taken over isomorphism classes of simple left $A$-modules.

\end{lem}

\begin{proof}
Let $\on{Jac}(A)$ be the Jacobson radical of $A$; that is, the ideal of elements that annihilate every simple left $A$-module. Let $B:=A/\on{Jac}(A)$. A simple consequence of Nakayama's lemma is that $\on{Jac}(A)$ is nilpotent. By \cite[Thm. II.5.1]{goodwillie1985cyclic}, the map
\[
A\onto B
\]

\noindent induces an isomorphism
\[
\on{HP}_*(A)\to \on{HP}_*(B).
\]

\noindent Moreover, by definition of $\on{Jac}(A)$, the simple left $A$-modules are in bijection with those of $B$. Thus, it suffices to give a natural isomorphism
\[
\underset{V\;\on{simple}}{\bigoplus}k \simeq \on{HP}_0(B).
\]

\noindent By Wedderburn's theorem, the map
\[
B\to \underset{V \on{simple}}{\bigoplus} \on{End}_k(V)
\]

\noindent is an isomorphism. Since $\on{End}_k(V)$ is Morita equivalent to $k$, the lemma follows.
\end{proof}

\subsubsection{} We now turn our study to the category $\sW_{e,\lambda}\on{-mod}^{\on{fin}}$. Recall the equivalence of Proposition \ref{p:keytoeverythingrs}:
\begin{equation}\label{eq:ssdescrip'''}
\sW_{e,\lambda}\on{-mod}^{\on{fin}}\simeq \sW_{e,\lambda}\on{-mod}\underset{\sH_{G,\chi}^{\on{ren}}}{\otimes} \sH_{G,\chi,\overline{\bO}}^{\on{ren}}.
\end{equation}

\noindent It follows from Proposition \ref{l:stable} that the inclusion
\[
i: \sW_{e,\lambda}\on{-mod}^{\on{fin}}\into \sW_{e,\lambda}\on{-mod}
\]

\noindent admits a left adjoint $i^L$. Let 
\[
\sW_{e,\lambda}^{\on{fin}}:=i^L(\sW_{e,\lambda})\in \sW_{e,\lambda}\on{-mod}^{\on{fin}}.
\]

\noindent Then $\sW_{e,\lambda}^{\on{fin}}$ is a compact generator of $\sW_{e,\lambda}\on{-mod}^{\on{fin}}$ and lives in non-positive cohomological degrees. Moreover, the functor
\begin{equation}\label{eq:trulyfinal}
\on{Hom}_{\sW_{e,\lambda}\on{-mod}^{\on{fin}}}(\sW_{e,\lambda}^{\on{fin}},-): \sW_{e,\lambda}\on{-mod}^{\on{fin}}\to \on{Vect}
\end{equation}

\noindent is t-exact.

\subsubsection{} To avoid heavy notation, we write
\[
A:=\on{End}_{\sW_{e,\lambda}\on{-mod}^{\on{fin}}}(\sW_{e,\lambda}^{\on{fin}})
\]

\noindent for the (derived) endomorphisms of $\sW_{e,\lambda}^{\on{fin}}$. Note that $A$ is concentrated in non-positive cohomological degrees. Let
\[
A^0:=H^0(A).
\]

\noindent Since $\sW_{e,\lambda}^{\on{fin}}$ is compact, it follows that the map
\[
\sW_{e,\lambda}^{\on{fin}}\to H^0(\sW_{e,\lambda}^{\on{fin}})
\]

\noindent factors through a finite-dimensional subspace. Thus, $H^0(\sW_{e,\lambda}^{\on{fin}})$ must be finite-dimensional. In particular, $A^0$ is finite-dimensional too by exactness of the functor (\ref{eq:trulyfinal}).

\subsubsection{}
Since $\sW_{e,\lambda}^{\on{fin}}$ is a compact generator of $\sW_{e,\lambda}\on{-mod}^{\on{fin}}$, it follows that we have a t-exact equivalence
\begin{equation}
\sW_{e,\lambda}\on{-mod}^{\on{fin}}\simeq A\on{-mod}.
\end{equation}

\noindent In particular, 
\begin{equation}\label{eq:txasdasd}
\sW_{e,\lambda}\on{-mod}^{\on{fin},\heartsuit}\simeq A\on{-mod}^{\heartsuit}.
\end{equation}
\begin{lem}\label{l:kzerotohp}
The Chern character map
\[
K_0(\sW_{e,\lambda}\on{-mod}^{\on{fin}, \heartsuit})_k\to \on{HP}_0(\sW_{e,\lambda}\on{-mod}^{\on{fin}})
\]

\noindent is an isomorphism.
\end{lem}

\begin{proof}
Since $A$ is connective, \cite[Thm. III.5.1]{goodwillie1985cyclic} says that the map
\[
A\to A^0
\]

\noindent induces an isomorphism
\[
\on{HP}_*(A)\to \on{HP}_*(A^0).
\]

\noindent Similarly, note that
\begin{equation}\label{eq:heartequivs}
\sW_{e,\lambda}\on{-mod}^{\on{fin},\heartsuit}\overset{(\ref{eq:txasdasd})}{\simeq} A\on{-mod}^{\heartsuit}\simeq A^0\on{-mod}^{\heartsuit}.
\end{equation}

\noindent By regularity of $\lambda$, a module $M\in \sW_{e,\lambda}\on{-mod}^{\on{fin},\heartsuit}$ is compact as an object of $\sW_{e,\lambda}\on{-mod}^{\on{fin}}$ if and only if it is finite-dimensional. Thus, we obtain a Chern character map
\[
K_0(\sW_{e,\lambda}\on{-mod}^{\on{fin},\heartsuit})_k\to \on{HP}_0(\sW_{e,\lambda}\on{-mod}^{\on{fin}})\simeq \on{HP}_0(A)\simeq \on{HP}_0(A^0),
\]

\noindent and we need to show that it is an isomorphism. However, this follows from (\ref{eq:heartequivs}) and (the construction of) Lemma \ref{l:finhp} using that $A^0$ is finite-dimensional.
\end{proof}

\subsubsection{} Let
\[
D(B,\chi\backslash G/(B,\chi)\on{-mon})
\]

\noindent be the full subcategory of $D(B,\chi\backslash G/N)$ generated under colimits by the essential image of the forgetful functor
\[
\sH_{G,\chi}\to D(B,\chi\backslash G/N).
\]

\noindent Note that $D(B,\chi\backslash G/(B,\chi)\on{-mon})$ carries a natural t-structure making the inclusion
\[
D(B,\chi\backslash G/(B,\chi)\on{-mon})\into D(B,\chi\backslash G/N)
\]

\noindent t-exact.

\subsubsection{} Similarly, define the category
\[
D_{\overline{\bO}}(B,\chi\backslash G/(B,\chi)\on{-mon})
\]

\noindent to be the full subcategory of $D(B,\chi\backslash G/N)$ generated under colimits by the essential image of the forgetful functor
\[
\sH_{G,\chi,\overline{\bO}}\to D(B,\chi\backslash G/N).
\]

\noindent We need an analogue of Lemma \ref{l:kzerotohp} for this category. Let
\[
K_0(D(B,\chi\backslash G/(B,\chi)\on{-mon}))_{\overline{\bO},k}
\]

\noindent be the Grothendieck group of the abelian category of coherent D-modules $\cF\in D_{\overline{\bO}}(B,\chi\backslash G/(B,\chi)\on{-mon})^{\heartsuit}$, base-changed to $k$. Note that each such object $\cF$ is compact in $D(B,\chi\backslash G/N)$.

We have an evident isomorphism:
\[
K_0(\sH_{G,\chi}^{\on{ren}})_{\overline{\bO},k}\overset{\simeq}{\to} K_0(D(B,\chi\backslash G/(B,\chi)\on{-mon}))_{\overline{\bO},k}.
\]

\begin{lem}\label{l:hpdescrip}
We have a canonical isomorphism
\[
\on{HP}_*(D_{\overline{\bO}}(B,\chi\backslash G/(B,\chi)\on{-mon}))\simeq K_0(D(B,\chi\backslash G/(B,\chi)\on{-mon}))_{\overline{\bO},k}\otimes k[u,u^{-1}],\;\; \vert u \vert=2
\]

\noindent such that the Chern character map
\[
K_0(\sH_{G,\chi}^{\on{ren}})_{\overline{\bO},k}\simeq K_0(D(B,\chi\backslash G/(B,\chi)\on{-mon}))_{\overline{\bO},k}\to \on{HP}_0(D_{\overline{\bO}}(B,\chi\backslash G/(B,\chi)\on{-mon}))
\]

\noindent is an isomorphism.
\end{lem}

\begin{proof}
The proof is similar to that of Lemma \ref{l:kzerotohp}. First, the category
\[
D(B,\chi\backslash G/N)\simeq U(\fg)_{\lambda}\on{-mod}^N
\]

\noindent is the derived category of its heart (see \cite{beilinson2003tilting} or \cite[Prop. A.2]{gannon2023universal}). Let $P\in D(B,\chi\backslash G/(B,\chi)\on{-mon})^{\heartsuit}$ be a compact projective generator of $D(B,\chi\backslash G/(B,\chi)\on{-mon})$.\footnote{E.g., one could take the direct sum over the projective covers of the simples in $D(B,\chi\backslash G/(B,\chi)\on{-mon})^{\heartsuit}$.} By Proposition \ref{l:stable}, the inclusion
\[
D_{\overline{\bO}}(B,\chi\backslash G/(B,\chi)\on{-mon})\simeq \sH_{G,\chi,\overline{\bO}}\underset{\sH_{G,\chi}}{\otimes} D(B,\chi\backslash G/(B,\chi)\on{-mon})\into D(B,\chi\backslash G/(B,\chi)\on{-mon})
\]

\noindent admits a left adjoint. Let us write $P^L$ for the image of $P$ under this left adjoint. Then we have a t-exact equivalence
\[
D_{\overline{\bO}}(B,\chi\backslash G/(B,\chi)\on{-mon})\simeq B\on{-mod},
\]

\noindent where $B=\on{End}_{D(B,\chi\backslash G/(B,\chi)\on{-mon})}(P^L)$. The same argument as in the proof of Lemma \ref{l:kzerotohp} now applies: the algebra $B$ is connective and $B^0:=H^0(B)$ is finite-dimensional. Thus, we get an equivalence of abelian categories
\begin{equation}\label{eq:equivhearts2}
D_{\overline{\bO}}(B,\chi\backslash G/(B,\chi)\on{-mon})^{\heartsuit}\simeq B^0\on{-mod}^{\heartsuit}.
\end{equation}

\noindent It suffices to show that the map
\[
K_0(D(B,\chi\backslash G/(B,\chi)\on{-mon}))_{\overline{\bO},k}\to \on{HP}_0(D_{\overline{\bO}}(B,\chi\backslash G/(B,\chi)\on{-mon}))\simeq \on{HP}_0(B)\simeq \on{HP}_0(B^0)
\]

\noindent is an isomorphism (here, the last equality is again \cite[Thm. III.5.1]{goodwillie1985cyclic}). However, this follows from (\ref{eq:equivhearts2}) and Lemma \ref{l:finhp}.
\end{proof}

\subsection{Proof of classification theorem - preliminaries}

\subsubsection{} Let $\on{Spr}_{\chi}\in D(G\overset{\on{ad}}{/}G)$ be the ($\chi$-twisted) Grothendieck Springer sheaf. Recall that it is semisimple. We let
\[
\on{Spr}_{\chi,\overline{\bO}}\subset \on{Spr}_{\chi}
\]

\noindent be the direct sum of all summands whose singular support is contained in $\overline{\bO}$. We define $\on{Spr}_{\chi,\partial\bO}$ similarly. Let 
\[
\on{Spr}_{\chi,\bO}:=\on{Spr}_{\chi,\overline{\bO}}/\on{Spr}_{\chi,\partial\bO}.
\]

\subsubsection{}\label{ss:factorthru} We remind that by Lemma \ref{l:HHiscoho} and \ref{l:HHiscoeff}, we have identifications
\[
H^*(\cB_e)[2\on{dim}\cB_e]\simeq \on{HH}_*(\sW_{e,\lambda}\on{-mod})\simeq \on{coeff}_{\bO}(\on{Spr}_{\chi}),
\]

\noindent and these carry an action of $W_{[\lambda]}$ by means of (\ref{eq:10/09:1716}). Moreover, by Theorem \ref{t:whitismicrostalkG}, we get an injection
\begin{equation}\label{eq:injofbo}
\on{coeff}_{\bO}(\on{Spr}_{\chi,\bO})\into H^{\on{top}}(\cB_e).
\end{equation}

\noindent Here, we consider $H^{\on{top}}(\cB_e)$ as sitting in cohomological degree zero.
\begin{prop}\label{p:finreductionstep}
The maps
\[
H^*(\cB_e)\underset{k[W_{[\lambda]}]}{\otimes}K_0(\sH_{G,\chi}^{\on{ren}})_{\overline{\bO},k}[2\on{dim}\cB_e] \to H^{\on{top}}(\cB_e)\underset{k[W_{[\lambda]}]}{\otimes} K_0(\sH_{G,\chi}^{\on{ren}})_{\overline{\bO},k},
\]
\[
H^*(\cB_e)\underset{k[W_{[\lambda]}]}{\otimes} K_0(\sH_{G,\chi}^{\on{ren}})_{\bO,k}\to H^*(\cB_e)\underset{k[W_{[\lambda]}]}{\otimes} K_0(\sH_{G,\chi}^{\on{ren}})_{\overline{\bO},k}
\]

\noindent are isomorphisms.
\end{prop}

\begin{proof}

\step 
Let us prove the first assertion. Let
\[
H^{<\on{top}}(\cB_e)=\tau^{<\on{top}}(H^*(\cB_e))
\]

\noindent be the truncation complex. We need to show that
\[
H^{<\on{top}}(\cB_e)\underset{k[W_{[\lambda]}]}{\otimes} K_0(\sH_{G,\chi}^{\on{ren}})_{\overline{\bO},k}=0.
\]

\noindent Let $\on{Spr}_{\chi,\neq \bO}=\on{Spr}_{\chi}/\on{Spr}_{\chi,\bO}$. By (\ref{eq:injofbo}), $H^{<\on{top}}(\cB_e)[2\on{dim}\cB_e]$ is a direct summand of $\on{coeff}_{\bO}(\on{Spr}_{\chi,\neq \bO})$. Thus, it suffices to show that
\begin{equation}\label{eq:1109:0034}
\on{coeff}_{\bO}(\on{Spr}_{\chi,\neq\bO}\underset{k[W_{[\lambda]}]}{\otimes} K_0(\sH_{G,\chi}^{\on{ren}})_{\overline{\bO},k})=0.
\end{equation}

\noindent Note that the image of the action map
\[
\on{Spr}_{\chi}\otimes K_0(\sH_{G,\chi}^{\on{ren}})_{\overline{\bO},k}\to \on{Spr}_{\chi}
\]

\noindent has singular support contained in $\overline{\bO}$. As such, the singular support of the sheaf
\[
\on{Spr}_{\chi,\neq\bO}\underset{k[W_{[\lambda]}]}{\otimes} K_0(\sH_{G,\chi}^{\on{ren}})_{\overline{\bO},k}
\]

\noindent is contained in $\partial\bO$. Thus, the vector space (\ref{eq:1109:0034}) vanishes by Theorem \ref{t:whitismicrostalkG}.

\step 

For the second assertion, we need to show that
\[
H^*(\cB_e)\underset{k[W_{[\lambda]}]}{\otimes} K_0(\sH_{G,\chi}^{\on{ren}})_{\partial\bO,k}=\on{coeff}_{\bO}(\on{Spr}_{\chi}\underset{k[W_{[\lambda]}]}{\otimes} K_0(\sH_{G,\chi}^{\on{ren}})_{\partial\bO,k})[-2\on{dim}\cB_e]
\]

\noindent vanishes. The singular support of the sheaf $\on{Spr}_{\chi}\underset{k[W_{[\lambda]}]}{\otimes} K_0(\sH_{G,\chi}^{\on{ren}})_{\partial\bO,k}$ is contained in $\partial\bO$ and so the above vector space similarly vanishes by Theorem \ref{t:whitismicrostalkG}.

\end{proof}

\subsubsection{} The action
\[
D(B,\chi\backslash G/(B,\chi)\on{-mon})\curvearrowleft T
\]

\noindent stabilizes $D_{\overline{\bO}}(B,\chi\backslash G/(B,\chi)\on{-mon})$, and we have
\[
\sH_{G,\chi,\overline{\bO}}\simeq D_{\overline{\bO}}(B,\chi\backslash G/(B,\chi)\on{-mon})^{T,\chi}.
\]

\noindent Note that the category
\[
D(T,\chi\backslash T/T,\chi)^{\on{ren}}
\]

\noindent acts on $\sH_{G,\chi,\overline{\bO}}^{\on{ren}}$ on the right. As a particular case of Proposition \ref{p:genfull} (with $G=T$), we get:
\[
D_{\overline{\bO}}(B,\chi\backslash G/(B,\chi)\on{-mon})\simeq \sH_{G,\chi,\overline{\bO}}^{\on{ren}}\underset{D(T,\chi\backslash T/T,\chi)^{\on{ren}}}{\otimes} \on{Vect.}
\]

\noindent Here, we identify $\on{Vect}\simeq D(T,\chi\backslash T)$. Thus:
\begin{lem}\label{l:finalmanip}
We have an isomorphism
\[
\on{HH}_*(D_{\overline{\bO}}(B,\chi\backslash G/(B,\chi)\on{-mon}))\simeq \on{HH}_*(\sH_{G,\chi,\overline{\bO}}^{\on{ren}})\underset{\on{Sym}(\ft^*[-1]\oplus \ft^*[-2])}{\otimes} k.
\]
\end{lem}

\begin{proof}

We need to show that whenever $A$ is a commutative dg-algebra and $\sM,\sN$ are dualizable module categories for $A\on{-mod}$, then we have
\[
\on{HH}_*(\sM\underset{A\on{-mod}}{\otimes} \sN)\simeq \on{HH}_*(\sM)\underset{\on{HH}_*(A)}{\otimes}\on{HH}_*(\sN).
\]

\noindent This is clearly sufficient since $D(T,\chi\backslash T/T,\chi)$ is compactly generated by its unit, and so is of form $A\on{-mod}$, where $A=C^*(\bB T)$. We thank Nick Rozenblyum for sketching the following argument:

It is explained in \cite{gaitsgory2022toy}[\S 3.8] that for any rigid symmetric monoidal category $\sR$ and dualizable module categories $\sM,\sN$, we may attach objects
\[
\on{HH}_*^{\on{enh}}(\sM),\on{HH}_*^{\on{enh}}(\sN)\in \textbf{HH}_*^{\on{enh}}(\sR).
\]

\noindent Here, $\textbf{HH}_*^{\on{enh}}(\sR)=\sR\underset{\sR\otimes\sR}{\otimes}\sR $ is defined to be the categorical Hochschild homology of $\sR$, which we consider a symmetric monoidal category with unit $\mathbf{1}$. Moreover, the above attachment is monoidal. By Theorem 3.8.5 of \emph{loc.cit}, the functor
\[
\on{Hom}_{\textbf{HH}_*^{\on{enh}}(\sR)}(\mathbf{1},-): \textbf{HH}_*^{\on{enh}}(\sR)\to \on{Vect}
\]

\noindent satisfies that $\on{Hom}_{\textbf{HH}_*^{\on{enh}}(\sR)}(\mathbf{1},\mathbf{1})\simeq \on{HH}_*(\sR)$ as algebras, and that $\on{Hom}_{\textbf{HH}_*^{\on{enh}}(\sR)}(\mathbf{1},\on{HH}_*^{\on{enh}}(\sM))=\on{HH}_*(\sM)$. As such that above functor upgrades to a functor
\begin{equation}\label{eq:obstohh}
\on{Hom}_{\textbf{HH}_*^{\on{enh}}(\sR)}(\mathbf{1},-): \textbf{HH}_*^{\on{enh}}(\sR)\to \on{HH}_*(\sR)\on{-mod}.
\end{equation}

\noindent In general, the above functor is not symmetric monoidal and this is exactly the obstruction to the present lemma holding for an arbitrary rigid symmetric monoidal category. However, when $\sR=A\on{-mod}$, we have $\textbf{HH}_*^{\on{enh}}(\sR)\simeq \on{HH}_*(A)\on{-mod}$, and the functor (\ref{eq:obstohh}) is a monoidal equivalence.

\end{proof}

\subsubsection{}\label{s:1009} The fully faithful embedding
\[
D(B,\chi\backslash B/B,\chi)^{\on{ren}}\into \sH_{G,\chi}^{\on{ren}}
\]

\noindent induces the map
\[
\on{Sym}(\ft^*[-1]\oplus \ft^*[-2])\simeq \on{HH}_*(D(B,\chi\backslash B/B,\chi)^{\on{ren}})\to \on{HH}_*(\sH_{G,\chi}^{\on{ren}}).
\]

\noindent By (the proof of) Lemma \ref{l:HHisgroupalg}, the map
\begin{equation}\label{eq:1009:1916}
\on{HH}_*(\sH_{G,\chi}^{\on{ren}})\to k\underset{\on{Sym}(\ft^*[-1]\oplus \ft^*[-2])}{\otimes}\on{HH}_*(\sH_{G,\chi}^{\on{ren}})\simeq k[W_{[\lambda]}]
\end{equation}

\noindent provides a splitting of the Chern character map $k[W_{[\lambda]}]\simeq K_0(\sH_{G,\chi})_k\to \on{HH}_*(\sH_{G,\chi}^{\on{ren}})$. Moreover, the map (\ref{eq:1009:1916}) has a natural structure of an algebra map coming from the formality of $\on{HH}_*(\sH_{G,\chi}^{\on{ren}})$, cf. \cite{li2023derived}[Thm. 2.10.11 + Thm. 3.1.1].\footnote{Another way to see that (\ref{eq:1009:1916} has the structure of an algebra map without using formality is the following. By passing to Hochschild homology, the action functor $\sH_{G,\chi}^{\on{ren}}\otimes D(B,\chi\backslash G/(B,\chi)\on{-mon})\to D(B,\chi\backslash G/(B,\chi)\on{-mon})$ gives an algebra map $\on{HH}_*(\sH_{G,\chi}^{\on{ren}})\to \on{End}_k(\on{HH}_*( D(B,\chi\backslash G/(B,\chi)\on{-mon})))\simeq \on{End}_k(k[W_{[\lambda]}])$. This action is compatible with the natural action of $k[W_{[\lambda]}]$ on the right, coming from identifying $k[W_{[\lambda]}]$ with $K_0(-)_k$ of (a suitably renormaliezd version of) $D((B,\chi)\on{-mon}\backslash G/(B,\chi)\on{-mon})$. Thus, we get a map $\on{HH}_*(\sH_{G,\chi}^{\on{ren}})\to \on{End}_{k[W_{[\lambda]}]}(k[W_{[\lambda]}])\simeq k[W_{[\lambda]}]$.}

\subsubsection{} Suppose $V$ is a finite-dimensional representation of $W_{[\lambda]}$ concentrated in cohomological degree zero. Consider $V$ as a module for $\on{HH}_*(\sH_{G,\chi}^{\on{ren}})$ via the map (\ref{eq:1009:1916}). The $S^1$-actions\footnote{Equivalently, $C_{\bullet}(S^1)$-action.} on $\on{HH}_*(\sH_{G,\chi}^{\on{ren}})$ and $\on{HH}_*(\sH_{G,\chi,\overline{\bO}}^{\on{ren}})$ induce an action on the vector space
\[
V\underset{\on{HH}_*(\sH_{G,\chi}^{\on{ren}})}{\otimes}\on{HH}_*(\sH_{G,\chi,\overline{\bO}}^{\on{ren}}).
\]

\noindent by acting trivially on $V$.

Recall that for $W\in C_{\bullet}(S^1)\on{-mod}$, we write 
\[
W^{tS^1}:=\on{Hom}_{C_{\bullet}(S^1)\on{-mod}}(k,W)\underset{k[u]}{\otimes}k[u,u^{-1}],
\]

\noindent cf. §\ref{s:hp}.

\begin{lem}\label{l:1009:2134}
We have a canonical isomorphism
\[
\big(V\underset{\on{HH}_*(\sH_{G,\chi}^{\on{ren}})}{\otimes}\on{HH}_*(\sH_{G,\chi,\overline{\bO}}^{\on{ren}})\big)^{tS^1}\simeq (V\underset{k[W_{[\lambda]}]}{\otimes}K_0(\sH_{G,\chi}^{\on{ren}})_{\overline{\bO},k})\otimes k[u,u^{-1}].
\]
\end{lem}
\begin{proof}
Note that
\[
V\underset{\on{HH}_*(\sH_{G,\chi}^{\on{ren}})}{\otimes}\on{HH}_*(\sH_{G,\chi,\overline{\bO}}^{\on{ren}})\simeq V\underset{k[W_{[\lambda]}]}{\otimes}\big(k\underset{\on{Sym}(\ft^*[-1]\oplus \ft^*[-2])}{\otimes}\on{HH}_*(\sH_{G,\chi,\overline{\bO}}^{\on{ren}})\big).
\]

\noindent Consequently: 
\[
\big(V\underset{\on{HH}_*(\sH_{G,\chi}^{\on{ren}})}{\otimes}\on{HH}_*(\sH_{G,\chi,\overline{\bO}}^{\on{ren}})\big)^{tS^1}\simeq V\underset{k[W_{[\lambda]}]}{\otimes}\big(k\underset{\on{Sym}(\ft^*[-1]\oplus \ft^*[-2])}{\otimes}\on{HH}_*(\sH_{G,\chi,\overline{\bO}}^{\on{ren}})\big)^{tS^1}.
\]

\noindent Thus, it suffices to show that
\[
\big(k\underset{\on{Sym}(\ft^*[-1]\oplus \ft^*[-2])}{\otimes}\on{HH}_*(\sH_{G,\chi,\overline{\bO}}^{\on{ren}})\big)^{tS^1}\simeq K_0(\sH_{G,\chi}^{\on{ren}})_{\overline{\bO},k}\otimes k[u,u^{-1}]
\]

\noindent as $W_{[\lambda]}$-representations. However, this follows by combining Lemma \ref{l:finalmanip} and Lemma \ref{l:hpdescrip}.
\end{proof}

\begin{lem}\label{l:trivcircleaction}
The $S^1$-action on $H^*(\cB_e)\simeq \on{HH}_*(\sW_{e,\lambda}\on{-mod})$ is trivial; that is, the action factors through the augmentation of $C_{\bullet}(S^1)$.
\end{lem}

\begin{proof}
We claim this follows more generally for any $C_{\bullet}(S^1)$-module whose underlying vector space has cohomology concentrated in even degrees.

Let $V$ be such a module. Assume that $H^*(V)$ is concentrated in two cohomological degrees; the general case follows by induction. We have a triangle:
\[
V_1\to V\to V_2,
\]

\noindent where $V_i$ is a $C_{\bullet}(S^1)$ concentrated in a single cohomological degree. As such, $V_i$ is a direct sum of the augmentation module $k$. The above triangle defines elements in $\on{Ext}^n_{C_{\bullet}(S^1)}(k,k)\simeq H^n(\on{Sym}(k[-2]))$ for some odd integer $n$ and hence vanishes.
\end{proof}

\begin{cor}\label{c:6.5}
We have a canonical isomorphism
\[
\big(H^*(\cB_e)\underset{\on{HH}_*(\sH_{G,\chi}^{\on{ren}})}{\otimes} \on{HH}_*(\sH_{G,\chi,\overline{\bO}}^{\on{ren}})\big)^{tS^1}\simeq (H^{\on{top}}(\cB_e)\underset{k[W_{[\lambda]}]}{\otimes}K_0(\sH_{G,\chi}^{\on{ren}})_{\bO,k})\otimes k[u,u^{-1}].
\]
\end{cor}

\begin{proof}
We claim that we have a map
\begin{equation}\label{eq:1009:2106}
H^{\on{top}}(\cB_e)\to H^*(\cB_e)[2\on{dim}\cB_e]
\end{equation}

\noindent of $\on{HH}_*(\sH_{G,\chi}^{\on{ren}})$-modules inducing an isomorphism after taking $H^0$. Here, we consider $H^{\on{top}}(\cB_e)$ as a $\on{HH}_*(\sH_{G,\chi}^{\on{ren}})$-module via its augmentation to $k[W_{[\lambda]}]$ as in §\ref{s:1009}. Indeed, the existence of (\ref{eq:1009:2106}) follows from \cite[Prop. 4.4.6]{dhillonchen}\footnote{See also \cite[§3]{dwyer2006duality}.} using that $\on{HH}_*(\sH_{G,\chi}^{\on{ren}})$ is a coconnective dg-algebra with $\on{HH}_0(\sH_{G,\chi}^{\on{ren}})$ semisimple, cf. Lemma \ref{l:HHisgroupalg}. 

Applying the uniqueness assertion of \cite[Prop.4.4.6(2)]{dhillonchen}, we see that the action of $\on{HH}_*(\sH_{G,\chi}^{\on{ren}})$ on a vector space concentrated in a single cohomological degree necessarily factors through $\on{HH}_*(\sH_{G,\chi}^{\on{ren}})\to k[W_{[\lambda]}]$.

Iterating, we obtain a filtration on $H^*(\cB_e)$ of $\on{HH}_*(\sH_{G,\chi}^{\on{ren}})$-modules with the associated graded pieces being concentrated in a single cohomological degree. Since the $S^1$-action on $H^*(\cB_e)$ is trivial by Lemma \ref{l:trivcircleaction}, the induced filtration on
\begin{equation}\label{eq:1009:2131}
\big(H^*(\cB_e)\underset{\on{HH}_*(\sH_{G,\chi}^{\on{ren}})}{\otimes} \on{HH}_*(\sH_{G,\chi,\overline{\bO}}^{\on{ren}})\big)^{tS^1}
\end{equation}

\noindent has associated graded pieces of the form
\begin{equation}\label{eq:1009:2130}
\big(V\underset{\on{HH}_*(\sH_{G,\chi}^{\on{ren}})}{\otimes} \on{HH}_*(\sH_{G,\chi,\overline{\bO}}^{\on{ren}})\big)^{tS^1},
\end{equation}

\noindent with $V$ concentrated in a single cohomological degree. Since $H^*(\cB_e)$ is concentrated in even degrees, by Lemma \ref{l:1009:2134}, so are the vector spaces (\ref{eq:1009:2130}). In particular, the boundary maps in the long exact sequence associated to the filtration on (\ref{eq:1009:2131}) vanish, and so we get an isomorphism of right $W_{[\lambda]}$-representations:
\[
\big(H^*(\cB_e)\underset{\on{HH}_*(\sH_{G,\chi}^{\on{ren}})}{\otimes} \on{HH}_*(\sH_{G,\chi,\overline{\bO}}^{\on{ren}})\big)^{tS^1}\simeq \underset{V}{\bigoplus} \big(V\underset{\on{HH}_*(\sH_{G,\chi}^{\on{ren}})}{\otimes} \on{HH}_*(\sH_{G,\chi,\overline{\bO}}^{\on{ren}})\big)^{tS^1}
\]
\[
\overset{\on{Lem.}\;\ref{l:1009:2134}}{\simeq} \underset{V}{\bigoplus} (V\underset{k[W_{[\lambda]}]}{\otimes}K_0(\sH_{G,\chi}^{\on{ren}})_{\overline{\bO},k})\otimes k[u,u^{-1}]\simeq (H^*(\cB_e)\underset{k[W_{[\lambda]}]}{\otimes}K_0(\sH_{G,\chi}^{\on{ren}})_{\overline{\bO},k})\otimes k[u,u^{-1}]
\]
\[
\overset{\on{Prop.}\;\ref{p:finreductionstep}}{\simeq} (H^{\on{top}}(\cB_e)\underset{k[W_{[\lambda]}]}{\otimes}K_0(\sH_{G,\chi}^{\on{ren}})_{\bO,k})\otimes k[u,u^{-1}].
\]

\end{proof}

\subsection{Proof of classification theorem}

Before proving Theorem \ref{t:classification}, we need versions of the results in the previous subsection where we replace $\sH_{G,\chi,\overline{\bO}}^{\on{ren}}$ by a slightly smaller category, which we denote $\sH_{G,\chi,\overline{\bO}}^{\on{ren, access}}$, following \cite{arinkin2020stack}. The main property of the latter is that it comes with a canonical embedding
\[
\sH_{G,\chi,\overline{\bO}}^{\on{ren,access}}\to \sH_{G,\chi,\overline{\bO}}^{\on{ren}},
\]

\noindent and the map
\[
\sH_{G,\chi,\overline{\bO}}^{\on{ren,access}}\to \sH_{G,\chi}^{\on{ren}}
\]

\noindent admits a continuous right adjoint. As such, we get an induced map on Hochschild homology:
\[
\on{HH}_*(\sH_{G,\chi,\overline{\bO}}^{\on{ren,access}})\to \on{HH}_*(\sH_{G,\chi}^{\on{ren}}).
\]

\noindent We emphasize that such a map does not exist if we replace $\sH_{G,\chi,\overline{\bO}}^{\on{ren,access}}$ with $\sH_{G,\chi,\overline{\bO}}^{\on{ren}}$.

\subsubsection{} Let 
\[
\sH_{G,\chi,\overline{\bO}}^{\on{ren,coh}}\subset \sH_{G,\chi}^{\on{ren}}
\]

\noindent be the small category consisting of bounded complexes, each of whose cohomologies are coherent D-modules whose singular support is contained in $\overline{\bO}$. Alternatively, this is the small subcategory of $\sH_{G,\chi,\overline{\bO}}^{\on{ren}}$ consisting of objects that are compact in $\sH_{G,\chi}^{\on{ren}}$. We let
\[
\sH_{G,\chi,\overline{\bO}}^{\on{ren,access}}:=\on{Ind}(\sH_{G,\chi,\overline{\bO}}^{\on{ren,coh}})\subset \sH_{G,\chi}^{\on{ren}}
\]

\noindent denote the corresponding Ind-completion. By construction, the embedding $\sH_{G,\chi,\overline{\bO}}^{\on{ren,access}}\into \sH_{G,\chi}^{\on{ren}}$ preserves compact objects and is a two-sided ideal. Moreover, we naturally have $\sH_{G,\chi,\overline{\bO}}^{\on{ren,access}}\subset \sH_{G,\chi,\overline{\bO}}^{\on{ren}}$.

\subsubsection{} Let $\sH_{G,\chi,\overline{\bO}}^{\on{ren,inaccess}}\subset \sH_{G,\chi,\overline{\bO}}^{\on{ren}}$ denote the right orthogonal complement to $\sH_{G,\chi,\overline{\bO}}^{\on{ren,access}}$. That is, we have a recollement of categories:
\[
\sH_{G,\chi,\overline{\bO}}^{\on{ren,access}}\rightleftarrows \sH_{G,\chi,\overline{\bO}}^{\on{ren}}\rightleftarrows \sH_{G,\chi,\overline{\bO}}^{\on{ren,inaccess}}.
\]

\noindent Similarly, for a dualizable $\sH_{G,\chi}^{\on{ren}}$-module category $\sM$, we define:
\[
\sM_{\overline{\bO}}:=\sM\underset{\sH_{G,\chi}^{\on{ren}}}{\otimes}\sH_{G,\chi,\overline{\bO}}^{\on{ren}},\;\;\sM_{\overline{\bO}}^{\on{access}}:=\sM\underset{\sH_{G,\chi}^{\on{ren}}}{\otimes}\sH_{G,\chi,\overline{\bO}}^{\on{ren,access}},\;\; \sM_{\overline{\bO}}^{\on{inaccess}}:=\sM\underset{\sH_{G,\chi}^{\on{ren}}}{\otimes}\sH_{G,\chi,\overline{\bO}}^{\on{ren,inaccess}}.
\]

\noindent As above, we have a recollement:
\[
\sM_{\overline{\bO}}^{\on{access}}\rightleftarrows \sM_{\overline{\bO}}\rightleftarrows \sM_{\overline{\bO}}^{\on{inaccess}}.
\]

\begin{lem}\label{l:inaccessdies}
The induced map 
\[
\on{HP}_*(\sM_{\overline{\bO}}^{\on{access}})\underset{k[W_{[\lambda]}]}{\otimes} K_0(\sH_{G,\chi}^{\on{ren}})_{\overline{\bO},k}\to\on{HP}_*(\sM_{\overline{\bO}})\underset{k[W_{[\lambda]}]}{\otimes} K_0(\sH_{G,\chi}^{\on{ren}})_{\overline{\bO},k}
\]

\noindent is an isomorphism. Equivalently:
\[
\on{HP}_*(\sM_{\overline{\bO}}^{\on{inaccess}})\underset{k[W_{[\lambda]}]}{\otimes} K_0(\sH_{G,\chi}^{\on{ren}})_{\overline{\bO},k}=0.
\]
\end{lem}

\begin{rem}
We expect the stronger statement that
\[
\on{HP}_*(\sM_{\overline{\bO}}^{\on{inaccess}})=0
\]

\noindent is true.
\end{rem}

\begin{proof}
For each $[\cF]\in K_0(\sH_{G,\chi}^{\on{ren}})_{\overline{\bO},k}$, we get an endomorphism:
\[
\on{HP}_*(\sM_{\overline{\bO}}^{\on{inaccess}})\to \on{HP}_*(\sM_{\overline{\bO}}^{\on{inaccess}})
\]

\noindent induced by taking $\on{HP}_*$ of the endofunctor
\[
\sM_{\overline{\bO}}^{\on{inaccess}}\xrightarrow{-\star\cF}\sM_{\overline{\bO}}^{\on{inaccess}},
\]

\noindent where $\cF\in \sH_{G,\chi,\overline{\bO}}^{\on{ren},\heartsuit}$ represents the class $[\cF]$. It suffices to show that this functor is identically zero. However, this follows as $\cF$ is an object of $\sH_{G,\chi,\overline{\bO}}^{\on{ren,access}}$ and that both $\sH_{G,\chi,\overline{\bO}}^{\on{ren,access}}$ and $\sH_{G,\chi,\overline{\bO}}^{\on{ren,inaccess}}$ are $\sH_{G,\chi}^{\on{ren}}$-stable.
\end{proof}

\begin{rem}\label{r:JSRIP}
By exactly the same argument (replacing $\sH_{G,\chi,\overline{\bO}}^{\on{ren,inaccess}}$ with the right orthogonal complement to $\sH_{G,\chi,\overline{\bO}}^{\on{ren,access}}$ in all of $\sH_{G,\chi}^{\on{ren}}$), one obtains that the map
\[
\on{HP}_*(\sM_{\overline{\bO}}^{\on{access}})\underset{k[W_{[\lambda]}]}{\otimes} K_0(\sH_{G,\chi}^{\on{ren}})_{\overline{\bO},k}\to\on{HP}_*(\sM)\underset{k[W_{[\lambda]}]}{\otimes} K_0(\sH_{G,\chi}^{\on{ren}})_{\overline{\bO},k}
\]

\noindent is an isomorphism.
\end{rem}

\subsubsection{} When $\sM=\cW_{e,\lambda}\on{-mod}$, recall that by Proposition \ref{p:keytoeverythingrs}, we have:
\[
\sW_{e,\lambda}\on{-mod}^{\on{fin}}\simeq \sW_{e,\lambda}\on{-mod}\underset{\sH_{G,\chi}^{\on{ren}}}{\otimes} \sH_{G,\chi,\overline{\bO}}^{\on{ren}}.
\]

\noindent In this case, we write:
\[
\sW_{e,\lambda}\on{-mod}^{\on{fin,access}}:=\sW_{e,\lambda}\on{-mod}\underset{\sH_{G,\chi}^{\on{ren}}}{\otimes}\sH_{G,\chi,\overline{\bO}}^{\on{ren,access}}, \;\; \sW_{e,\lambda}\on{-mod}^{\on{fin,inaccess}}:=\sW_{e,\lambda}\on{-mod}\underset{\sH_{G,\chi}^{\on{ren}}}{\otimes}\sH_{G,\chi,\overline{\bO}}^{\on{ren,inaccess}}.
\]

\noindent By construction, the inclusion $\sW_{e,\lambda}\on{-mod}^{\on{fin,access}}\into \sW_{e,\lambda}\on{-mod}^{\on{fin}}$ induces an equivalence on objects cohomologically bounded from below.

\subsubsection{Classification theorem} Recall the statement of the classification theorem we are trying to prove: we need to show that the Chern character map
\[
K_0(\sW_{e,\lambda}\on{-mod}^{\on{fin},\heartsuit})_k\to \on{HH}_0(\sW_{e,\lambda}\on{-mod})=H^{\on{top}}(\cB_e)
\]

\noindent maps isomorphically onto 
\[
H^{\on{top}}(\cB_e)\underset{k[W_{[\lambda]}]}{\otimes} K_0(\sH_{G,\chi}^{\on{ren}})_{\bO,k}.
\]

\noindent We first prove:

\begin{prop}\label{p:thepdoddy}
The Chern character map factors through $H^{\on{top}}(\cB_e)\underset{k[W_{[\lambda]}]}{\otimes} K_0(\sH_{G,\chi}^{\on{ren}})_{\bO,k}\subset H^{\on{top}}(\cB_e)$. Moreover, the resulting map
\[
K_0(\sW_{e,\lambda}\on{-mod}^{\on{fin},\heartsuit})_k\to H^{\on{top}}(\cB_e)\underset{k[W_{[\lambda]}]}{\otimes} K_0(\sH_{G,\chi}^{\on{ren}})_{\bO,k}
\]

\noindent is injective.
\end{prop}

\begin{proof}

Recall that the inclusion $\sW_{e,\lambda}\on{-mod}^{\on{fin}}\into \sW_{e,\lambda}\on{-mod}$ admits a left adjoint: $\sW_{e,\lambda}\on{-mod}\to \sW_{e,\lambda}\on{-mod}^{\on{fin}}$. Note that the resulting map:
\[
H^*(\cB_e)[2\on{dim} \cB_e]\simeq \on{HH}_*(\cW_{e,\lambda}\on{-mod})\to \on{HH}_*(\cW_{e,\lambda}\on{-mod}^{\on{fin}})=\on{HH}_*(\cW_{e,\lambda}\on{-mod}\underset{\sH_{G,\chi}^{\on{ren}}}{\otimes}\sH_{G,\chi,\overline{\bO}}^{\on{ren}})
\]

\noindent factors as:
\[
H^*(\cB_e)[2\on{dim} \cB_e]\to \on{HH}_*(\cW_{e,\lambda}\on{-mod})\underset{\on{HH}_*(\sH_{G,\chi}^{\on{ren}})}{\otimes}\on{HH}_*(\sH_{G,\chi,\overline{\bO}}^{\on{ren}})\to \on{HH}_*(\cW_{e,\lambda}\on{-mod}^{\on{fin}}).
\]

\noindent Applying $(-)^{tS^1}$ and taking $H^0$, we get the map:
\[
H^{\on{top}}(\cB_e)\to H^0\bigg(\on{HH}_*(\cW_{e,\lambda}\on{-mod})\underset{\on{HH}_*(\sH_{G,\chi}^{\on{ren}})}{\otimes}\on{HH}_*(\sH_{G,\chi,\overline{\bO}}^{\on{ren}})\bigg)^{tS^1}\to 
\]
\[
\to \on{HP}_0(\cW_{e,\lambda}\on{-mod}^{\on{fin}}).
\]

By construction, precomposing the Chern character map yields the isomorphism of Lemma \ref{l:kzerotohp}. By Corollary \ref{c:6.5}, we have thus provided a surjective map 
\[
H^{\on{top}}(\cB_e)\underset{k[W_{[\lambda]}]}{\otimes}K_0(\sH_{G,\chi}^{\on{ren}})_{\bO,k}\onto K_0(\sW_{e,\lambda}\on{-mod}^{\on{fin},\heartsuit})_k
\]

\noindent along with a splitting. To prove that the Chern character map factors through
\[
H^{\on{top}}(\cB_e)\underset{k[W_{[\lambda]}]}{\otimes}K_0(\sH_{G,\chi}^{\on{ren}})_{\bO,k}\subset H^{\on{top}}(\cB_e),
\]

\noindent it suffices to show that if an irreducible $W_{[\lambda]}$-representation $V$ occurs in $K_0(\sH_{G,\chi}^{\on{ren}})_{\bO,k}$, then all copies of $V$ in $k[W_{[\lambda]}]$ appear in $K_0(\sH_{G,\chi}^{\on{ren}})_{\bO,k}$. This is the content of Lemma \ref{l:allcopies} below.
\end{proof}

\begin{lem}\label{l:allcopies}
Let $V\subset K_0(\sH_{G,\chi}^{\on{ren}})_{\bO,k}$ be an irreducible $W_{[\lambda]}$-representation. Then the map
\[
\on{Hom}_{W_{[\lambda]}}(V,K_0(\sH_{G,\chi}^{\on{ren}})_{\bO,k})\to \on{Hom}_{W_{[\lambda]}}(V,k[W_{[\lambda]}])
\]

\noindent is an isomorphism.
\end{lem}

\begin{proof}
We have a decomposition 

\[
k[W_{[\lambda]}]\simeq \underset{\bO}{\bigoplus}\; K_0(\sH_{G,\chi}^{\on{ren}})_{\bO,k}
\]

\noindent as $W_{[\lambda]}$-representations. We need to show that for two distinct orbits $\bO_1,\bO_2$, we have:
\[
\on{Hom}_{W_{[\lambda]}}(K_0(\sH_{G,\chi}^{\on{ren}})_{\bO_1,k},K_0(\sH_{G,\chi}^{\on{ren}})_{\bO_2,k})=0.
\]

\noindent Let $\phi\in \on{Hom}_{W_{[\lambda]}}(K_0(\sH_{G,\chi}^{\on{ren}})_{\bO_1,k},K_0(\sH_{G,\chi}^{\on{ren}})_{\bO_2,k})$. It suffices to show that for all orbits $\bO$, the map
\[
\on{id}\otimes \phi: K_0(\sH_{G,\chi}^{\on{ren}})_{\bO,k}\underset{k[W_{[\lambda]}]}{\otimes} K_0(\sH_{G,\chi}^{\on{ren}})_{\bO_1,k}\to K_0(\sH_{G,\chi}^{\on{ren}})_{\bO,k}\underset{k[W_{[\lambda]}]}{\otimes} K_0(\sH_{G,\chi}^{\on{ren}})_{\bO_2,k}
\]

\noindent is zero. In turn, it suffices to prove that for distinct orbits $\bO',\bO$, we have
\begin{equation}\label{eq:17151715}
K_0(\sH_{G,\chi}^{\on{ren}})_{\bO,k}\underset{k[W_{[\lambda]}]}{\otimes} K_0(\sH_{G,\chi}^{\on{ren}})_{\bO',k}=0.
\end{equation}

\noindent This follows from the fact that any object in the image of the convolution functor
\[
\sH_{G,\chi,\overline{\bO}}\otimes \sH_{G,\chi,\overline{\bO'}}\to \sH_{G,\chi}
\]

\noindent has singular support contained in either $\partial \bO$ or $\partial \bO'$.

\end{proof}

\subsubsection{} We may now finish the proof of the classification theorem:

\begin{proof}[Proof of Theorem \ref{t:classification}]

By Proposition \ref{p:thepdoddy}, it suffices to construct an injective map:
\[
H^{\on{top}}(\cB_e)\underset{k[W_{[\lambda]}]}{\otimes} K_0(\sH_{G,\chi}^{\on{ren}})_{\bO,k}\to K_0(\sW_{e,\lambda}\on{-mod}^{\on{fin},\heartsuit})_k.
\]

We have:
\[
H^{\on{top}}(\cB_e)\underset{k[W_{[\lambda]}]}{\otimes} K_0(\sH_{G,\chi}^{\on{ren}})_{\bO,k}\simeq\on{HP}_0(\cW_{e,\lambda}\on{-mod})\underset{k[W_{[\lambda]}]}{\otimes} K_0(\sH_{G,\chi}^{\on{ren}})_{\bO,k}
\]
\[
\overset{\on{Rem.} \ref{r:JSRIP}}{\simeq} \on{HP}_0(\cW_{e,\lambda}\on{-mod}^{\on{fin,access}})\underset{k[W_{[\lambda]}]}{\otimes} K_0(\sH_{G,\chi}^{\on{ren}})_{\bO,k}
\]
\[
\overset{\on{Lem.} \ref{l:inaccessdies}}{\simeq} \on{HP}_0(\cW_{e,\lambda}\on{-mod}^{\on{fin}})\underset{k[W_{[\lambda]}]}{\otimes} K_0(\sH_{G,\chi}^{\on{ren}})_{\bO,k}\subset\on{HP}_0(\cW_{e,\lambda}\on{-mod}^{\on{fin}})\overset{\on{Lem.}\ref{l:kzerotohp}}{\simeq} K_0(\sW_{e,\lambda}\on{-mod}^{\on{fin}})_k,
\]

\noindent as required.

\end{proof}

\newpage

\bibliographystyle{alpha}
\bibliography{Bip}

\end{document}